\documentclass{amsart}
\usepackage{amsmath, amsfonts} 
\usepackage{amsthm} 
\usepackage{amssymb}    
\usepackage{graphicx} 
\usepackage[top=1in, bottom=1in, left=1in, right=1in]{geometry}
\usepackage{mathtools}
\usepackage{mathrsfs}
\usepackage[colorlinks=true,linkcolor=black,anchorcolor=black,citecolor=black,filecolor=black,menucolor=black,runcolor=black,urlcolor=black]{hyperref}
\hypersetup{colorlinks=true,linkcolor=magenta,citecolor=blue}

\usepackage{enumerate}
\usepackage{longtable}
\usepackage{color,comment}
\usepackage{extarrows}
\usepackage{centernot}
\usepackage{tikz}
\usetikzlibrary{trees,arrows,positioning}
\usepackage{tikz-cd}

\usepackage{fancybox}
\usepackage{shadow}
\usepackage[all]{xy}
\usepackage{url}
\usepackage{mathrsfs}
\usepackage{float}
\usepackage{geometry}                		
\usepackage{graphicx}				
\usepackage{faktor}								
\usepackage{amssymb}
\usepackage{bbm}
\usepackage{hyperref}
\usepackage{tikz}
\usepackage{tikz-cd}

\theoremstyle{plain}
\newtheorem{theorem}{Theorem}[section]
\newtheorem*{theorem*}{Theorem}
\newtheorem{maintheorem}{Theorem}
\newtheorem{lemma}[theorem]{Lemma}
\newtheorem*{claim*}{Claim}
\newtheorem{proposition}[theorem]{Proposition}
\newtheorem{corollary}[theorem]{Corollary}

\theoremstyle{definition}
\newtheorem{definition}[theorem]{Definition}

\newtheorem{remark}[theorem]{Remark}

\numberwithin{equation}{section}
\numberwithin{figure}{section}

\newcommand{\diag}{\mathrm{diag}}

\newcommand{\T}{\mathcal{O}}

\newcommand{\SL}{\mathrm{SL}}

	\newcommand{\matrixconjugatinganisotropictorustodiagonaltorus}{\mathsf{h}}
	\newcommand{\nicetrickelement}{\xi}

	\newcommand{\compactinduction}{c\scalebox{0.9}[0.9]{-} \mathrm{Ind}}
	\newcommand{\mat}[1]{\begin{pmatrix}#1\end{pmatrix}}

\begin{document}
	
	\title[Branching rules for irreducible supercuspidal representations of  unramified $\mathrm{U}(1,1)$]{Branching rules for irreducible supercuspidal representations of  unramified $\mathrm{U}(1,1)$}
	\author{Ekta Tiwari}
	\address{Department of Mathematics and Statistics, University of Ottawa, Ottawa, Canada K1N 6N5}
	\email{etiwa049@uottawa.ca}

	\begin{abstract}
		Let $G$ denote the unramified quasi-split unitary group $\mathbb{U}(1,1)(F)$ over a $p$-adic field $F$ with residual characteristic $p \neq 2$. In this article, we determine the branching rules for all irreducible supercuspidal representations of $G$, that is, we explicitly describe their decomposition upon restriction to a fixed maximal compact subgroup $\mathcal{K}$ in terms of irreducible representations of $\mathcal{K}$. We also present two applications of these decompositions, which verify two new conjectures in the literature for $G$. Together with the results from a previous paper by the author, this article completes the description of the branching rules for all irreducible smooth representations of $G$.
	\end{abstract}
	
\maketitle

\section{Introduction}

The study of branching rules refers to determining the decomposition of an irreducible representation of a group $G$ upon restriction to interesting subgroups. The aim is to uncover hidden internal symmetries within the original representation and to highlight shared features across related families of representations. For real Lie groups, this problem has been extensively studied where one seeks to determine the branching rules for irreducible representations of a group upon restriction to its maximal compact subgroup. In the \( p \)-adic setting, however, there is no unique conjugacy class of maximal compact subgroups. Nevertheless, these subgroups are rich enough in structure to make the study of the decomposition of irreducible representations upon restriction to them both natural and meaningful.

Our goal is to provide the branching rules for all irreducible smooth representations of $G$ upon restriction to a fixed maximal compact subgroup $\mathcal{K}$. In a previous article, we determined the branching rules for all the principal series representations. In the present article, we establish those for all the irreducible supercuspidal representations. As these representations together generate the category of smooth irreducible representations of $G$, our results completely solve the branching problem for $G$.

Our main tool is Mackey theory, which provides a decomposition into infinitely many components that we refer to as \emph{Mackey components}. The main challenge is to prove that these Mackey components are indeed irreducible. The strategy to achieve this differs between depth-zero and positive-depth supercuspidal representations. In the depth-zero case, irreducibility is established using character theory, whereas in the positive-depth case it follows from a combination of Frobenius reciprocity, Mackey theory, and several structural properties of the groups and representations involved. Note that these techniques are quite different from those used in the principal series case, where Mackey theory produced a single highly reducible component.

For supercuspidal representations, we establish the irreducibility of the Mackey components by showing that they are isomorphic to certain key representations of $\mathcal{K}$ constructed in~\cite[Theorem~5.6]{ET251}. We briefly recall this construction here. For each integer $d>0$, we select elements $X$ in the Lie algebra $\mathfrak{g}$ of $G$ whose centralizer $T(X)$ in $\mathcal{K}$ is abelian. Each such element $X$ determines a character $\Psi_{X}$ of depth $d$ on a subgroup $\mathcal{J}_{d}\subset\mathcal{K}$. Given a character $\zeta$ of $T(X)$ whose restriction to $T(X)\cap \mathcal{J}_{d}$ agrees with $\Psi_{X}$, we denote by $\Psi_{X,\zeta}$ the unique extension of $\Psi_{X}$ and $\zeta$ to $T(X)\mathcal{J}_{d}\subset \mathcal{K}$. Then
\[
\mathcal{S}_{d}(X,\zeta)\;:=\;\mathrm{Ind}_{T(X)\mathcal{J}_{d}}^{\mathcal{K}}\Psi_{X,\zeta}
\]
is an irreducible representation of $\mathcal{K}$ of depth $d$ and degree $(q^{2}-1)q^{d-1}$.

The principal results of this article, stated precisely in Theorems~\ref{T:depthzero} and~\ref{positive-depth branching}, provide a multiplicity-free decomposition characterized by a unique degree for all irreducible supercuspidal representations of \( G \) upon restriction to \( \mathcal{K} \). The analogous result for all principal series representations of \( G \) was established in~\cite[Theorem~6.3.1 and Corollary~6.3.2]{ET251}.
  Together, these results yield a complete description of the restriction of all irreducible representations of $G$ to $\mathcal{K}$. In terms of the notation introduced above, this may be summarized as follows.

\begin{maintheorem}
	Let $\pi$ be an irreducible smooth representation of $G$ of depth $r$ with central character $\theta$. Then
	\[
	\mathrm{Res}_{\mathcal{K}}\pi \;\cong\; \pi_{r} \;\oplus\; \bigoplus_{d}\,\mathcal{S}_{d}(X,\zeta),
	\]
	where, up to twisting by a character of $G$,
	\begin{itemize}
		\item $\pi_{r}$ appears only for certain supercuspidal representations and for all principal series representations, where $\pi$ is of depth $r$ and is either irreducible or a twist of $\mathbbm{1}\oplus \mathrm{St}$; here $\mathbbm{1}$ and $\mathrm{St}$ denote the inflations to $\mathcal{K}$ of the trivial and Steinberg representations of $\mathbb{U}(1,1)(\mathbb{F}_{q}) \cong \mathcal{K}/\mathcal{K}_{0+}$, respectively;
		\item the sum is taken over integers $d \geq r+\tfrac{1}{2}$, with either all odd or all even such $d$ appearing when $\pi$ is a supercuspidal representation of integral depth (the parity depending on $\pi$), and all $d \geq r+\tfrac{1}{2}$ appearing otherwise;
		\item when $\pi$ has depth zero, $X$ is a nilpotent element and $\zeta=\theta$, while when $\pi$ has positive depth, $X$ and $\zeta$ are twists of the data used to construct $\pi$.
	\end{itemize}
\end{maintheorem}

As noted in the first article, such explicit descriptions have been obtained only for a few low rank cases, underscoring the significance of the present work. This article thus represents a first step toward resolving the problem for higher rank unitary groups. Moreover, having an explicit decomposition enables us to identify new instances of existing conjectures in the literature, two of which we verify here as applications of our results.

As a first application, we show that one can fix four depth-zero irreducible supercuspidal representations 
$\{\tau_{00}, \tau_{01}, \tau_{10}, \tau_{11}\}$ such that the higher depth components of any irreducible supercuspidal representation or any principal series representation, upon restriction to $\mathcal{K}$, coincide with the higher depth components, upon restriction to $\mathcal{K}$, of a subset of these four fixed representations, up to twisting by a character of $G$. 
This answers~\cite[Question~1.2]{guy2024representationsglndnearidentity} for all smooth irreducible representations of $G$.

\begin{maintheorem}[Corollary~\ref{Question1.2supercuspidalseries}]\label{Question1.2}
	Given any irreducible smooth representation $\pi$ of $G$, the higher depth components occurring in the decomposition of $\pi$ upon restriction to $\mathcal{K}$ are the same as the higher depth components occurring in the decomposition upon restriction to $\mathcal{K}$ of a subset of $\{\tau_{00}, \tau_{01}, \tau_{10}, \tau_{11}\}$, up to twisting by a  character of $G$, where $\tau_{ij}$ are some fixed depth-zero irreducible supercuspidal representations of $G$.
\end{maintheorem}

For our second application, we recall from the first paper that our unitary group $G$ admits exactly two non-zero nilpotent $G$-orbits in its Lie algebra $\mathfrak{g}$, which we denote by $\mathcal{N}_{1}$ and $\mathcal{N}_{\varpi}$. To each of these orbits, we associate a highly reducible representation
\[
\begin{aligned}
	\tau_{\mathcal{N}_{1}}(\theta) &= \bigoplus_{d \in 2\mathbb{Z}_{> 0}} \mathcal{S}_{d}(X_{d}, \theta), \qquad
	\tau_{\mathcal{N}_{\varpi}}(\theta) = \bigoplus_{d \in 2\mathbb{Z}_{\geq 0} + 1 } \mathcal{S}_{d}(X_{d}, \theta),
\end{aligned}
\]
where $\theta$ is a character of the center of $G$ and $X_{d}$ is a nilpotent element in the Lie algebra of $G$. Then, in the Grothendieck group of representations, we obtain a decomposition analogous to~\cite[Theorem~7.4]{ET251}. Combining the two results, we derive the following theorem for all irreducible smooth representations of $G$, thereby proving a representation-theoretic analogue of the local character expansion, that is, verifying ~\cite[Theorem~1.1]{Mon2024} for $G$.

\begin{maintheorem}[Theorem~\ref{restrictiontoK2r}]
	Let $\pi$ be an irreducible supercuspidal representation of $G$ of depth $r$, and let $\theta$ denote its central character. Then there exists an integer $n(\pi)$ such that
	\[
	\mathrm{Res}_{\mathcal{K}_{2r+}}\pi
	\;=\;
	n(\pi)\, \mathbbm{1}
	\;\oplus\;
	a_{\mathcal{N}_{1}}\,\mathrm{Res}_{\mathcal{K}_{2r+}}\bigl( \tau_{\mathcal{N}_{1}}(\theta) \bigr)
	\;\oplus\;
	a_{\mathcal{N}_{\varpi}}\,\mathrm{Res}_{\mathcal{K}_{2r+}}\bigl( \tau_{\mathcal{N}_{\varpi}}(\theta) \bigr),
	\]
	where $a_{\mathcal{N}_{1}}, a_{\mathcal{N}_{\varpi}} \in \{0,1\}$.
\end{maintheorem}

These remarkable similarities in the branching rules across all classes of irreducible smooth representations of $G$ naturally lead to the question of whether a similar phenomenon persists for higher rank unitary groups, thereby opening several directions for future research.

This article is organized as follows. In~\S\ref{notation and background}, we present the necessary background, and in~\S\ref{irreduciblerepresentationsofK} we recall the construction of irreducible representations of $\mathcal{K}$ from the first article. In~\S\ref{chapter 4}, we determine the branching rules for the depth-zero supercuspidal representations. We first provide an explicit description of all such representations using the constructions due to Moy and Prasad~\cite{MoyPrasad1994, MoyPrasad1996} and, independently, due to Morris~\cite{Morris1999}. Using Mackey theory, we then describe their decomposition upon restriction to $\mathcal{K}$. We show that each component appearing in this decomposition, referred to as its Mackey components, is irreducible and has distinct depth and degree (Theorem~\ref{T:depthzero}); consequently, the decomposition is multiplicity free. In~\S\ref{positive-depth}, we establish the branching rules for positive-depth irreducible supercuspidal representations of $G$. As in the depth-zero case, we first provide an explicit description of all such representations (Theorem~\ref{positivedepth}), this time using the Adler--Fintzen--Yu method~\cite{Adl98, JKKYu2001, Fintzen2021}. We then restrict these representations to $\mathcal{K}$, and Mackey theory again yields several components. We prove that each of these Mackey components is irreducible and has distinct depth and degree, and thus the decomposition is again multiplicity free (Theorems~\ref{thm4}, \ref{positive-depth branching}). Finally, in~\S\ref{section5.3}, we present two important applications.

\subsection*{Acknowledgements}
I am extremely grateful to Monica Nevins, my PhD supervisor, for suggesting this problem and for patiently answering my countless questions on branching rules for $\mathrm{SL}(2,F)$. This paper could not have been written without her generous guidance and invaluable feedback on earlier versions of this work. I would also like to thank the University of Ottawa for its support and for providing a stimulating research environment.

\section{Notations and background}\label{notation and background}
\subsection{The field}\label{the field}	Let $F$ be a non-archimedean local field, and let $E = F[\sqrt{\epsilon}]$ be the quadratic unramified extension of $F$, where $\epsilon$ is a non-square element of $\mathcal{O}_{F}^{\times}$.  We fix $\varpi$ a uniformizer of $F$ and normalize the valuation $\nu$ on $F$ such that $\nu(\varpi) = 1$. The valuation $\nu$ is extended to $E$, and we use the same symbol $\nu$ for this extension. We denote the ring of integers of $F$ by $\mathcal{O}_{F}$ and that of $E$ by $\mathcal{O}_{E}$.  
Their respective maximal ideals are denoted by $\mathfrak{p}_{F}$ and $\mathfrak{p}_{E}$.  
The residue fields of $F$ and $E$ are given by $\mathfrak{f} \cong \mathcal{O}_{F}/\mathfrak{p}_{F}
\; \text{and} \;
\mathfrak{e} \cong \mathcal{O}_{E}/\mathfrak{p}_{E},$\; respectively. The cardinality of $\mathfrak{f}$ is $q$ where $q=p^{n}$ for some $n\in\mathbb{N}$, and that of $\mathfrak{e}$ is $q^{2}$. Let $\sigma$ denote the nontrivial element of $\mathrm{Gal}(E/F)$. For $x \in E$, we write $\overline{x} := \sigma(x)$. The \emph{norm} and \emph{trace} from $E$ to $F$ are defined by
\[
\mathrm{N}_{E/F}(x) = x \, \overline{x}, \qquad
\mathrm{Tr}_{E/F}(x) = x + \overline{x}, \quad \text{for\;all } x \in E.
\]
We denote by \( E^1 \subset E^\times \) the subgroup consisting of elements of norm one, \emph{i.e.}, $E^1 := \{ x \in E^\times \mid \mathrm{N}_{E/F}(x) = 1 \}.$ The norm map induces a surjection
\begin{equation}\label{surjectionofthenormmap1}
	\mathrm{N}_{E/F} \colon \mathcal{O}_E^\times \twoheadrightarrow \mathcal{O}_F^\times, \quad \mathrm{N}_{E/F} \colon 1 + \mathfrak{p}_E^d \twoheadrightarrow 1 + \mathfrak{p}_F^d,\quad \text{for every}\;d\in\mathbb{Z}_{\geq 1}.
\end{equation}	

\begin{proposition}\label{E1capE11+pE}
	For any integer \( d \geq 1 \), one has
	\[
	\left[ E^1 : E^1 \cap (1 + \mathfrak{p}_E^d) \right] = (q + 1) q^{d - 1}.
	\]
\end{proposition}

\begin{proof}
	By~\eqref{surjectionofthenormmap1}, the norm map induces a surjection
$\mathrm{N}_{E/F} \colon \mathcal{O}_E^\times / (1 + \mathfrak{p}_E^d) \twoheadrightarrow \mathcal{O}_F^\times / (1 + \mathfrak{p}_F^d),$  whose kernel is \( E^1 / (E^1 \cap (1 + \mathfrak{p}_E^d)) \). Therefore,
	\[
	\left[ E^1 : E^1 \cap (1 + \mathfrak{p}_E^d) \right] = 
	\frac{[\mathcal{O}_E^\times : 1 + \mathfrak{p}_E^d]}{[\mathcal{O}_F^\times : 1 + \mathfrak{p}_F^d]}.
	\]
	Since $|\mathcal{O}_{E}^{\times}/(1 + \mathfrak{p}_E)|=|\mathfrak{e}|-1=q^2-1 $, and $|(1+\mathfrak{p}_{E})/(1+\mathfrak{p}_{E}^{d})  \mid=\mid\mathfrak{p}_{E}/\mathfrak{p}_{E}^{d}|=q^{2(d-1)}$, we have
	\[
	[\mathcal{O}_E^\times : 1 + \mathfrak{p}_E^d] = (q^2 - 1) q^{2(d - 1)}, \quad
	[\mathcal{O}_F^\times : 1 + \mathfrak{p}_F^d] = (q - 1) q^{d - 1},
	\]
	and hence
	\[
	\left[ E^1 : E^1 \cap (1 + \mathfrak{p}_E^d) \right] = \frac{(q^2 - 1) q^{2(d - 1)}}{(q - 1) q^{d - 1}} = (q + 1) q^{d - 1}.
	\]
\end{proof}

\subsection{Some structure theory}\label{the group} We adopt the convention of using blackboard bold font to denote an algebraic group \( \mathbb{G} \) defined over \( F \), and the corresponding roman letter for its group of \( F \)-rational points, that is, \( G = \mathbb{G}(F) \). For simplicity of notation, we may refer to a group \( T \) as a ``maximal torus of \( G \)" when we mean that \( T \) is the group of \( F \)-rational points of an algebraic torus \( \mathbb{T} \subset \mathbb{G} \) that is defined over $F$.

Given subsets \( L_i \) of \( F \) or \( E \), we use the shorthand notation
\[
\left\{\begin{pmatrix}L_1 & L_2 \\ L_3 & L_4 \end{pmatrix}\right\}
:= 
\left\{
\begin{pmatrix}
	a_1 & a_2 \\
	a_3 & a_4
\end{pmatrix}
\;\middle|\; a_i \in L_i
\right\}.
\]
For \( g \in G \) and a subgroup \( K \subset G \), we write \( K^g := gKg^{-1} \) and  \( {}^gK := g^{-1}Kg \).

\subsubsection{The group $\mathrm{U}(1,1)$}\label{the group U(1,1)} Let $\mathbb{G}$ denote the quasi-split unitary group $\mathbb{U}(1,1)$ whose group of $F$-points is  
\begin{align*}
	G:=\mathbb{G}(F) = \left\{g \in M_{2}(E) \mid \overline{g}^{\top}\mathrm{w}g = \mathrm{w}\right\},
\end{align*}
where $\mathrm{w}=\begin{pmatrix}0& 1\\ 1&0\end{pmatrix}$, and  $\overline{g}^{\top}=(\overline{g_{ij}})^{\top}=(\overline{g_{ji}})$.
More explicitly, $G$ has the form
\begin{align}\label{definitionofG}
	G= & \left\{
	\begin{pmatrix}
		a & b \\
		c & d
	\end{pmatrix}
	\in M_{2}(E)
	\; \middle| \;
	\begin{aligned}
		& \overline{a}d + \overline{c}b = 1, \\
		& \overline{a}c, \overline{b}d \in \sqrt{\epsilon}F
	\end{aligned}
	\right\}.
\end{align}	
The Lie algebra $\mathfrak{g}$ of $G$ is defined as $\{X\in \mathfrak{gl}(2,E)\;|\;\overline{X}^{\top}\mathrm{w}+\mathrm{w}X=0\}$ and thus has the form
$$
\mathfrak{g} := \mathfrak{u}(1,1)(F) = \left\{
\begin{pmatrix}
	u+\sqrt{\epsilon}v & y\sqrt{\epsilon} \\
	z\sqrt{\epsilon} & -u+\sqrt{\epsilon}v
\end{pmatrix}
\;\middle|\; u, v, y, z \in F
\right\}.
$$ 

The center of $G$ is the subgroup $Z$ of scalar matrices, which we identify with $E^{1}$. Let $\mathbb{G}_{\mathrm{der}}=\mathbb{SU}(1,1)$ denote the derived subgroup of $\mathbb{G}$, and we denote its group of $F$-points by $G_{\mathrm{der}}$.
		\begin{lemma}\label{nicetrick}
			The quotient group  $G/ZG_{\mathrm{der}}$ has order two and the nontrivial coset of $ZG_{\mathrm{der}}$ in $G$ is represented by the matrix
			$$
			\nicetrickelement:=\mat{\varepsilon_E&0\\0&\overline{\varepsilon_E}^{-1}},
			$$
			where $\varepsilon_E\in \mathcal{O}_E^\times$ represents the nontrivial class of $\mathcal{O}_E^\times/(\mathcal{O}_E^\times)^2$.
		\end{lemma}
	
		\begin{proof}
			If $g\in G$, then $\overline{g}^T\mathrm{w}g=\mathrm{w}$,  it follows that $\det(g)\in E^1$.  The determinant of any element of $ZG_{\mathrm{der}}$ lies in $(E^1)^2$. By Hilbert's Theorem 90~\cite[V.11.6 Theorem 3]{BourbakiAlgebraII}, the map $x \mapsto x\overline{x}^{-1}$ induces an isomorphism  $\mathcal{O}_E^\times/\mathcal{O}_F^\times \to E^1$, so in particular $(E^1)^2$ is the image of the subgroup $(\mathcal{O}_E^\times)^2\mathcal{O}_F^\times/\mathcal{O}_F^\times$.  Write $\mathcal{O}_E^\times/(\mathcal{O}_E^\times)^2 =\{1,\varepsilon_E\}$, where $\varepsilon_E$ is a fixed nonsquare of $\mathcal{O}_E^\times$; then $\varepsilon_E\notin \mathcal{O}_F^\times$ since the extension is unramified (implying that every element of $\mathcal{O}_F^\times$ is a square of an element of $\mathcal{O}_E^\times$).  We conclude that $E^1/(E^1)^2$ has order $2$, and that the nontrivial coset is represented by $\gamma =\varepsilon_E \overline{\varepsilon_E}^{-1} \in E^1$.
			Now $\nicetrickelement$	lies in $G$ and has determinant $\gamma \in E^1 \setminus (E^1)^2$.  It follows that
			$$
			G = ZG_{\mathrm{der}} \sqcup \nicetrickelement ZG_{\mathrm{der}} .
			$$
		\end{proof}	
		The group \(G\) has two conjugacy classes of maximal compact subgroups, represented by 
		$$\mathcal{K}
		=\mathbb{G}(\mathcal{O}_{E})
		=\begin{pmatrix}
			\mathcal{O}_{E} & \mathcal{O}_{E} \\
			\mathcal{O}_{E} & \mathcal{O}_{E}
		\end{pmatrix}\cap G\;\hspace{2em}\text{and}\;\hspace{2em}\mathcal{K}^{\eta}
		=\begin{pmatrix}
			\mathcal{O}_{E} & \mathfrak{p}_{E}^{-1} \\
			\mathfrak{p}_{E} & \mathcal{O}_{E}
		\end{pmatrix}\cap G,$$
		where 
		\(\eta=\begin{psmallmatrix}1&0\\ 0&\varpi\end{psmallmatrix}\in\mathrm{GL}(2,E)\) is such that it normalizes our unitary group \(G\). The next lemma was established in~\cite[Lemma~2.1]{ET251}; we include it here for ease of reference.
		\begin{lemma}\label{Kproperty}
			Let $\begin{pmatrix}x&y \\z&w\end{pmatrix}\in \mathcal{K}$. Then $x, w\in\mathcal{O}_{E}^{\times}$ or $y, z\in \mathcal{O}_{E}^{\times}$.
		\end{lemma}	
		We now describe certain double coset representatives, which will be of use in the arguments developed in the subsequent sections.
		\begin{lemma}\label{doublecosetrepresentativeofKGK}
			A set of double coset representatives for 
			\(\mathcal{K}\backslash G/\mathcal{K}\), 
			\(\mathcal{K}\backslash G/\mathcal{K}^{\eta}\), 
			and \(\mathcal{K}^{\eta}\backslash G/\mathcal{K}^{\eta}\) is
			\[
			\Bigl\{
			\alpha^{t}:=\begin{pmatrix}\varpi^{-t}&0\\[2pt]0&\varpi^{t}\end{pmatrix}
			\;\Bigm|\; t\ge 0
			\Bigr\}.
			\]
		\end{lemma}
		
		\begin{proof}
			By the Cartan decomposition we have \(G=\mathcal{K}T\mathcal{K}\), where
			\[
			T=\{\diag(u\varpi^{k},\,\overline{u}^{-1}\varpi^{-k}) \mid u\in\mathcal{O}_E^\times,\; k\in\mathbb{Z}\}.
			\]
			For any such element,
			\[
			\begin{pmatrix}u\varpi^{k}&0\\0&\overline{u}^{-1}\varpi^{-k}\end{pmatrix}
			=
			\begin{pmatrix}u&0\\0&\overline{u}^{-1}\end{pmatrix}
			\begin{pmatrix}\varpi^{k}&0\\0&\varpi^{-k}\end{pmatrix}
			\in \mathcal{K}\,\alpha^{k}\,\mathcal{K},
			\]
			so every double coset \(\mathcal{K}g\mathcal{K}\) meets some \(\alpha^{k}\) with \(k\in\mathbb{Z}\).	Recall that \(\mathrm{w}=\begin{psmallmatrix}0&1\\1&0\end{psmallmatrix}\in\mathcal{K}\).  
			Then \(\mathrm{w}\alpha^{k}\mathrm{w}^{-1}=\alpha^{-k}\), so \(\alpha^{k}\) and \(\alpha^{-k}\) lie in the same coset.  
			Thus we may assume \(k\ge 0\). Using Lemma~\ref{Kproperty}, an entrywise case analysis of $\mathcal{K}\alpha^{t}\mathcal{K}$ shows that the minimal valuation among its entries is $-t$, and this minimum is attained by at least one entry; hence distinct $t$ yield distinct double cosets. Hence \(\{\alpha^{t}\}_{t\ge 0}\) is a complete set of representatives for \(\mathcal{K}\backslash G/\mathcal{K}\). Since $\eta$ commutes with $\alpha^{t}$, conjugating the Cartan decomposition of $G$ with respect to $\mathcal{K} $ by $\eta$ yields 
			\(G=\bigsqcup_{t\ge 0}\mathcal{K}^{\eta}\alpha^{t}\mathcal{K}^{\eta}\). 	A similar argument gives 	\(G=\bigsqcup_{t\ge 0}\mathcal{K}\alpha^{t}\mathcal{K}^{\eta}\).
		\end{proof}
		
		Finally, we conclude this subsection with a lemma that lets us build new double coset representatives from known ones; we will apply it in later sections.
		\begin{lemma}\label{db1}
			Let $G$ be a group and let $K, J$ be subgroups of $G$ such that $J\subseteq K$. If $B$ is a set of double coset representatives of $K\backslash G/ K$, and for all $g\in B$, $A_{g}$ is a set of double coset representatives for $(K\cap K^{g^{-1}})\backslash K/J$, then $\{gh\mid g\in B, h\in A_{g}\}$ is a set of  double coset representatives for $K\backslash G/J$.
		\end{lemma}

		\subsubsection{Moy--Prasad filtrations of $G$}\label{Moy-Prasad filtrations of G}
		In this subsection, we introduce the Moy--Prasad filtrations of parahoric subgroups of $G$ 
		and recall the structure of the building of $G$. Our interest lies in the explicit description 
		of these filtrations for our unitary group $G = \mathrm{U}(1,1)$, so most of the definitions are presented 
		in a specialized form adapted to this case. For the general theory, standard references include 
		\cite{KalTasPra2023, MoyPrasad1994, MoyPrasad1996, Bru1972}. 
		
		Let $\mathbb{G}_{m}$ denote the multiplicative group $\mathrm{GL}_{1}$. A \emph{torus} $\mathbb{T}$ is an algebraic group that is isomorphic to $(\mathbb{G}_{m})^{n}$ for some $n>0$. When it is defined over $F$, we call it an $F$-\emph{torus}. It is \emph{$F$-split} if the isomorphism $\mathbb{T}$ to $\mathbb{G}_{m}^{n}$ is defined over $F$ and at the other extreme it is \emph{anisotropic} if it contains no $F$-split subtori. The \emph{group of characters} $X^{*}(\mathbb{T})$ of $\mathbb{T}$ is the homomorphism space $\mathrm{Hom}(\mathbb{T}, \mathbb{G}_{m})$, and the \emph{group of cocharacters} $X_{*}(\mathbb{T})$ is the homomorphism space $\mathrm{Hom}(\mathbb{G}_{m}, \mathbb{T})$. Moreover, there is a natural bilinear pairing
		\[
		\langle \cdot, \cdot \rangle \colon X^{*}(\mathbb{T}) \times X_{*}(\mathbb{T}) \longrightarrow \mathbb{Z}
		\]
		defined by the condition that for $\chi \in X^{*}(\mathbb{T})$ and $\gamma \in X_{*}(\mathbb{T})$, one has $\chi(\gamma(t)) = t^{m},$ for all $t \in \mathbb{G}_{m},$ where the integer $m$ is identified with the corresponding element of $\mathrm{Hom}(\mathbb{G}_{m}, \mathbb{G}_{m}) \cong \mathbb{Z}$.

		We now turn to our unitary group $G$. Note that every maximal torus $T$ of $G$ contains the center $Z$ of $G$. Let $\mathbb{T}$ be the diagonal torus of $\mathbb{U}(1,1)$, whose group of $F$-points is 			
		\begin{equation}\label{definitionofmaximallyFsplitFtorus}
			T:=\mathbb{T}(F)=\left\{\begin{pmatrix}
				a&0\\ 0&\overline{a}^{-1}
			\end{pmatrix}\mid a\in E^{\times}\right\}.\end{equation} 
		We have $\mathbb{T}(F)\cong E^\times$ is a maximally $F$-split maximal $F$-torus of $G$.  Any other maximal torus of $G$ is a $G$-conjugate of $T$. The maximal $F$-split subtorus of $\mathbb{T}$ is $\mathbb{S}$ where
		$$S:=\mathbb{S}(F)=\left\{\begin{pmatrix}
			a&0\\ 0&a^{-1}
		\end{pmatrix}\mid a\in F^{\times}\right\}.$$
		A direct computation shows that the centralizer of $S$ in $G$, denoted $C_{G}(S)$, coincides with the centralizer of $T$ in $G$ and is equal to $T$.
		
		One of the ingredients in defining a parahoric subgroup is the family of 
		filtration subgroups of $T$ for each real number $r \geq 0$. 
		For real numbers $r\geq0$, the filtration on $T$ is defined as follows~\cite[Definition~7.2.2]{KalTasPra2023}:
		\begin{align}
			T_{0}&=\left\{\,t\in T \;\middle|\; \nu(\chi(t))=0 \ \text{for all}\ \chi\in X^{*}(T)\,\right\},\label{filtration of torus}\\
			T_{r}&=\left\{\,t\in T_{0} \;\middle|\; \nu(\chi(t)-1)\geq r \ \text{for all}\ \chi\in X^{*}(T)\,\right\},\hspace{2em}\mathrm{for}\;r>0.\label{filtration of torus2}
		\end{align}
		In general, these filtration subgroups are rather difficult to compute. 
		However, in our case they admit a nice description, given by 
		\begin{align*}
			T_{0}&=\left\{
			\begin{pmatrix}
				a & 0 \\
				0 & \overline{a}^{-1}
			\end{pmatrix}
			\;\middle|\; a\in \mathcal{O}_{E}^{\times}\right\}\qquad\text{and}\qquad
			T_{r}=\left\{
			\begin{pmatrix}
				a & 0 \\
				0 & \overline{a}^{-1}
			\end{pmatrix}
			\;\middle|\; a\in 1+\mathfrak{p}_{E}^{\lceil r \rceil}\right\}, 
			\qquad \text{for } r>0.
		\end{align*}
		A second ingredient required for defining parahoric subgroups are \emph{affine apartments} and 
		\emph{affine root subgroups}. We begin by describing the root system on $G$ 
		associated with $\mathbb{S}$.  The root system $\Phi:=\Phi(\mathbb{G},\mathbb{S})$ associated with $\mathbb{S}$ consists of 
		$\alpha$ and $-\alpha$, where for $a\in F^{\times}$,  
		\[
		\alpha\!\left(\begin{pmatrix}
			a & 0 \\ 0 & a^{-1}
		\end{pmatrix}\right)=a^{2}, 
		\qquad 
		(-\alpha)\!\left(\begin{pmatrix}
			a & 0 \\ 0 & a^{-1}
		\end{pmatrix}\right)=a^{-2}.
		\]
		The corresponding root subgroups are
		\begin{align*}
			U_{\alpha}:=\mathbb{U}_{\alpha}(F)&=\left\{\begin{pmatrix}
				1 & \sqrt{\epsilon}b \\[2pt] 0 & 1
			\end{pmatrix} \;\middle|\; b\in F\right\}\qquad \text{and}\qquad
			U_{-\alpha}:=\mathbb{U}_{-\alpha}(F)=\left\{\begin{pmatrix}
				1 & 0 \\[2pt] \sqrt{\epsilon}b & 1
			\end{pmatrix} \;\middle|\; b\in F\right\}.
		\end{align*}
		The elements $X_{\alpha}=\begin{psmallmatrix}0&\sqrt{\epsilon}\\0&0\end{psmallmatrix}$ and 
		$X_{-\alpha}=\begin{psmallmatrix}0&0\\ \sqrt{\epsilon}&0\end{psmallmatrix}$ 
		form a Chevalley system. Let $\Phi^{\vee}=\{\alpha^{\vee}\mid\alpha\in \Phi\}$ denote the set of coroots of $\mathbb{T}$.  The \emph{affine apartment} of $G$ associated with $\mathbb{S}$, denoted 
		$\mathcal{A}(\mathbb{G}, \mathbb{S}, F)$, is the real vector space 
		$X_{*}(\mathbb{S})\otimes_{\mathbb{Z}} \mathbb{R}$. Since $S \cong F^{\times}$, it follows that 
		$X_{*}(S) \cong \mathbb{Z}$, and therefore $\mathcal{A}:=\mathcal{A}(\mathbb{G}, \mathbb{S}, F)
		= X_{*}(\mathbb{S}) \otimes_{\mathbb{Z}} \mathbb{R} 
		\;\cong\; \mathbb{R}.$ Thus, in this case, the affine apartment is simply the real line. The pairing $\langle \cdot, \cdot \rangle$ extends linearly to 
		$\mathcal{A}(\mathbb{G}, \mathbb{S}, F)$. We fix a parametrization of 
		$\mathcal{A}\cong \mathbb{R}$ so that $\alpha(y):=\langle \alpha,y\rangle=y$ for all $y\in \mathcal{A}$.  
		
		Given the root system $\Phi=\Phi(\mathbb{G}, \mathbb{S})$, we define the set of \emph{affine roots} by  
		\[
		\Phi^{\mathrm{aff}}=\{\alpha+n \mid \alpha\in \Phi,\ n\in \mathbb{Z}\}.
		\]
		If $\varphi\in \Phi^{\mathrm{aff}}$, then $\varphi=\alpha+n$ for some $\alpha\in \Phi$ and $n\in\mathbb{Z}$. 
		Each $\varphi=\alpha+n$ defines an affine function on $\mathcal{A}(\mathbb{G}, \mathbb{S}, F)$ via  
		\[
		\varphi(y):=\langle \varphi, y\rangle=\langle \alpha, y \rangle+n=\alpha(x)+n, 
		\qquad y\in \mathcal{A}(\mathbb{G}, \mathbb{S}, F).
		\]
		Note that $\langle \alpha, \alpha^{\vee}\rangle=2$.  For $\varphi = \alpha+n \in \Phi^{\mathrm{aff}}$, the associated 
		\emph{affine root subgroup} is defined by  $U_{\varphi} := \mathbb{U}_{\varphi}(F) = \mathbb{U}_{\alpha}(\mathfrak{p}_{F}^{n}).$ Finally, for $y\in \mathcal{A}$, the corresponding \emph{parahoric subgroup} of $G$ is
		\[
		G_{y,0}=\langle\,T_0, U_{\varphi}\;\mid\; \varphi\in \Phi^{\mathrm{aff}},\ \varphi(y)\geq 0\,\rangle.
		\]
		If $\varphi=\alpha+n$, then for $y\in \mathcal{A}$ we have $\varphi(y)\geq 0$ if and only if $\alpha(y)+n\geq 0$ if and only if $n\geq -\alpha(y)$. In particular, for $y\in\mathbb{R}$,
		\begin{align*}
			G_{y,0}=\begin{psmallmatrix}
				\T_{E}&\mathfrak{p}_{E}^{\lceil -y\rceil}\\
				\mathfrak{p}_{E}^{\lceil y\rceil}&\T_{E}
			\end{psmallmatrix}\cap G.
		\end{align*}
		For $r>0$, the parahoric subgroups admit a further filtration, called the 
		\emph{Moy--Prasad filtration}, given by  
		\[
		G_{y,r}=\langle\, T_r, U_{\varphi}\;\mid\; \varphi\in\Phi^{\mathrm{aff}},\ \varphi(y)\geq r \,\rangle.
		\]
		
		Explicitly, we have  
		\begin{equation}\label{moyprasadfiltration}
			G_{y,r}=\begin{psmallmatrix}
				1+\mathfrak{p}_{E}^{\lceil r \rceil}&\mathfrak{p}_{E}^{\lceil r-y\rceil}\\
				\mathfrak{p}_{E}^{\lceil r+y\rceil}& 1+\mathfrak{p}_{E}^{\lceil r \rceil}
			\end{psmallmatrix}\cap G, \qquad\text{	and we set}\qquad 	G_{y,r+}=\bigcup_{s>r} G_{y,s}.
		\end{equation}
		For all $y\in \mathcal{A}$ and $r\geq 0$, the subgroup $G_{y,r}$ is normal in $G_{y,0}$. The quotient $G_{x,0}/G_{x,0+}$ is the group of $\mathbb{F}_{q}$-points of a connected reductive group $\mathbf{G}_{x}$ defined over the residue field $\mathbb{F}_{q}$ of $F$. For $r\in \mathbb{R}_{>0}$, the quotient $G_{x,r}/G_{x,r+}$ is abelian and can be naturally identified with an $\mathbb{F}_{q}$-vector space~\cite{MoyPrasad1994, MoyPrasad1996}. For example, if we take $y = 0 \in \mathcal{A}$, then  
		$G_{0,0}$ is the maximal compact subgroup $\mathcal{K}$ of $G$. The other maximal compact subgroup, denoted $\mathcal{K}^{\eta}$, corresponds to the parahoric $G_{1,0}$ associated with $y = 1 \in \mathcal{A}$.
		
		
		For $y\in\mathcal{A}$ and $r\in\mathbb{R}$, we can similarly define a filtration $\mathfrak{g}_{y, r}$ of the Lie algebra  $\mathfrak{g}$ and a filtration $\mathfrak{g}^{*}_{y, r}$ of the $F$-linear dual of the Lie algebra $\mathfrak{g}$ as follows. Let $\mathfrak{t}$ denote the Lie algebra of the torus $T$. Then we set
		$$\mathfrak{t}_{r}=\{X\in \mathfrak{t}\mid  \nu(d_{\chi}(X))\geq r\;\; \mathrm{for\;all}\;\chi\in X^{*}(T)\},$$
		where $d_{\chi}$ denotes the derivative of $\chi$,
		$\mathfrak{g}_{\alpha, y,r}=\varpi^{\lceil r-y \rceil}\mathcal{O}_{F}X_{\alpha},$ for $\alpha\in \Phi$, and
		$$\mathfrak{g}_{y, r}=\mathfrak{t}_{r}\oplus \bigoplus_{\alpha\in \Phi }\mathfrak{g}_{\alpha, y, r}.$$
		Thus we have
		\begin{equation}
			\mathfrak{g}_{y, r}=\begin{pmatrix}
				\mathfrak{p}_{E}^{\lceil r \rceil}&\sqrt{\epsilon}\mathfrak{p}_{F}^{\lceil r-y\rceil}\\
				\sqrt{\epsilon}\mathfrak{p}_{F}^{\lceil r+ y\rceil}& \mathfrak{p}_{E}^{\lceil r \rceil}
			\end{pmatrix}\cap \mathfrak{g}.
		\end{equation}			
		
		We define the filtration subspace $\mathfrak{g}^{*}_{y, -r}$ of the dual of the Lie algebra by
		\begin{equation}
			\mathfrak{g}^{*}_{y,-r}=\{\lambda\in\mathfrak{g}^{*}\mid \lambda(Y)\in\mathfrak{p}_{E}, \;\forall Y\in\mathfrak{g}_{y,s},\; s>r\}.\label{filtrationofdual}
		\end{equation}		
		We can identify $\mathfrak{g}^{*}$ with $\mathfrak{g}$ by using the trace form.
		In particular, for $X\in\mathfrak{g}$, we define $\lambda_{X}\in\mathfrak{g}^*$ by the equation
		$\lambda_{X}(Y)=\mathrm{Tr}(XY)$ for all $Y\in\mathfrak{g}$. Under this identification, $\mathfrak{g}^{*}_{y,r}=\mathfrak{g}_{y,r}$ for all $y\in \mathcal{A}$ and all $r\in\mathbb{R}_{\geq 0}$.

		For \( r,s \in \mathbb{R}_{> 0} \) with 
		\( 2r \geq s > r \), we have the  \emph{Moy--Prasad isomorphism} $G_{x,r}/G_{x,s} \;\cong\; \mathfrak{g}_{x,r}/\mathfrak{g}_{x,s}.$ For our unitary group $\mathrm{U}(1,1)$, it is the isomorphism that sends a coset represented by $g$ to a coset represented by $g-I$. We are now ready to introduce the building of $G$. The Weyl group $W$ of $G$ is 
		\[
		W := N_{G}(S)/ C_{G}(S) \;\cong\; \mathbb{Z}/2\mathbb{Z} \;=\; \{ I, \mathrm{w} \}.
		\]
		The \emph{affine Weyl group}, denoted by $W^{\mathrm{aff}}$, is defined as $W^{\mathrm{aff}} := N_{G}(T)/T_{0} \;\cong\; T/T_{0} \ltimes W.$ The group $T/T_{0}$ acts on the apartment $\mathcal{A}$ by translations of the form $t\cdot y \;=\; y - \nu(t),$ where $\nu(t)$ denotes the valuation of the upper-left entry of the matrix $t \in T$, 
		while the nontrivial element $\mathrm{w}$ acts by reflection along the line $y=0$ 
		\cite[Chapter~3(d)]{KalTasPra2023}. With this action in place, the \emph{enlarged Bruhat--Tits building} 
		$\mathscr{B}^{\mathrm{en}}(\mathbb{G},F)$ is defined as the polysimplicial complex $\mathscr{B}^{\mathrm{en}}(\mathbb{G},F) 
		= \bigl(G \times \mathcal{A}(\mathbb{G}, \mathbb{S}, F)\bigr) \big/ \sim,$ where  $(g, x) \sim (h, y)$ if and only if there exists $ n \in N_{G}(T)$ such that $y = n \cdot x 
		\; \text{ and }\; g^{-1} h n \in G_{x,0}.$
		
		Let $\mathscr{B}(\mathbb{G}, F) := \mathscr{B}^{\mathrm{en}}(\mathbb{G}_{\mathrm{der}},F)$ denote the reduced Bruhat--Tits building of $\mathbb{G}$ over $F$. For our unitary group $G$, since the center $Z$ of $G$ is compact, we have $\mathscr{B}^{\mathrm{en}}(\mathbb{G},F) \;=\; 
		\mathscr{B}(\mathbb{G}_{\mathrm{der}}, F) \;=\; \mathscr{B}(G).$
		For every $y \in \mathscr{B}(G)$, there exists $g \in G$ such that $g \cdot y \in \mathcal{A}$. 
		The parahoric subgroup attached to $y$ is then defined by $G_{y,0} := {}^{g^{-1}} G_{g \cdot y,0},$ and for $r > 0$, the corresponding Moy--Prasad filtration subgroups are  $G_{y,r} := {}^{g^{-1}} G_{g \cdot y,r}.$
		

		\subsubsection{Conjugacy classes of anisotropic tori of $G$} As mentioned in the introduction, to describe the branching rules for positive-depth supercuspidal representations we first provide an explicit description of these representations using the Adler--Fintzen--Yu method. One of the key ingredients required in this construction is a sequence of twisted Levi subgroups. In our case the first element of such a sequence is an anisotropic torus of $G$.  We describe these sequences in greater detail in ~\S\ref{positive-depth}, but for the present subsection our goal is to classify all anisotropic tori of $G$ up to conjugacy and provide filtrations on these tori.
		
		Each maximal anisotropic torus of \(G = \mathrm{U}(1,1)\) splits over some extension of \(F\). If it splits over an unramified extension, we call the torus \emph{unramified}, and if it splits over a ramified extension, we call the torus \emph{ramified}.  
		
		\begin{proposition}\label{conjugacyclassesofanisotropictoriofG}
			The group $G$ has four conjugacy classes of maximal anisotropic tori. They are given by  
			\[
			\mathcal{T}_{\gamma_1, \gamma_2}
			= \left\{
			\begin{pmatrix}
				a & b\gamma_1 \\
				b\gamma_2 & a
			\end{pmatrix}
			\;\middle|\;
			a\overline{a} + b\overline{b}\,\gamma_1\gamma_2 = 1,\;
			\overline{a}b \in \sqrt{\epsilon}F
			\right\},
			\]
			where $(\gamma_1, \gamma_2) \in 
			\bigl\{ (1,1),\, (\varpi^{-1}, \varpi),\, (1,\varpi),\, (1,\epsilon^{-1}\varpi) \bigr\}.$ The pairs $(1,1)$ and $(\varpi^{-1}, \varpi)$ correspond to unramified tori, while $(1,\varpi)$ and $(1,\epsilon^{-1}\varpi)$ correspond to ramified tori.
		\end{proposition}
	\begin{proof}  Since the center of $G$ is compact, every anisotropic torus of $G$ arises as the centralizer of an anisotropic torus of $G_{\mathrm{der}}$. In our case,  there is a simple isomorphism $\Xi:\mathrm{SL}(2,F)\to G_{\mathrm{der}}$ given by
			\begin{equation}\label{isoGder}
				\Xi(g) = 
				\begin{pmatrix}
					\sqrt{\epsilon} & 0 \\
					0 & 1
				\end{pmatrix}
				g
				\begin{pmatrix}
					\sqrt{\epsilon}^{-1} & 0 \\
					0 & 1
				\end{pmatrix}
				\quad \text{for } g \in \mathrm{SL}(2,F).
			\end{equation}
			The $\mathrm{SL}(2,F)$-conjugacy classes of maximal anisotropic tori  of $\mathrm{SL}(2,F)$ are well-known.  By \cite[\S2.3]{MN13}, each such class is represented by a matrix of the form
			$$
			T^{\mathrm{SL}(2,F)}_{\delta_1,\delta_2}= \left\{ \begin{pmatrix}a&b\delta_1\\b\delta_2&a\end{pmatrix}\;\middle|\; a,b\in F, a^2-b^2\delta_1\delta_2=1\right\}
			$$
			where $(\delta_1,\delta_2)\in \{(1, \epsilon), (\varpi^{-1}, \epsilon\varpi), (1, \varpi),
			(\epsilon, \epsilon^{-1}\varpi), (1, \epsilon\varpi), (\epsilon, \varpi)\}$, understanding that if $-1$ is not a square, then $T^{\SL(2,F)}_{1,\varpi}\simeq T^{\mathrm{SL}(2,F)}_{\epsilon,\epsilon^{-1}\varpi}$ and $T^{\mathrm{SL}(2,F)}_{1,\epsilon\varpi}\simeq T^{\mathrm{SL}(2,F)}_{\epsilon,\varpi}$.
			Moreover, $T^{\mathrm{SL}(2,F)}_{1, \varpi}\ncong T^{\mathrm{SL}(2,F)}_{1, \epsilon\varpi}$ as they split over different field extensions. Here, $\simeq$ denotes tori representing the same conjugacy class, whereas $\cong$ denotes isomorphic tori.
			
			The centralizer in $\mathrm{U}(1,1)$ of $\Xi(T^{\mathrm{SL}(2,F)}_{\delta_1,\delta_2})$ is 
			$$
			\mathcal{T}_{\gamma_1, \gamma_2}=C_{U(1,1)}\left( \begin{pmatrix}0&\sqrt{\epsilon} \delta_1\\ \sqrt{\epsilon}^{-1}\delta_2&0\end{pmatrix}\right) = \left\{\begin{pmatrix}t&r\delta_1\\ r\epsilon^{-1}\delta_2&t\end{pmatrix}\;\middle\vert\;t,r\in E\right\}\cap G,$$
			where $\gamma_1=\delta_1$ and $\gamma_2=\epsilon^{-1}\delta_2$. Thus we get six anisotropic tori $\mathcal{T}_{\gamma_1, \gamma_2}$ where $$(\gamma_1, \gamma_2)=\{(1,1), (\varpi^{-1}, \varpi), (1, \epsilon^{-1}\varpi), (\epsilon, \epsilon^{-2}\varpi), (1, \varpi), (\epsilon, \epsilon^{-1}\varpi)\}.$$
			Let $g:=\begin{pmatrix} \sqrt{\epsilon}&0\\ 0&-\sqrt{\epsilon}^{-1}\end{pmatrix}\in G$; then we have  $\mathcal{T}_{1, \varpi}^{g}=\mathcal{T}_{\epsilon, \epsilon^{-1}\varpi}$, and  $\mathcal{T}_{1, \epsilon^{-1} \varpi}^{g}=\mathcal{T}_{\epsilon,\epsilon^{-2} \varpi}$. Since $\mathcal{T}_{1, \epsilon^{-1} \varpi}$ and $\mathcal{T}_{1, \varpi}$  correspond to the non-ismorphic tori  $T^{\mathrm{SL}(2,F)}_{1, \varpi}$ and $T^{\mathrm{SL}(2,F)}_{1, \epsilon\varpi}$ respectively, they are not $G$-conjugates.

			Now let us verify the nonconjugacy of the two tori
			\begin{equation*}
				\mathcal{T}_{1,1 } = \left\{\begin{pmatrix}a & b\\ b& a \end{pmatrix}\;\middle|\; a\overline{a}+b\overline{b}=1\right\}\;\mathrm{and} \quad 
				\mathcal{T}_{\varpi^{-1}, \varpi} = \left\{\begin{pmatrix}a & b\varpi^{-1}\\ b\varpi& a\end{pmatrix}\;\middle|\; a\overline{a}+b\overline{b}=1\right\}.
			\end{equation*}
		If there were an element $zh\in Z\mathrm{SU}(1,1)$ conjugating $\mathcal{T}_{1,1 }$ to $\mathcal{T}_{\varpi^{-1}, \varpi}$, then since $\mathrm{SU}(1,1)$ is normal, $h$ would conjugate $\mathcal{T}_{1,1 }\cap \mathrm{SU}(1,1)$ to $\mathcal{T}_{\varpi^{-1}, \varpi}\cap \mathrm{SU}(1,1)$, a contradiction.  
			Let $\nicetrickelement=\mat{\epsilon_E&0\\0&\overline{\epsilon_E}^{-1}}$ be the non-trivial coset representative of $\mathrm{U}(1,1)/Z\mathrm{SU}(1,1)$ as in Lemma~\ref{nicetrick}.  To prove that $\mathcal{T}_{1,1 }$ and $\mathcal{T}_{\varpi^{-1}, \varpi}$ are not $\mathrm{U}(1,1)$-conjugate, it suffices to prove that $\nicetrickelement\mathcal{T}_{1,1 }\nicetrickelement^{-1}$ is conjugate to $\mathcal{T}_{1,1 }$ via an element of $\mathrm{SU}(1,1)$.

			Note that since $\epsilon_E\overline{\epsilon_E} \in \mathcal{O}_F^\times$, Hensel's Lemma assures us that $\mathcal{O}_E^\times$ contains a square root of $\epsilon_E\overline{\epsilon_E}$. Set
			
		$$	h= \mat{0 & -\sqrt{\epsilon_E\overline{\epsilon_E}} \\\sqrt{\epsilon_E\overline{\epsilon_E}}^{-1} & 0}.$$
		
			We readily verify that $h\in \mathrm{SU}(1,1)$.  Then
		\begin{align*}
				\nicetrickelement\mathcal{T}_{1,1 }\nicetrickelement^{-1}&= \left\{\mat{a & b\epsilon_E\overline{\epsilon_E}\\ b(\epsilon_E\overline{\epsilon_E})^{-1}& a}\;\middle\vert\; a\overline{a}+b\overline{b}=1\right\}
				= \left\{\mat{a & -b\epsilon_E\overline{\epsilon_E}\\ -b(\epsilon_E\overline{\epsilon_E})^{-1}& a}\;\middle\vert\; a\overline{a}+b\overline{b}=1\right\}
				= h\mathcal{T}_{1,1 }h^{-1}.
			\end{align*}
Therefore, the conjugacy class of $\mathcal{T}_{1,1 }$ under $\mathrm{U}(1,1)$ is equal to its conjugacy class under $\mathrm{SU}(1,1)$, whence $\mathcal{T}_{1,1 } \not\simeq \mathcal{T}_{\varpi^{-1}, \varpi}$.
		\end{proof}

		Since we have already introduced the definition of the building, it is useful to recall an important fact. The building of an anisotropic torus consists of a single point, and this point embeds naturally into the building of \(G\), as stated in the lemma below.
		\begin{lemma}\label{lemma:pointofanisotropictorus}
			Let $\mathbb{T}$ be an anisotropic torus of $\mathbb{G}$ defined over $F$, with splitting field $F'$. Then the corresponding point in $\mathcal{B}(\mathbb{G},F)$ is the unique class $[g,x] \in \mathcal{B}(\mathbb{G},F')$ such that $x \in \mathcal{A}$ and
			\[
			\eta(g^{-1})g \in \mathbb{G}(F')_{x, 0} \qquad \text{for all } \eta \in \mathrm{Gal}(F'/F),
			\]
			where $\mathbb{T} = \mathbb{S}^g$ for some $g \in \mathbb{G}(F')$. 
		\end{lemma}
		
		\begin{proof}
			The argument is given in \cite[\S~2.3]{MN13}; but we recall it here for the sake of completeness.  
			Since $p \neq 2$, the splitting field $F'$ of $\mathbb{T}$ is tamely ramified. By \cite[2.6.1]{Bru1979}, we may realize the Bruhat--Tits building $\mathcal{B}(\mathbb{G},F)$ as the subset of $\mathcal{B}(\mathbb{G},F')$ fixed by $\mathrm{Gal}(F'/F)$, and
			$\mathcal{A}(\mathbb{G},\mathbb{T},F) \;=\; \mathcal{A}(\mathbb{G},\mathbb{T},F') \cap \mathcal{B}(\mathbb{G},F).$
			
			Since $\mathbb{T}(F)$ is totally anisotropic, this intersection consists of a single point. Moreover, $\mathcal{B}(\mathbb{G},F')$ admits the description $\mathcal{B}(\mathbb{G},F') \;\cong\; (\mathbb{G}(F') \times \mathcal{A})/\!\sim,$ where $(g,x) \sim (h,y)$ if there exists $n \in N = N_{\mathbb{G}(F')}(C_{\mathbb{G}(F')}(\mathbb{S}(F')))$ such that $n \cdot x = y$ and $g^{-1} h n \in \mathbb{G}(F')_{x,0}$ \cite[7.4.1]{Bru1972}. Since $\mathbb{T}$ splits over $F'$, there exists $g \in \mathbb{G}(F')$ such that $\mathbb{T} = \mathbb{S}^g$, and hence $\mathcal{A}(\mathbb{G},\mathbb{T},F') \;=\; g \cdot \mathcal{A}.$ This is identified with the image of $\{g\} \times \mathcal{A}$ in $\mathcal{B}(\mathbb{G},F')$. We denote by $[g,x]$ the class of $(g,x)$, so that $[1,x] = x$.
			
			The Galois action on the building is given by $\eta([g,x]) = [\eta(g),x],$ for all $\eta \in \mathrm{Gal}(F'/F).$ Thus, the class $[g,x]$ is fixed by $\mathrm{Gal}(F'/F)$ if and only if $	\eta(g^{-1})g \in \mathbb{G}(F')_{x, 0}$ for all $\eta \in \mathrm{Gal}(F'/F)$. This condition determines a unique $x \in \mathcal{A}$, and hence a unique fixed point $[g,x] \in \mathcal{B}(\mathbb{G},F)$ corresponding to $\mathbb{T}$.
		\end{proof}

		\begin{proposition}
			Let $\mathbb{T}(F)$ be one of the four anisotropic tori in $G$, represented by $\mathcal{T}_{1,1}$, $\mathcal{T}_{\varpi^{-1}, \varpi}$, $\mathcal{T}_{1, \varpi}$, and $\mathcal{T}_{1, \epsilon^{-1}\varpi}$, as in Proposition~\ref{conjugacyclassesofanisotropictoriofG}. Then the unique Galois-fixed point $y \in \mathcal{A}(\mathbb{G}, \mathbb{T}, F') \cap \mathcal{B}(\mathbb{G}, F)$ is given by $y = 0$ if $\mathbb{T}(F) = \mathcal{T}_{1,1}$, $y = 1$ if $\mathbb{T}(F) = \mathcal{T}_{\varpi^{-1}, \varpi}$, and $y = \tfrac{1}{2}$ if $\mathbb{T}(F) \in \{\mathcal{T}_{1,\varpi}, \mathcal{T}_{1,\epsilon^{-1}\varpi}\}$.
		\end{proposition}
		
		\begin{proof}  The splitting field $F'$ and the Galois group depend on the anisotropic torus $\mathbb{T}(F)$. In particular, if $\mathbb{T}(F) = \mathcal{T}_{1,1}$ or $\mathcal{T}_{\varpi^{-1}, \varpi}$, then we have $F' = E$, with $\mathrm{Gal}(E/F) = \{1, \sigma\}$,  and it acts on $\mathbb{G}(E)$ by $\sigma(g)=J{}^{\top}\overline{g}^{-1}J$. 
			If $\mathbb{T}(F) = \mathcal{T}_{1, \varpi}$ or $\mathcal{T}_{1, \epsilon^{-1}\varpi}$, then we have $F' = E[\sqrt{\varpi}]$, with $\mathrm{Gal}(F'/F) = \{1, \sigma, \delta, \sigma\delta\}$,  where $\delta$ acts entry-wise on the elements of $\mathbb{G}(F[\sqrt{\epsilon}, \sqrt{\varpi}])$.
			
			A direct computation shows that the element
			\[
			\matrixconjugatinganisotropictorustodiagonaltorus := \begin{pmatrix}1 & -\sqrt{\gamma_1^{-1} \gamma_2}^{-1} \\ \sqrt{\gamma_1^{-1} \gamma_2} & 1 \end{pmatrix}
			\]
			satisfies $\mathbb{T} = \mathbb{S}^{\matrixconjugatinganisotropictorustodiagonaltorus}$.  One verifies that for each such $\matrixconjugatinganisotropictorustodiagonaltorus$ and corresponding $\mathbb{T}$, the elements $\matrixconjugatinganisotropictorustodiagonaltorus$ and $\eta(\matrixconjugatinganisotropictorustodiagonaltorus^{-1})\matrixconjugatinganisotropictorustodiagonaltorus$ lie in $\mathbb{G}(F')_{y,0}$ for the value of $y$ stated in the lemma. Hence  the unique Galois-fixed point of $\mathcal{A}(\mathbb{G}, \mathbb{T}, F^{\prime})$ is $[\matrixconjugatinganisotropictorustodiagonaltorus,y]=[1,y]=y\in \mathcal{A}(\mathbb{G}, \mathbb{T}, F^{\prime})\cap\mathcal{A}\subset \mathcal{B}$.
					%
					%
					%
					%
					%
		\end{proof}
		The biquadratic extension  
		$
		\widetilde{E} = F(\sqrt{\epsilon}, \sqrt{\varpi})
		$
		of \(F\) has three quadratic intermediate fields
		$F(\sqrt{\epsilon}),\;F(\sqrt{\varpi}),$ and $F(\sqrt{\epsilon\varpi}).$ These are precisely the distinct quadratic extensions of \(F\). For any quadratic subfield \(F' \subset \widetilde{E}\), we write  
		$
		U(1)_{\widetilde{E}/F'}
		$
		for the group of elements of \(\widetilde{E}\) of norm one relative to the extension \(\widetilde{E}/F'\).
		Table~\ref{table:1} provides a complete list of the conjugacy classes of anisotropic tori in $G$ and the corresponding points in $\mathscr{B}(G)$.
		
		\begingroup
		\renewcommand{\arraystretch}{1.5}
		\begin{table}[htbp!]
			\centering
			\begin{tabular}{ |c|c|c|c|}
				
				\hline
				
				Isomorphism classes& Anisotropic torus& $t(a,b)=$ & $\mathcal{A}(G,\mathcal{T},F)$\\
				&$\mathbb{T}(F)=\mathcal{T}_{\gamma_1, \gamma_2}$&$\begin{psmallmatrix}
					a&b\gamma_1\\b\gamma_2&a
				\end{psmallmatrix}$&=\{$y$\}\\
				
				\hline
				
				$\begin{array}{c}
					U(1)_{E/F}\times U(1)_{E/F}\\
					\mathrm{unramified\;torus}
				\end{array}$&$\begin{array}{c}
					\mathcal{T}_{1,1}\\
					\mathcal{T}_{\varpi, \varpi^{-1}}=\mathcal{T}_{1, 1}^{\eta}
				\end{array}$& $\begin{array}{c}
					\begin{psmallmatrix}
						a&b\\b&a
					\end{psmallmatrix} \\
					\begin{psmallmatrix}
						a&b\varpi^{-1}\\b\varpi&a
					\end{psmallmatrix}
				\end{array}$&   $\begin{array}{c}
					y=0  \\
					y=1 
				\end{array}$\\ 
				
				\hline  
				$ \begin{array}{c}
					U(1)_{\widetilde{E} / F[\sqrt{\varpi}]}  \\
					\mathrm{ramified\;torus} 
				\end{array}$& $\mathcal{T}_{1, \varpi}$ & $\begin{psmallmatrix}
					a&b\\b\varpi&a
				\end{psmallmatrix}$&   $y=\frac{1}{2}$\\
				\hline 
				$ \begin{array}{c}
					U(1)_{\widetilde{E} / F[\sqrt{\epsilon\varpi}]} \\
					\mathrm{ramified\;torus} 
				\end{array}$  &$\mathcal{T}_{1, \epsilon^{-1}\varpi}$ & $\begin{psmallmatrix}
					a&b\epsilon^{-1}\\b\varpi&a
				\end{psmallmatrix}$&   $y=\frac{1}{2}$\\ 
				
				\hline
			\end{tabular}
			\vspace{1em}
			\caption{Representatives of equivalence classes of anisotropic tori $\mathcal{T}_{\gamma_1, \gamma_2}$ in $G$. Here $a, b$ are elements of $E$, and  $y$ is the image of building of $\mathcal{T}_{\gamma_1, \gamma_2}$ in the building of $G$.}
			\label{table:1}
		\end{table}
		\endgroup
		
		\subsubsection{Filtrations on anisotropic tori of $G$}			
		
		In~\S\ref{Moy-Prasad filtrations of G}, we described a filtration on the maximal parahoric subgroups of \(G\). We now use this to define a filtration on the anisotropic tori of \(G\) as follows. By~\cite[\S~2.5]{HM08}, for $r>0$ the filtration on $\mathcal{T}_{\gamma_1, \gamma_2}$ is simply the intersection of $\mathcal{T}_{\gamma_1, \gamma_2}$ with filtrations of $G_{y, r}$ where $y$ is the point corresponding to $\mathcal{T}_{\gamma_1, \gamma_2}$ in the building of $G$. In particular, for $r>0$ we have
		\[
		\mathcal{T}_{r} = \left\{\begin{pmatrix}a&b\gamma_1\\ b\gamma_2&a\end{pmatrix}\in\mathcal{T}_{\gamma_1, \gamma_2}\; \middle\vert\; a\in 1+\mathfrak{p}_{E}^{\lceil r\rceil},\;b\gamma_1\in \mathfrak{p}_{E}^{\lceil r-y \rceil} \right\}.
		\]
		The Lie algebra $\mathfrak{t}$ of the torus $\mathcal{T}$ is a two-dimensional subalgebra of $\mathfrak{g}$, spanned by
		\[
		\mathfrak{z} =
		\begin{pmatrix}
			\sqrt{\epsilon} & 0 \\
			0 & \sqrt{\epsilon}
		\end{pmatrix},
		\hspace{2em}
		Y_{\mathcal{T}} =
		\begin{pmatrix}
			0 & \gamma_1\sqrt{\epsilon} \\
			\gamma_2 \sqrt{\epsilon} & 0
		\end{pmatrix}.
		\]
		For any $r\in\mathbb{R}$, by~\cite[\S~2.5]{HM08} the corresponding filtration subring of $\mathfrak{t}$ is given by $\mathfrak{t}\cap \mathfrak{g}_{y,r}$, and therefore
		\[
		\mathfrak{t}_{r} = \left\{ a\mathfrak{z} + bY_{\mathcal{T}} \mid a\in\mathfrak{p}_{F}^{\lceil r\rceil},\;b\in\mathfrak{p}_{F}^{\lceil r-y \rceil} \right\}.
		\]
		For $r\in\mathbb{R}$, we define the filtration on the $F$-linear dual of $\mathfrak{t}$  by
		\[
		\mathfrak{t}_{r}^{*} = \left\{ \lambda\in \mathfrak{t^*} \mid \langle \lambda, Y\rangle\in \varpi\mathcal{O}_{E}\;\;\forall\;Y\in\mathfrak{t}_{s}\;\mathrm{with}\; s>-r \right\}.
		\]
		We can again identify $\mathfrak{t}^{*}$ with $\mathfrak{t}$ using the trace form.
		In particular, for $\Gamma_{u, v}=u\mathfrak{z}+vY_{\mathcal{T}}$ we define $\lambda_{u,v}\in\mathfrak{t}^*$ by the equation
		$\lambda_{u,v}(Y)=\mathrm{Tr}(\Gamma_{u,v}Y)$ for all $Y\in\mathfrak{t}$.

		\subsection{Representation theory}\label{representation theory}In this section, we recall some basic definitions and concepts from representation theory specializing to the case of our unitary group $G$ (see~\cite[\S I]{Car79}).
		
		\subsubsection{Classical definitions and results}
		Let $\mathbb{G}$ be a connected reductive algebraic group defined over $F$, and let $G$ denote the group of its $F$-rational points.
		\begin{lemma}\label{positive depth lemma}
			Let $(\lambda, \mathbb{C})$ be a character of $G$. Let $(\rho, V)$ be an irreducible representation of a compact open subgroup $H$ of $G$ modulo its center. Then 
			$$c  \scalebox{0.9}[0.9]{-} \mathrm{Ind}_{H}^{G}(\lambda \otimes \rho)\cong \lambda \otimes c  \scalebox{0.9}[0.9]{-} \mathrm{Ind}_{H}^{G}\rho. $$
		\end{lemma}

		We now review two important results that we will use quite frequently~\cite[\S2.4]{Bus06}.
		\begin{theorem}[Frobenius Reciprocity] Let $H$ be a closed subgroup of $G$,  $(\rho, W)$ a smooth representation of $H$ and  $(\pi, V)$ a smooth representation of $G$. Then 
			\begin{align*}
				\mathrm{Hom}_{G}(\pi, \mathrm{Ind}_{H}^{G}\rho)&\cong \mathrm{Hom}_{H}(\pi, \rho)
			\end{align*}	
		\end{theorem}	
		
		Similarly, for open subgroups $H$ of $G$, compact induction has its own form of Frobenius Reciprocity property.  
		\begin{theorem} Let $H$ be an open subgroup of $G$, let $(\rho, W)$ be a smooth representation of $H$ and $(\pi, V)$ a smooth representation of $G$. Then
			\begin{align*}
				\mathrm{Hom}_{G}( c  \scalebox{0.9}[0.9]{-} \mathrm{Ind}_{H}^{G}\rho, \pi)&\cong \mathrm{Hom}_{H}(\rho, \pi).
			\end{align*}	
		\end{theorem}	 
		
		\begin{remark} Note that if $G/ H$ is compact then $c  \scalebox{0.9}[0.9]{-} \mathrm{Ind}_{H}^{G}\rho=\mathrm{Ind}_{H}^{G}\rho$. Thus, in that case, induction is both left and right adjoint to restriction.
		\end{remark}
		We now mention a theorem that will be used repeatedly throughout the paper, as it describes how a compactly induced representation decomposes when restricted to a closed subgroup~\cite[Theorem 1.1]{Yuk22}.
		\begin{theorem}[Mackey decomposition]
			Let $H$ and $K$ be closed subgroups of $G$, and let $(\rho, W)$ be a smooth representation of $H$. Suppose that at least one of $H$ and $K$ is open in $G$. 
			Then we have the \emph{Mackey decomposition}
			\[
			\operatorname{Res}^{G}_{K}\,\mathrm{c}\text{-}\mathrm{Ind}_{H}^{G}\rho
			\;=\;
			\bigoplus_{c\in K\backslash G / H}
			\mathrm{c}\text{-}\mathrm{Ind}^{K}_{K\cap H^{c}}
			\operatorname{Res}^{H^{c}}_{K\cap H^{c}}\rho^{c}.
			\]
		\end{theorem}
		
		\begin{remark}
			In our setting, we shall see that the supercuspidal representations arise as compact inductions from compact open subgroups of the unitary group $G = \mathrm{U}(1,1)$, and to describe the  branching rules we will restrict these representations to the compact open subgroup $\mathcal{K}$ defined in \S\ref{the group U(1,1)}. 
			Since in locally compact totally disconnected groups every open subgroup is also closed, the above theorem applies in our case.
		\end{remark}

		We now recall the definition of the \emph{depth} of a representation, which provides a fundamental link between representations, the Moy--Prasad filtrations, and the building of $G$.
		\begin{definition}[Depth of a representation]
			Let $(\pi, V)$ be a smooth representation of $G$.  
			The \emph{depth} of $\pi$ is defined as the minimal real number $r \geq 0$ such that there exists a point $x$ in the building $\mathscr{B}(G)$ with $	V^{G_{x,r+}} \neq \{0\}.$
		\end{definition}
		Analogously, for a representation $(\sigma, W)$ of $G_{x,t}$ with $t \geq 0$, we define
		\begin{equation}\label{definitionofdeptp}
			\mathrm{depth}(\sigma) := \min\{ r \geq t \;\mid\; W^{G_{x,r+}} \neq \{0\} \}.
		\end{equation}

		\subsubsection{Generic characters of $\mathbb{T}(F)$}\label{genericcharacters}
		
		To construct positive-depth supercuspidal representations of $G$, one requires $G$-generic characters of positive depth of the anisotropic torus $\mathcal{T}$. In this subsection, we define and describe $G$-generic characters of positive depth for the maximal torus $\mathbb{T}(F)$, which includes both the diagonal torus~\eqref{definitionofmaximallyFsplitFtorus} and the anisotropic tori listed in Table~\ref{table:1}.
		
		Let $\mathbb{T}$ be a maximal torus of $\mathbb{G}$, and let 
		$\Phi_{\mathrm{abs}}(\mathbb{G}, \mathbb{T})$ denote the absolute root system of $\mathbb{G}$ with respect to $\mathbb{T}$. 
		For each root $a \in \Phi_{\mathrm{abs}}(\mathbb{G}, \mathbb{T})$, set 
		$H_a = da^{\vee}(1) \in \mathfrak{g},$ where $da^{\vee}$ denotes the derivative of the coroot $a^{\vee}$ associated to $a$.
		
		\begin{definition}\label{def:G-generic}
			An element $\lambda \in \mathfrak{t}^*_{-r}$ is called \emph{$G$-generic of depth $-r$} if $\nu(\lambda(H_a)) = -r$ for each
			$a \in \Phi_{\mathrm{abs}}(\mathbb{G}, \mathbb{T})$.	
		\end{definition}
		Since we have identified $\mathfrak{t}^*$ with $\mathfrak{t}$, we may equivalently say that 
		an element $\Gamma \in \mathfrak{t}_{-r}$ is $G$-generic of depth $r$ if $\nu\bigl(\mathrm{Tr}(\Gamma H_a)\bigr) = -r$ for each $ a \in \Phi_{\mathrm{abs}}(\mathbb{G}, \mathbb{T}).$
		
		\begin{lemma}\label{lem:generic-depths}
			Let $T=\mathbb{T}(F)$ be a maximal $F$-torus of $G$.
			
			If $T$ is the maximally $F$-split torus as in~\eqref{definitionofmaximallyFsplitFtorus}, then the depth $r$ of any $G$-generic element of $\mathfrak{t}$ satisfies $r\in\mathbb{Z}$.  
			
			If $T=\mathcal{T}_{\gamma_1,\gamma_2}$ is an anisotropic torus from Table~\ref{table:1}, then the depth $r$ of any $G$-generic element in $\mathfrak{t}$ satisfies  
			\[
			r\in \mathbb{Z}\quad\text{when }\mathcal{T}_{\gamma_1, \gamma_2}\;\text{is unramified},
			\qquad
			r\in \tfrac{1}{2}+\mathbb{Z}\quad\text{when }\mathcal{T}_{\gamma_1, \gamma_2}\;\text{is ramified}.
			\]
		\end{lemma}

		\begin{proof} Note that for our unitary group $\mathbb{G}$, the absolute root system $\Phi_{\mathrm{abs}}(\mathbb{G}, \mathbb{T})$ has a unique positive root $\alpha$.
			In the split case we have $	H_{\alpha}=\begin{psmallmatrix}1&0\\0&-1\end{psmallmatrix}.$ Let $\Gamma=\begin{psmallmatrix}a&0\\0&-\overline{a}\end{psmallmatrix}\in\mathfrak{t}_{-r}$ with $a=a_0+a_1\sqrt{\epsilon}$. Then $\nu(a_{i})\geq r$ and $\mathrm{Tr}(\Gamma H_{\alpha})=a+\overline{a}=2a_0.$ The genericity condition gives $\nu(2a_0)=-r$, hence $\nu(a_0)=-r$. Since $a_0\in F$ and valuations on $F$ are integral, this forces $r\in\mathbb{Z}$.
			
			On the other hand, if we take $\mathbb{T}(F)=\mathcal{T}_{\gamma_1, \gamma_2}$ to be one of the anisotropic tori listed in Table~\ref{table:1}, then 
			$H_{\alpha}=\begin{psmallmatrix}1&0\\ 0&-1\end{psmallmatrix}^{\matrixconjugatinganisotropictorustodiagonaltorus},$
			where $\matrixconjugatinganisotropictorustodiagonaltorus=\begin{psmallmatrix}1&-\sqrt{\gamma_{1}^{-1}\gamma_{2}}^{-1}\\\sqrt{\gamma_{1}^{-1}\gamma_{2}}&1 \end{psmallmatrix}$ is the matrix conjugating $\mathbb{T}$ to $\mathbb{S}$. An element $\Gamma_{u,v}=\begin{psmallmatrix}u\sqrt{\epsilon}&v\sqrt{\epsilon}\\v\sqrt{\epsilon}&u\sqrt{\epsilon}\end{psmallmatrix} \in \mathfrak{t}_{-r}$ is $G$-generic of depth $-r$ if and only if $\nu\left(\mathrm{Tr}(\Gamma_{u,v}H_{\alpha})\right) 
			= \nu\left(2v\sqrt{\epsilon}\,\sqrt{\gamma_1\gamma_2}\right) 
			= -r.$ Thus the condition $\nu(\mathrm{Tr}(\Gamma_{u,v}H_{\alpha}))=-r$ is equivalent to $\nu(v) = -r - \nu(\sqrt{\gamma_1\gamma_2}).$ As $\nu(\sqrt{\gamma_1\gamma_2}) \in \{0, \tfrac{1}{2}\}$ and $\nu(v) \in \mathbb{Z}$, 
			we deduce that 
			\begin{equation}\label{depth is Z or half plus Z}
				r \in \mathbb{Z} \quad \text{when } \nu(\sqrt{\gamma_1\gamma_2}) = 0,
				\qquad
				r \in \tfrac{1}{2} + \mathbb{Z} \quad \text{when } \nu(\sqrt{\gamma_1\gamma_2}) = \tfrac{1}{2}.
			\end{equation}
			Since $\nu(\sqrt{\gamma_1\gamma_2})=0$ precisely when $\mathcal{T}_{\gamma_1,\gamma_2}$ is unramified, 
			and $\nu(\sqrt{\gamma_1\gamma_2})=\tfrac{1}{2}$ precisely when $\mathcal{T}_{\gamma_1,\gamma_2}$ is ramified, 
			the lemma follows.
		\end{proof}

		We are now in a position to define $G$-generic characters of $\mathbb{T}(F)$ of depth $r$. Set $s:=\frac{r}{2}$ and let $e$ be the Moy--Prasad isomorphisms of the abelian groups
		$$\mathfrak{t}_{s+}/\mathfrak{t}_{r+}\rightarrow\mathcal{T}_{s+}/\mathcal{T}_{r+},\hspace{2em}\mathfrak{g}_{x,s+}/\mathfrak{g}_{x, r+}\rightarrow G_{x,s+}/G_{x,r+}.$$
		Let $\psi'$ be an additive character of the field $F$, chosen to be trivial on the maximal ideal $\mathfrak{p}_F$ and nontrivial on the ring of integers $\mathcal{O}_F$. Using $\psi'$, we define an additive character $\psi$ of the quadratic extension $E$ by  $\psi(x) \;=\; \psi'\!\left( \frac{x + \overline{x}}{2} \right)$ for all $x \in E$.
		Then $\psi$ is a character of $E$ that is nontrivial on $\mathcal{O}_{E}$ but trivial on $\mathfrak{p}_{E}$. We will fix this choice of $\psi$ for the remainder of the paper.
		\begin{definition}   
			A smooth complex-valued character $\phi$ of $\mathbb{T}(F)$ is called \emph{$G$-generic of depth $r$} if it is trivial on $\mathbb{T}(F)_{r+}$, non-trivial on $\mathbb{T}(F)_{r}$ and is realized by an element $\Gamma \in \mathfrak{t}_{-r}$ that is $G$-generic of depth $r$, \emph{i.e.}, for every $t \in \mathbb{T}(F)_{s+}$, $\phi(t) = \psi(\mathrm{Tr}(\Gamma(e^{-1}(t)))).$
		\end{definition}
		
		Note that the image of $\Gamma$ in $\mathfrak{t}_{-r}/\mathfrak{t}_{-s}$ is uniquely determined by this relation, and each distinct representative gives rise to a distinct $G$-generic character of $\mathbb{T}(F)$ of depth $r$. 
		
		\begin{remark}
			To define a $G$-generic character $\phi$ of depth $r$ for an anisotropic torus $\mathcal{T}$, one requires a $G$-generic element of $\mathfrak{t}_{-r}$ of depth $r$. 
			By Lemma~\ref{lem:generic-depths}, such an element can exist only when $r$ is an integer or a half-integer. 
			Therefore, the depths of $G$-generic characters are integral when $\mathcal{T}$ is unramified and half-integral when $\mathcal{T}$ is ramified.
		\end{remark}

		\section{Irreducible representations of $\mathcal{K}$}\label{irreduciblerepresentationsofK}
		In this section, we recall the description of irreducible representations of $\mathcal{K}$ that we constructed in~\cite[\S5]{ET251}.
		Let $d\in\mathbb{Z}_{>0}$ and let $z, u, v\in F$ such that $\nu(z)\geq-d$ and $\nu(v)>\nu(u)=-d$. We choose a particular $X$ element of $\mathfrak{g}_{0, -d}$ of the following form:
		\begin{equation}\label{definitionofX}
			X=X(z)+\widetilde{X}(u, v)=\begin{pmatrix}z\sqrt{\epsilon}& 0\\ 0&z \sqrt{\epsilon}\end{pmatrix}+\begin{pmatrix}0&u\sqrt{\epsilon}\\
				v\sqrt{\epsilon}&0\end{pmatrix}\in \mathfrak{g}_{0, -d}.\end{equation}

		The function $k \mapsto \Psi_X(k) = \psi(\mathrm{Tr}(X(k-I)))$ defines a character of the group $\mathcal{K}_{\frac{d}{2}+}$ of depth $d$, where $\psi$ is an additive character of $E$ that is trivial on $\mathfrak{p}_{E}$ but nontrivial on $\mathcal{O}_{E}$ that we fixed in~\ref{genericcharacters}. The centralizer of \( X \) in \( \mathcal{K} \) is denoted by \( T(X) \), and is equal to
		\begin{equation}\label{definitionofTofX}
			T(X) = \left\{   \begin{pmatrix} a & b \\ bu^{-1}v & a \end{pmatrix}\;\middle\vert\; a,b\in \mathcal{O}_{E}, a\overline{a}+b\overline{b}u^{-1}v =1, \overline{a}b\in\sqrt{\epsilon}\mathcal{O}_{F} \right\} = \mathcal{T}_{1,u^{-1}v}.
		\end{equation}		
		The following theorem recalls the construction of a family of irreducible representations of $\mathcal{K}$ associated with elements of $\mathfrak{g}_{0,-d}$.	
		\begin{theorem}\label{Rep of K}
			Let $X$ be as in~\eqref{definitionofX}, and let $\zeta$ be a character of $T(X)$ which coincides with $\Psi_X$ on the intersection $T(X) \cap \mathcal{J}_{d}$.  Write $\Psi_{X, \zeta}$ for the unique character of $T(X)\mathcal{J}_{d}$ which extends $\zeta$ and $\Psi_X$. Then 
			\begin{equation}\label{repofK1}
				\mathcal{S}_d(X, \zeta) := \mathrm{Ind}_{T(X)\mathcal{J}_{d}}^\mathcal{K} \Psi_{X, \zeta}
			\end{equation}
			is an irreducible representation of $\mathcal{K}$ of depth $d$ and of degree $q^{d-1}(q^2-1)$.
		\end{theorem}		
		
		\subsection{Representations associated with nilpotent orbits}\label{nilpotentorbits}
		
		It is shown in~\cite[\S5.3]{ET251} that there are three nilpotent $G$-orbits in $\mathfrak{g}$, denoted by  $\mathcal{N}_{\delta}=G\cdot X_{\delta}$, where $X_{\delta}=\begin{psmallmatrix}0&\delta \sqrt{\epsilon}\\ 0&0 \end{psmallmatrix}$ and $\delta \in \{0, 1, \varpi\} $. For $d\in\mathbb{Z}_{>0}$ and $a\in\mathcal{O}_{F}^{\times}$, let $X_{a\varpi^{-d}}\in\mathfrak{g}_{0,-d}$. Then the centralizer $T(X_{a\varpi^{-d}})$ of $X_{a\varpi^{-d}}$ in $\mathcal{K}$ is given by $Z\mathcal{U}$, where $\mathcal{U}=\left\{\begin{psmallmatrix}1&\sqrt{\epsilon}b\\0&1\end{psmallmatrix}\right\}\cap \mathcal{K}$. 
		
		The following theorem summarizes the results proved in~\cite[\S5.3]{ET251}.
		\begin{theorem}\label{repfromnilpotentelements}
			Let $d\in\mathbb{Z}_{>0}$, $X_{a\varpi^{-d}}\in\mathfrak{g}_{0,-d}$ with $a\in\mathcal{O}_{F}^{\times}$. Let $\theta$ be a character of $Z$ extented trivially to $\mathcal{U}$, and let $\Psi_{X_{a\varpi^{-d}}, \theta}$ for the unique character of $Z\mathcal{U}\mathcal{J}_{d}$ that extends $\theta$ and $\Psi_{X_{a\varpi^{-d}}}$. Then  
			\begin{equation}\label{nilpotentrep1}
				\mathcal{S}_d(X_{a\varpi^{-d}}, \theta) := \mathrm{Ind}_{Z\mathcal{U}\mathcal{J}_{d}}^\mathcal{K} \Psi_{X_{a\varpi^{-d}}, \theta}
			\end{equation}
			is an irreducible representation of $\mathcal{K}$ of depth $d$ and of degree $q^{d-1}(q^2-1)$. Moreover, for all $a\in\mathcal{O}_{F}^{\times}$ 
			\begin{equation}
				\mathcal{S}_{d}(X_{\varpi^{-d}}, \theta)\cong \mathcal{S}_{d}(X_{a\varpi^{-d}}, \theta).
			\end{equation}	
		\end{theorem}	
		
		Finally, we recall the following lemma, proved in~\cite[Lemma~5.9]{ET251}, which will be used in the subsequent sections.
		\begin{lemma}\label{d>2rcong}
			Let $X = X(z) + \widetilde{X}(u,v) \in \mathfrak{g}_{0,-d}$ be such that $\nu(z), \nu(v) > -\lceil \tfrac{d}{2} \rceil$ and $\nu(u) = -d$. Then
			\[
			\Psi_{X} = \Psi_{X_{u}}
			\qquad \text{and} \qquad
			T(X)\mathcal{J}_{d} = T(X_{u})\mathcal{J}_{d} = Z\mathcal{U}\mathcal{J}_{d}.
			\]
		\end{lemma}

\section{Branching rules for depth-zero supercuspidal representations of $G$}\label{chapter 4}

In this section, we provide an explicit description of all depth-zero irreducible supercuspidal representations of \( G \), together with their branching rules upon restriction to \( \mathcal{K} \).

\subsection{Construction of the representations}The depth-zero irreducible supercuspidal representations of $G$ are induced from the cuspidal representations of  $\mathbb{G}(\mathfrak{f})= \mathbb{U}(1,1)(\mathfrak{f})$, which are well-known; see, for example, \cite{Campbell2014}. Briefly: let $\mathfrak{e}$ denote the unique quadratic extension field of the residue field $\mathfrak{f}$ of $\mathcal{O}_{F}$, and let $N \colon \mathfrak{e} \to \mathfrak{f}$ be the norm map.  For distinct characters \( \alpha \) and \( \beta \) of  \(\ker(N)=\mathfrak{e}^{1}\), Deligne-Lusztig induction associates a representation \( \sigma = \sigma(\alpha, \beta) \) of \( \mathbb{U}(1,1)(\mathfrak{f}) \) of degree \( q-1 \). This representation is cuspidal, and all cuspidal representations of \( \mathbb{U}(1,1)(\mathfrak{f}) \) arise in this way.

Since $\begin{psmallmatrix} x&y \\ 0& x\end{psmallmatrix}=\begin{psmallmatrix}0& 1\\1 &0\end{psmallmatrix}\begin{psmallmatrix}x&0 \\ y&x\end{psmallmatrix}\begin{psmallmatrix}0 &1\\ 1&0\end{psmallmatrix}$, the matrices $\begin{psmallmatrix} x&y \\ 0& x\end{psmallmatrix}$ and $\begin{psmallmatrix}x&0 \\ y&x\end{psmallmatrix}$ belong to the same conjugacy class of $G$. Therefore, Table~\ref{Table2} (c.f. \cite{Campbell2014}) below  summarizes the values of the character \( \chi_{\sigma(\alpha, \beta)} \) corresponding to the cuspidal representation \( \sigma(\alpha, \beta) \) of \( \mathbb{U}(1,1)(\mathfrak{f}) \).
\begingroup
\renewcommand{\arraystretch}{1.5}
\begin{table}[htbp!]
	\centering
	\begin{tabular}{ |c|c|c|c|c| }
		\hline
		Conjugacy class &$\begin{psmallmatrix}x&0\\0&x\end{psmallmatrix},$&$\begin{psmallmatrix}x&0\\y&x\end{psmallmatrix}$, &$\begin{psmallmatrix}x&0\\0&y\end{psmallmatrix}$, &$\begin{psmallmatrix}x&y\\y&x\end{psmallmatrix}$, \\
		representatives&$x\in\mathfrak{e}^{1}$ &$x\in\mathfrak{e}^{1}, y\neq 0$ & $x\neq y$&$y\neq 0$ \\
		\hline
		$\chi_{\sigma(\alpha, \beta)}$&$(q-1)\alpha(x)\beta(x)$& $-\alpha(x)\beta(x)$& $0$&$\alpha(x+y)\beta(x-y)+$\\
		& & & & $\alpha(x-y)\beta(x+y)$ \\
		\hline
	\end{tabular} 
	\vspace{1em}
	\caption{Character $\chi_{\sigma(\alpha, \beta)}$ of the cuspidal representation $\sigma(\alpha, \beta)$ of $\mathbb{U}(1,1)(\mathfrak{f})$.}
	\label{Table2} 
\end{table}
\endgroup

Let $\sigma$ be a cuspidal representation of $\mathbb{G}(\mathfrak{f})\cong \mathcal{K}/\mathcal{K}_{0+}$.  Inflate $\sigma$ to a representation (also denoted by $\sigma$) of $\mathcal{K}$, and let $\sigma^\eta$ denote the corresponding representation of $\mathcal{K}^\eta$, where $\eta=\begin{psmallmatrix}1&0\\ 0&\varpi\end{psmallmatrix}$. As a special case of the general result for depth-zero supercuspidal representations found in \cite{MoyPrasad1996, Morris1999}, one has the following well-known fact.
\begin{proposition}
	Let $\sigma$ be a cuspidal representation of $\mathbb{G}(\mathfrak{f}) = \mathcal{K}/\mathcal{K}_{0+}$.  
	Then the compact inductions
	$$c  \scalebox{0.9}[0.9]{-} \mathrm{Ind}_{\mathcal{K}}^{G}\sigma\;\hspace{2em}\mathrm{and}\;\hspace{2em}c  \scalebox{0.9}[0.9]{-} \mathrm{Ind}_{\mathcal{K}^{\eta}}^{G}\sigma^{\eta}$$
	are  depth-zero irreducible supercuspidal representations of $G$, and every depth-zero irreducible supercuspidal representation of $G$ arises in this way.
\end{proposition}

\begin{remark}
	From now on, we identify a cuspidal representation $\sigma$ of $\mathcal{K}/\mathcal{K}_{0+}$ with its inflation to $\mathcal{K}$.
\end{remark}
\subsection{Restriction to $\mathcal{K}$}\label{depthzero rep of G}

In this section, we restrict the depth-zero irreducible supercuspidal representations of $G$ to $\mathcal{K}$ and obtain a canonical decomposition using Mackey theory. The Mackey components occurring in this decomposition may \emph{a priori} be reducible. Therefore,  we compute the degree of each component and determine the maximal depth of any irreducible subrepresentation it contains. The proof that these components are in fact irreducible is deferred to the next section, which builds on the results established here.

Let $\sigma$ be a cuspidal representation of $\mathcal{K}/\mathcal{K}_{0+}$. Then from Mackey theory it follows that 
$$
\mathrm{Res}_{\mathcal{K}}^{G}\mbox{c-Ind}_{\mathcal{K}}^{G}\sigma\cong\bigoplus_{g\in \mathcal{K}\backslash G/\mathcal{K}  }\mathrm{Ind}_{\mathcal{K}\cap (\mathcal{K})^{g}}^{\mathcal{K}} \sigma^{g}\hspace{1em}\mathrm{and}\hspace{1em}\mathrm{Res}_{\mathcal{K}}c  \scalebox{0.9}[0.9]{-} \mathrm{Ind}_{\mathcal{K}^{\eta}}^{G}\sigma^{\eta}=\bigoplus_{g\in \mathcal{K}\backslash G/\mathcal{K}^{\eta} }\mathrm{Ind}_{\mathcal{K}\cap\mathcal{K}^{g\eta}}^{\mathcal{K}}\sigma^{g\eta}.
$$
By Lemma~\ref{doublecosetrepresentativeofKGK}, a set of double coset representatives for $\mathcal{K}\backslash G/ \mathcal{K}$ or $\mathcal{K}\backslash G/ \mathcal{K}^{\eta}$ is given by
$\left\{\alpha^{t}\;|\; t\geq 0\right\},$ where $\alpha^{t}:=\begin{psmallmatrix}\varpi^{-t}& 0\\ 0& \varpi^{t}\end{psmallmatrix}$.
Therefore, we obtain the following decompositions
\begin{equation}\label{firstformfordepthzer0}\mathrm{Res}_{\mathcal{K}}c  \scalebox{0.9}[0.9]{-} \mathrm{Ind}_{\mathcal{K}}^{G}\sigma=\bigoplus_{t\geq 0}\mathrm{Ind}_{\mathcal{K}\cap\mathcal{K}^{\alpha^{t}}}^{\mathcal{K}}\sigma^{\alpha^{t}}\hspace{1em}\mathrm{and}\hspace{1em}\mathrm{Res}_{\mathcal{K}}c  \scalebox{0.9}[0.9]{-} \mathrm{Ind}_{\mathcal{K}^{\eta}}^{G}\sigma^{\eta}=\bigoplus_{t\geq 0}\mathrm{Ind}_{\mathcal{K}\cap\mathcal{K}^{\alpha^{t}\eta}}^{\mathcal{K}}\sigma^{\alpha^{t}\eta}.
\end{equation}	
For each $t>0$, the matrix $z_{t}=\varpi^{-t}I$ satisfies $\alpha^{t}=z_{t}\eta^{2t}$. Since $z_{t}$ centralizes $G$, the conjugated representations $\sigma^{\alpha^{t}}$ and $\sigma^{\eta^{2t}}$ are equal. Therefore, we may write~\eqref{firstformfordepthzer0} as
\begin{equation}\label{secondformfordepthzer0}\mathrm{Res}_{\mathcal{K}}c  \scalebox{0.9}[0.9]{-} \mathrm{Ind}_{\mathcal{K}}^{G}\sigma=\bigoplus_{t\geq 0}\mathrm{Ind}_{\mathcal{K}\cap\mathcal{K}^{\eta^{2t}}}^{\mathcal{K}}\sigma^{\eta^{2t}}\hspace{1em}\mathrm{and}\hspace{1em}\mathrm{Res}_{\mathcal{K}}c  \scalebox{0.9}[0.9]{-} \mathrm{Ind}_{\mathcal{K}^{\eta}}^{G}\sigma^{\eta}=\bigoplus_{t\geq 0}\mathrm{Ind}_{\mathcal{K}\cap\mathcal{K}^{\eta^{2t+1}}}^{\mathcal{K}}\sigma^{\eta^{2t+1}}.
\end{equation}

\begin{lemma}\label{KcapKd=BKd}
	Let $d \in \mathbb{Z}_{>0}$. Then $\mathcal{K} \cap \mathcal{K}^{\eta^{d}} = \mathcal{B} \mathcal{K}_{d}$, 
	where $\mathcal{B}$ denotes the subgroup of upper triangular matrices in $\mathcal{K}$.
\end{lemma}	

\begin{proof}
	Let $a = (a_{ij}) \in \mathcal{K} \cap \mathcal{K}^{\eta^{d}}$. Then $a \in \mathcal{K} \cap \mathcal{K}^{\eta^{d}}\; \text{if and only if}\;(\eta^{d})^{-1} a \eta^{d} \in \mathcal{K}\;\text{and}\;a\in\mathcal{K}.$ Solving this condition yields $\nu(a_{21}) \geq d$. By Lemma~\ref{Kproperty}, either $a_{11}, a_{22} \in \mathcal{O}_{E}^{\times}$ or $a_{12}, a_{21} \in \mathcal{O}_{E}^{\times}$.  Since $a_{21} \in \mathfrak{p}_{E}^{d}$, it follows that $a_{11}, a_{22} \in \mathcal{O}_{E}^{\times}$. 
	We claim that $a=bk$ with 
	$$b=\begin{pmatrix} \overline{a_{22}}^{-1} & a_{12} \\ 0 & a_{22} \end{pmatrix}\in \mathcal{B}\quad\text{and}\quad k=\begin{pmatrix}
		1 & 
		0\\
		a_{21} a_{22}^{-1} & 1
	\end{pmatrix}\in\mathcal{K}_{d}.$$
	Indeed, since $a \in \mathcal{K}$, we have $\overline{a_{12}} a_{22}+ a_{12}\overline{a_{22}}=0$ and $a_{11}\overline{a_{22}}+a_{21}\overline{a_{12}}=1$, therefore 
	\begin{equation*}
		\overline{a_{22}}^{-1}+a_{12}a_{21}a_{22}^{-1}=\frac{1+\overline{a_{22}}a_{12}a_{21}a_{22}^{-1}}{\overline{a_{22}}}=\frac{1-a_{22}\overline{a_{12}}a_{21}a_{22}^{-1}}{\overline{a_{22}}}=\frac{a_{11}\overline{a_{22}}}{\overline{a_{22}}}=a_{11}.
	\end{equation*}	
	Since $a_{22}\in\mathcal{O}_{E}^{\times}$ and $\overline{a_{12}}a_{22}\in \sqrt{\epsilon} \mathcal{O}_{F}$, we have $b\in\mathcal{B}$.	As $a = bk$ with $a, b \in \mathcal{K}$, we must have $k \in \mathcal{K}$. Moreover, since $a_{21} \in \mathfrak{p}_{E}^{d}$ and $a_{22}\in\mathcal{O}_{E}^{\times}$, it follows that $a_{21}a_{22}^{-1}\in\mathfrak{p}_{E}^{d}$, and hence $k\in \mathcal{K}_{d}$.  Therefore we have shown that
	$\mathcal{K} \cap \mathcal{K}^{\eta^{d}} \subseteq \mathcal{B} \mathcal{K}_{d}.$

	Conversely, let $h \in \mathcal{B} \mathcal{K}_{d}$ with $h = bk$, where $b = (b_{ij}) \in \mathcal{B}$ and $k = (k_{ij}) \in \mathcal{K}_{d}$. Clearly, $h \in \mathcal{K}$. Moreover, since $b_{21} = 0$ and $k_{21} \in \mathfrak{p}_{E}^{d}$, a direct computation shows that $(\eta^{d})^{-1} h \eta^{d} = (\eta^{d})^{-1} b \eta^{d} \, (\eta^{d})^{-1} k \eta^{d} \in \mathcal{K},$ and hence $h \in \mathcal{K} \cap \mathcal{K}^{\eta^{d}}$. Therefore, $\mathcal{B} \mathcal{K}_{d} \subseteq \mathcal{K} \cap \mathcal{K}^{\eta^{d}},$ and the lemma follows.
\end{proof}

By the preceding lemma, for any $d>0$, $\mathcal{K}\cap\mathcal{K}^{\eta^{d}}=\mathcal{B}\mathcal{K}_{d}$; that is; 
elements of $\mathcal{B}\mathcal{K}_{d}$ can be represented by matrices $(a_{ij})\in\mathcal{K}$ such that $a_{21}\in\mathfrak{p}_{E}^{d}$. Consequently,~\eqref{secondformfordepthzer0} may be rewritten as 
\begin{equation}\label{decompositiondepthzero1}
	\mathrm{Res}_{\mathcal{K}}c  \scalebox{0.9}[0.9]{-} \mathrm{Ind}_{\mathcal{K}}^{G}\sigma=\sigma \oplus\bigoplus_{d\in 2\mathbb{Z}_{>0}}\mathrm{Ind}_{\mathcal{B}\mathcal{K}_{d}}^{\mathcal{K}}\sigma^{\eta^{d}}\hspace{1em}\mathrm{and}\hspace{1em}\mathrm{Res}_{\mathcal{K}}c  \scalebox{0.9}[0.9]{-} \mathrm{Ind}_{\mathcal{K}^{\eta}}^{G}\sigma^{\eta}=\bigoplus_{d\in 2\mathbb{Z}_{>0}+1}\mathrm{Ind}_{\mathcal{B}\mathcal{K}_{d}}^{\mathcal{K}}\sigma^{\eta^{d}}.
\end{equation}

\begin{lemma}\label{valueonBKd}
	Let $d\in\mathbb{Z}_{>0}$. Then $(\mathcal{B}\mathcal{K}_{d})^{\eta^{-d}}=\mathcal{B}^{op}\mathcal{K}_{d}$.
\end{lemma}

\begin{proof}
	Let $a = (a_{ij}) \in \mathcal{B}\mathcal{K}_{d} $. Then $a^{\eta^{-d}}=\begin{psmallmatrix} a_{11}& a_{12}\varpi^{d}\\ a_{21}\varpi^{-d}& a_{22}\end{psmallmatrix}$. By Lemma~\ref{KcapKd=BKd} we have $a_{21}\in\mathfrak{p}_{E}^{d}$ and $a_{11}, a_{22} \in \mathcal{O}_{E}^{\times}$. We claim that we can factor the matrix $a^{\eta^{-d}}$ as $bk$, where
	\begin{equation*}
		b = \begin{pmatrix} a_{11} & 0 \\ a_{21}\varpi^{-d} & \overline{a_{11}}^{-1} \end{pmatrix}\in\mathcal{B}^{op} 
		\quad \text{and} \quad
		k = \begin{pmatrix} 1 & a_{11}^{-1} a_{12} \varpi^{d} \\ 0 & 1 \end{pmatrix}\in \mathcal{K}_{d}.
	\end{equation*}
	Since $a$ is an element of $\mathcal{K}$, we have $\overline{a_{11}}a_{21}+a_{11}\overline{a_{21}}=0$, and it follows that
	\begin{equation*}
		a_{11}^{-1}a_{12}a_{21}+\overline{a_{11}}^{-1}=\frac{a_{11}^{-1}\overline{a_{11}}a_{21}a_{12}+1}{\overline{a}_{11}}=\frac{-a_{11}^{-1}a_{11}\overline{a_{21}}a_{12}+1}{\overline{a_{11}}}=\frac{\overline{a_{11}}a_{22}}{\overline{a_{11}}}=a_{22}.
	\end{equation*}
	Thus $a=bk$. As $a_{11} \in \mathcal{O}_{E}^{\times}$ and  $\overline{a_{11}} a_{21} \in \sqrt{\epsilon}\mathcal{O}_{F}$, we have $b \in \mathcal{B}^{op}$.  Since $a = bk$ with $a, b \in \mathcal{K}$ and $a_{12}\in\mathfrak{p}_{E}^{d}$ , we have $k \in \mathcal{K}_{d}$. Thus we have shown that $(\mathcal{B}\mathcal{K}_{d})^{\eta^{-d}}\subseteq \mathcal{B}^{op}\mathcal{K}_{d}$. 
	
	Conversely, let $a=(a_{ij})\in \mathcal{B}^{op}\mathcal{K}_{d}$. Then $a\in (\mathcal{B}\mathcal{K}_{d})^{\eta^{-d}}$ if and only if $\eta^{d}a\eta^{-d}=a^{\eta^{d}}\in \mathcal{B}\mathcal{K}_{d}$. Since the $(1,1)$-entry of any element of $\mathcal{B}^{op}$ lies in $\mathcal{O}_{E}^{\times}$ and the $(1,2)$-entry of any element of $\mathcal{K}_{d}$ lies in $\mathfrak{p}_{E}^{d}$, we deduce that $a_{12}\in \mathfrak{p}_{E}^{d}$. Thus we can factor $a^{\eta^{d}}$ as $a^{\eta^{d}} = bk$, where
	\[
	b = \begin{pmatrix} \overline{a_{22}}^{-1} & a_{12}\varpi^{-d} \\ 0 & a_{22} \end{pmatrix} 
	\quad \text{and} \quad
	k = \begin{pmatrix}
		1 & 
		0\\
		a_{21} a_{22}^{-1}\varpi^{d} & 1
	\end{pmatrix}.
	\]
	As in the proof of Lemma~\ref{KcapKd=BKd}, since $a \in \mathcal{K}$, we have $\overline{a_{12}} a_{22}+ a_{12}\overline{a_{22}}=0$ and $a_{11}\overline{a_{22}}+a_{21}\overline{a_{12}}=1$, therefore 
	\begin{equation*}
		\overline{a_{22}}^{-1}+a_{12}a_{21}a_{22}^{-1}=\frac{1+\overline{a_{22}}a_{12}a_{21}a_{22}^{-1}}{\overline{a_{22}}}=\frac{1-a_{22}\overline{a_{12}}a_{21}a_{22}^{-1}}{\overline{a_{22}}}=\frac{a_{11}\overline{a_{22}}}{\overline{a_{22}}}=a_{11}.
	\end{equation*}	
	Since $a\in\mathcal{K}$ and $a_{12}\in\mathfrak{p}_{E}^{d}$, it follows that $a_{22} \in \mathcal{O}_{E}^{\times}$ and $\overline{a_{22}} a_{12} \in \sqrt{\epsilon}\mathcal{O}_{F}$, hence $b \in \mathcal{B}$. As $a = bk$ with $a, b \in \mathcal{K}$ and $a_{12}\in\mathfrak{p}_{E}^{d}$ , we have $k \in \mathcal{K}_{d}$. Thus we have shown that $a^{\eta^{d}}\in\mathcal{B}\mathcal{K}_{d}$,  hence $a\in (\mathcal{B}\mathcal{K}_{d})^{\eta^{-d}}$, and therefore  the inclusion $\mathcal{B}^{op}\mathcal{K}_{d}\subseteq (\mathcal{B}\mathcal{K}_{d})^{\eta^{-d}} $ follows. Combining both the inclusions we obtain $\mathcal{B}^{op}\mathcal{K}_{d}= (\mathcal{B}\mathcal{K}_{d})^{\eta^{-d}}$.
\end{proof}

\begin{proposition} \label{L:depth}
	Let $\sigma$ be a cuspidal representation of $\mathcal{K}/\mathcal{K}_{0+}$.  Then for any $d\geq 0$, the maximum depth of any irreducible component of $\mathrm{Ind}_{\mathcal{B}\mathcal{K}_d}^\mathcal{K} \sigma^{\eta^d}$ is $d$. Moreover, there exists at least one component of depth $d$.
\end{proposition}
\begin{proof}  When $d = 0$, the summand is simply $\sigma$, which has depth zero by hypothesis. 
	Let $d > 0$. By Lemma~\ref{valueonBKd}, the action of $\sigma^{\eta^d}$ on $\mathcal{B}\mathcal{K}_{d}$ is given by the action of $\sigma$ on $\mathcal{B}^{op}\mathcal{K}_{d}$, and since $\sigma$ is trivial on $\mathcal{K}_{d}$, this is determined by $\mathrm{Res}_{\mathcal{B}^{op}}\sigma$. Explicitly,  for $a = (a_{ij})\in \mathcal{B}\mathcal{K}_d$
	we have 
	\begin{equation*} \label{E:match}
		\sigma^{\eta^d}(a) = \sigma\left(a^{\eta^{-d}}\right) = \sigma\left(\mat{a_{11}&a_{12}\varpi^{d}\\ a_{21}\varpi^{-d} & a_{22}}\right) 
		=  \sigma\left(\mat{a_{11}&0\\ a_{21}\varpi^{-d} & \overline{a_{11}}^{-1}}\right).
	\end{equation*}
	Since $\mathcal{K}_{d+1} \subset \mathcal{K}_{d}$ and $\sigma$ has depth zero, we deduce that $\mathcal{K}_{d+1} \subseteq \ker(\sigma^{\eta^{d}})$. As $\mathcal{K}_{d+1}$ is a normal subgroup of $\mathcal{K}$, it follows that $\mathcal{K}_{d+1}\subseteq \mathrm{ker}(\mathrm{Ind}_{\mathcal{B}\mathcal{K}_d}^\mathcal{K} \sigma^{\eta^d})$. Therefore, the maximum possible depth of any irreducible component of $\mathrm{Ind}_{\mathcal{B}\mathcal{K}_d}^\mathcal{K} \sigma^{\eta^d}$ is $d$.
	
	We claim that $\mathcal{K}_{d} \nsubseteq \ker(\mathrm{Ind}_{\mathcal{B}\mathcal{K}_d}^{\mathcal{K}} \sigma^{\eta^d})$, and hence at least one irreducible component of $\mathrm{Ind}_{\mathcal{B}\mathcal{K}_d}^{\mathcal{K}} \sigma^{\eta^d}$ has depth $d$. Suppose, for contradiction, that $\mathcal{K}_{d} \subseteq \ker(\mathrm{Ind}_{\mathcal{B}\mathcal{K}_d}^{\mathcal{K}} \sigma^{\eta^d})$. Since $\mathcal{K}_{d}$ is a normal subgroup of $\mathcal{K}$, it would then follow that $\mathcal{K}_{d} \subseteq \ker(\sigma^{\eta^{d}})$. Therefore, to prove that $\mathcal{K}_{d} \nsubseteq \ker(\mathrm{Ind}_{\mathcal{B}\mathcal{K}_d}^{\mathcal{K}} \sigma^{\eta^d})$, it suffices to show that $\mathcal{K}_{d} \nsubseteq \ker(\sigma^{\eta^{d}})$. Let $b = \begin{psmallmatrix} 1 & 0 \\ c & 1 \end{psmallmatrix} \in \mathcal{B}^{op}$, where $c \in \sqrt{\epsilon}\mathcal{O}_F^{\times}$, and let $\chi_{\sigma}$ denote the character of $\sigma$. From Table~\ref{Table2}, we have $\chi_{\sigma}(b) = -1$, hence $\sigma(b) \ne \mathrm{Id}$. 
	Now consider the element \( \begin{psmallmatrix} 1 & 0 \\ c \varpi^d & 1 \end{psmallmatrix} \in \mathcal{K}_d \). Applying \( \sigma^{\eta^d} \) to this element, we obtain
	\[
	\sigma^{\eta^d} \left( \begin{pmatrix} 1 & 0 \\ c \varpi^d & 1 \end{pmatrix} \right) = \sigma\left( \begin{pmatrix} 1 & 0 \\ c & 1 \end{pmatrix} \right) \neq \mathrm{Id}.
	\]
	This shows that the element \( \begin{psmallmatrix} 1 & 0 \\ c \varpi^d & 1 \end{psmallmatrix} \) is not in the kernel of \( \sigma^{\eta^d} \), and we conclude that $\mathcal{K}_d \not\subseteq \ker(\sigma^{\eta^d})$.
\end{proof}

\begin{proposition} \label{D:depth}
	Let $\sigma$ be a cuspidal representation of $\mathcal{K}/\mathcal{K}_{0+}$.  Then for any $d> 0$, the degree of $\mathrm{Ind}_{\mathcal{B}\mathcal{K}_d}^\mathcal{K} \sigma^{\eta^d}$ is $q^{d-1}(q^2-1)$.
\end{proposition}

\begin{proof}
	Since $\sigma$ has degree $q-1$, the degree of $\mathrm{Ind}_{\mathcal{B}\mathcal{K}_d}^\mathcal{K} \sigma^{\eta^d}$ is given by $(q-1)$ times the index $[\mathcal{K} : \mathcal{B}\mathcal{K}_{d}]$.
	Noting the inclusions $\mathcal{K}_{d} \subseteq \mathcal{B}\mathcal{K}_{d} \subseteq \mathcal{K}$, and using group isomorphism theorems we have
	\[
	[\mathcal{K} : \mathcal{B}\mathcal{K}_{d}] = \frac{[\mathcal{K}\colon\mathcal{K}_{d}]}{[\mathcal{B}\mathcal{K}_{d}\colon \mathcal{K}_{d}]}=\frac{[\mathcal{K}\colon\mathcal{K}_{1}][\mathcal{K}_{1}\colon \mathcal{K}_{d}]}{[\mathcal{B}\colon (\mathcal{B}\cap \mathcal{K}_{d})]}=\frac{[\mathcal{K}\colon\mathcal{K}_{1}][\mathcal{K}_{1}\colon \mathcal{K}_{d}]}{[\mathcal{B}\colon (\mathcal{B}\cap \mathcal{K}_{1})][(\mathcal{B}\cap\mathcal{K}_1)\colon (\mathcal{B}\cap \mathcal{K}_{d})]}.
	\]
	
	Since $\mathcal{K}/\mathcal{K}_{1}\cong \mathbb{U}(1,1)(\mathbb{F}_{q})$, and $\mathcal{B}/(\mathcal{B}\cap \mathcal{K}_{1})\cong \mathbb{B}(\mathbb{F}_{q})$, we have
	\[[\mathcal{K}\colon \mathcal{K}_1]=|\mathrm{U}(1,1)(\mathbb{F}_{q})|=q(q-1)(q+1)^2\hspace{1em}\mathrm{and}\hspace{1em}[\mathcal{B}\colon\mathcal{B}\cap \mathcal{K}_{1}]=| \mathrm{B}(\mathbb{F}_{q})|=q(q^2-1).\]	
	The indices $[\mathcal{K}_{1}\colon\mathcal{K}_{d}]$ and 	$[(\mathcal{B}\cap\mathcal{K}_1)\colon (\mathcal{B}\cap \mathcal{K}_{d})]$ are the same as the indices of the corresponding $\mathcal{O}_{F}$-modules in the Lie algebra. Thus we have $[\mathcal{K}_{1}\colon\mathcal{K}_{d}]=q^{4(d-1)}\;\mathrm{and}\;[(\mathcal{B}\cap\mathcal{K}_1)\colon (\mathcal{B}\cap \mathcal{K}_{d})]=q^{3(d-1)}.$ Thus, the degree of $\mathrm{Ind}_{\mathcal{B}\mathcal{K}_d}^\mathcal{K} \sigma^{\eta^d}$ becomes
	\[
	(q-1) \cdot [\mathcal{K} : \mathcal{B}\mathcal{K}_{d}] = (q-1)\frac{q(q-1)(q+1)^{2}q^{4(d-1)}}{q(q^2-1)q^{3(d-1)}} = (q^2-1)q^{d-1}.
	\]	
\end{proof}

The central character of $\sigma$ is the map $\theta \colon \mathfrak{e}^{1}\rightarrow \mathbb{C}^{\times}$ defined by $\sigma(z I)=\theta(z)\mathrm{Id}$ for all $z\in\mathfrak{e}^{1}$. From Table~\ref{Table2}, we infer that $\theta(z)=\alpha(z)\beta(z)$. As $Z\subset \mathcal{K} $, the induced representations $\compactinduction_{\mathcal{K}}^{G}\sigma$ and $\compactinduction_{\mathcal{K}^{\eta}}^{G}\sigma^{\eta}$ have the same central character as $\sigma$.

\begin{lemma}
	Let $d>0$. For $i\in \{1,2\}$,	let $\sigma_i$ be a cuspidal representation of $\mathcal{K}/\mathcal{K}_{0+}$ with central character  $\theta_i$. Then  $\mathrm{Res}_{\mathcal{B}\mathcal{K}_d} \sigma_{1}^{\eta^d}\cong\mathrm{Res}_{\mathcal{B}\mathcal{K}_d} \sigma_{2}^{\eta^d} $ if and only if $\theta_1=\theta_2$.
\end{lemma}

\begin{proof}
	By Lemma~\ref{valueonBKd}, the action of $\sigma_i^{\eta^d}$ on $\mathcal{B}\mathcal{K}_{d}$ is given by the action of $\sigma_{i}$ on $\mathcal{B}^{op}\mathcal{K}_{d}$, and since $\sigma_i$ is trivial on $\mathcal{K}_{d}$, this is determined by $\mathrm{Res}_{\mathcal{B}^{op}}\sigma_i$. It thus suffices to show that $\mathrm{Res}_{\mathcal{B}^{op}} \sigma_{1}\cong\mathrm{Res}_{\mathcal{B}^{op}} \sigma_{2} $ if and only if  $\theta_1=\theta_2$.
	Since these representations factor through the finite group quotient $\mathbb{U}(1,1)(\mathfrak{f}) \cong \mathcal{K}/\mathcal{K}_1$, it suffices to compare their characters.  Write $\mathcal{B}^{op}$ also for the image of $\mathcal{B}^{op}$ in $\mathbb{U}(1,1)(\mathfrak{f})$. Using Table~\ref{Table2}, the character $  \chi_{\sigma_{i}}$ of $\mathrm{Res}_{\mathcal{B}^{op}}\sigma_i$ is given on elements of $\mathcal{B}^{op}$ by 
	$$
	\chi_{\sigma_{i}} \left( \begin{pmatrix}a & 0 \\ c & \overline{a}^{-1}\end{pmatrix} \right) 
	= \begin{cases}
		(q-1)\theta_i(a) & \text{if $a\in\mathfrak{e}^{1}$, $c=0$};\\
		-\theta_i(a) & \text{if $a\in\mathfrak{e}^{1}$, $c \neq 0$};\\
		0 & \text{otherwise}.
	\end{cases}
	$$ 
	We now calculate the intertwining number between $  \chi_{\sigma_{1}}$ and $  \chi_{\sigma_{2}}$ to be
\begin{align*}
		\mathcal{I}(  \chi_{\sigma_{1}},  \chi_{\sigma_{2}})= &\frac{1}{\vert \mathcal{B}^{op} \vert} \sum_{g \in \mathcal{B}^{op}}  \chi_{\sigma_{1}}(g) \overline{  \chi_{\sigma_{2}}(g)} \\
		=&\frac{1}{q(q^2-1)}\left( \sum_{a\in\mathfrak{e}^{1}, c=0} \theta_1(a)\overline{\theta_2(a)}(q-1)^2 + \sum_{a\in\mathfrak{e}^1,c \neq 0} \theta_1(a)\overline{\theta_2(a)}\right)\\
		= &\begin{cases}
			\frac{1}{q(q^2-1)}\left((q+1)(q-1)^2+(q-1)(q+1)\right)=1 & \text{if $\theta_1 = \theta_2$; } \\ 
			0 & \text{otherwise.}
		\end{cases}
	\end{align*}
	Thus we obtain that $  \chi_{\sigma_{1}}=  \chi_{\sigma_{2}}$ if and only if $\theta_1=\theta_2$, as required.
\end{proof}	
\begin{remark}
	For any cuspidal representation $\sigma$ of $\mathcal{K}/\mathcal{K}_{0+}$, the preceding lemma ensures that, for every $d \in \mathbb{Z}_{\geq 0}$, the restriction $\mathrm{Res}_{\mathcal{B}\mathcal{K}_d} \sigma^{\eta^d}$ remains irreducible as a representation of $\mathcal{B}\mathcal{K}_d$.
\end{remark}

The following proposition shows that, for each $d>0$ and each cuspidal representation $\sigma$ of $\mathcal{K}/\mathcal{K}_{0+}$, each of the Mackey components $\mathrm{Ind}_{\mathcal{B}\mathcal{K}_{d}}^{\mathcal{K}}\sigma^{\eta^{d}}$ is independent of the choice of $\sigma$ up to its central character $\theta$.

\begin{proposition}
	Let $d>0$. For $i\in \{1,2\}$,	let $\sigma_i$ be a cuspidal representation of $\mathcal{K}/\mathcal{K}_{0+}$ with central character  $\theta_i$. Then $\mathrm{Ind}_{\mathcal{B}\mathcal{K}_{d}}^{\mathcal{K}}\sigma_{1}^{\eta^{d}}\cong \mathrm{Ind}_{\mathcal{B}\mathcal{K}_{d}}^{\mathcal{K}}\sigma_{2}^{\eta^{d}}$ if and only if $\theta_1=\theta_2$.
\end{proposition}

\begin{proof}	By the preceding lemma, if $\theta_{1} = \theta_{2}$ then $\mathrm{Res}_{\mathcal{B}\mathcal{K}_{d}}\sigma_{1}^{\eta^{d}} \cong \mathrm{Res}_{\mathcal{B}\mathcal{K}_{d}}\sigma_{2}^{\eta^{d}}$, and  by functoriality of induction we have $\mathrm{Ind}_{\mathcal{B}\mathcal{K}_{d}}^{\mathcal{K}}\sigma_{1}^{\eta^{d}}\cong \mathrm{Ind}_{\mathcal{B}\mathcal{K}_{d}}^{\mathcal{K}}\sigma_{2}^{\eta^{d}}$. Conversely, assume that 	$\mathrm{Ind}_{\mathcal{B}\mathcal{K}_{d}}^{\mathcal{K}}\sigma_{1}^{\eta^{d}}\cong \mathrm{Ind}_{\mathcal{B}\mathcal{K}_{d}}^{\mathcal{K}}\sigma_{2}^{\eta^{d}}$. As $Z\subseteq\mathcal{B}\mathcal{K}_{d}$, the induced representations $\mathrm{Ind}_{\mathcal{B}\mathcal{K}_{d}}^{\mathcal{K}}\sigma_{i}^{\eta^{d}}$ have the same central character as that of $\sigma_i$. Since $\mathrm{Ind}_{\mathcal{B}\mathcal{K}_{d}}^{\mathcal{K}}\sigma_{1}^{\eta^{d}}\cong \mathrm{Ind}_{\mathcal{B}\mathcal{K}_{d}}^{\mathcal{K}}\sigma_{2}^{\eta^{d}}$, it follows that $\theta_1=\theta_2$.
\end{proof}

Since $Z\cong E^{1}$, and $E^{1}/(E^{1})^{2}$ has index two, let $\delta$  denote that unique non-trivial quadratic character of $E^{1}$. It can be shown that $\delta$ has depth-zero. We now prove a lemma that will be used in~\S\ref{section5.3}.
\begin{lemma}\label{reductiontodeltafordepthzero}
	Every depth-zero irreducible supercuspidal representation \(\pi\) of \(G\)   can be expressed as $\pi \cong \lambda \otimes \pi'$, where \(\lambda\) is a character of \(G\)  and \(\pi'\) is a depth-zero irreducible supercuspidal representation of \(G\) whose central character is either $\mathbbm{1}$ or $\delta$.
\end{lemma}

\begin{proof}
	Let $\pi$ be a depth-zero supercuspidal representation of $G$ with central character $\theta$. Then there exists a cuspidal representation $\sigma$ of $\mathcal{K}/\mathcal{K}_{0+}$ with central character $\theta$ such that $\pi \cong \compactinduction_{\mathcal{K}}^{G}\sigma$ or $\pi \cong \compactinduction_{\mathcal{K}^{\eta}}^{G}\sigma^{\eta}$. We can write $\theta = \delta^{k}\mu^{2}$ for some $k \in \{0,1\}$. Let $\widetilde{\sigma} = (\phi^{-1}\circ\mathrm{det}) \otimes \sigma$, which is again a cuspidal representation of $\mathcal{K}/\mathcal{K}_{0+}$. Then $\compactinduction_{\mathcal{K}}^{G}\sigma \cong (\phi\circ\mathrm{det}) \otimes \compactinduction_{\mathcal{K}}^{G}\widetilde{\sigma}$ and $\compactinduction_{\mathcal{K}^{\eta}}^{G}\sigma^{\eta} \cong (\phi\circ\mathrm{det}) \otimes \compactinduction_{\mathcal{K}^{\eta}}^{G}\widetilde{\sigma}^{\eta}$, and the result follows.
\end{proof}

\subsection{Irreducibility of the Mackey components}

Let $\sigma$ be a depth-zero irreducible supercuspidal representation of $\mathcal{K}/\mathcal{K}_{0+}$ with central character $\theta$. In the previous section we saw that the restrictions to $\mathcal{K}$ of the depth-zero irreducible supercuspidal representations 
$c  \scalebox{0.9}[0.9]{-} \mathrm{Ind}_{\mathcal{K}}^{G}\sigma$ 
and 
$c  \scalebox{0.9}[0.9]{-} \mathrm{Ind}_{\mathcal{K}^{\eta}}^{G}\sigma^{\eta}$ 
are given by 
\begin{equation*}
	\mathrm{Res}_{\mathcal{K}}\,c  \scalebox{0.9}[0.9]{-} \mathrm{Ind}_{\mathcal{K}}^{G}\sigma
	=
	\sigma \oplus \bigoplus_{d\in 2\mathbb{Z}_{>0}}
	\mathrm{Ind}_{\mathcal{B}\mathcal{K}_{d}}^{\mathcal{K}}\sigma^{\eta^{d}}
	\hspace{1em}\text{and}\hspace{1em}
	\mathrm{Res}_{\mathcal{K}}\,c  \scalebox{0.9}[0.9]{-} \mathrm{Ind}_{\mathcal{K}^{\eta}}^{G}\sigma^{\eta}
	=
	\bigoplus_{d\in 2\mathbb{Z}_{>0}+1}
	\mathrm{Ind}_{\mathcal{B}\mathcal{K}_{d}}^{\mathcal{K}}\sigma^{\eta^{d}}.
\end{equation*}

In this section, we prove that for each \( d \in \mathbb{Z}_{>0} \), each of the Mackey components $\mathrm{Ind}_{\mathcal{B}\mathcal{K}_{d}}^{\mathcal{K}} \sigma^{\eta^{d}}
$ intertwine with the irreducible representations of the same degree constructed in Theorem~\ref{repfromnilpotentelements}, and, as a consequence, we deduce their irreducibility.


\begin{theorem} \label{P:zeropm}
	Let $\sigma$ be a cuspidal representation of $\mathcal{K}/\mathcal{K}_{0+}$ with central character $\theta$. Then for each $d\in\mathbb{Z}_{>0}$, we have
	$ \mathrm{Ind}_{\mathcal{B}\mathcal{K}_d}^\mathcal{K} \sigma^{\eta^d}\cong \mathcal{S}_{d}(X_{\varpi^{-d}}, \theta) $.
	Consequently, each $\mathcal{K}$-representation $\mathrm{Ind}_{\mathcal{B}\mathcal{K}_d}^\mathcal{K} \sigma^{\eta^d}$ is irreducible.
\end{theorem}

\begin{proof}
	Let $\sigma$ be a depth-zero irreducible supercuspidal representation of $G$ with central character $\theta$. Since $\sigma$ has depth zero, for each $d\in\mathbb{Z}_{>0}$, we have that $\theta\mid_{Z\cap\mathcal{J}_{d}}=\mathbbm{1}$. Therefore, by Theorem~\eqref{repfromnilpotentelements}
	$$S_{d}(X_{-\varpi^{-d}}, \theta)=\mathrm{Ind}_{Z\mathcal{U}\mathcal{J}_{d}}^{\mathcal{K}}\Psi_{X_{-\varpi^{-d}}, \theta}$$ 
	is an irreducible representation of $\mathcal{K}$ of depth $d$ and degree $(q^2-1)q^{d-1}$.
	Since both the representations $S_{d}(X_{-\varpi^{-d}}, \theta)$ and $\mathrm{Ind}_{\mathcal{B}\mathcal{K}_d}^\mathcal{K} \sigma^{\eta^d}$ have the same degree and one of them is irreducible, it suffices to show that the space of intertwining operator between them is nonzero. As
	$$\mathrm{Hom}_{\mathcal{K}}(\mathrm{Ind}_{\mathcal{B}\mathcal{K}_d}^\mathcal{K} \sigma^{\eta^d}, \mathcal{S}_{d}(X_{-\varpi^{-d}}, \theta))\cong\bigoplus_{g\in \mathcal{B}\mathcal{K}_{d}\backslash \mathcal{K}/ ZU\mathcal{J}_{d} } \mathrm{Hom}_{\mathcal{B}\mathcal{K}_{d}\cap (ZU\mathcal{J}_{d})^{g}}(\sigma^{\eta^{d}}, \Psi_{X_{-\varpi^{-d}}, \theta}^{g}),$$ 
	it is enough to show that 
	$$\mathrm{Hom}_{\mathcal{B}\mathcal{K}_{d}\cap (ZU\mathcal{J}_{d})^{g}}(\sigma^{\eta^{d}}, \Psi_{X_{-\varpi^{-d}}, \theta})\neq 0.$$
	
	Since these representations have (maximal) depth $d$, they factor through the finite group quotient $(\mathcal{B}\mathcal{K}_d\cap ZU\mathcal{J}_{d})/\mathcal{K}_{d+1}$.  Thus our approach is to evaluate the characters of these representations of finite groups and calculate their intertwining number $\mathcal{I}$.

	Computing the characters is straightforward. The character $\chi_{\eta^{d}}$ of $\sigma^{\eta^{d}}$ on an element $g=(g_{ij})\in \mathcal{B}\mathcal{K}_d$ is given by
	\begin{align*}
		\chi_{\eta^{d}}(g) &= \mathrm{Tr}\left(\sigma^{\eta^d} \left( \begin{pmatrix}g_{11} & g_{12} \\ g_{21} & g_{22}\end{pmatrix}\right)\right) 
		= \mathrm{Tr} \left(\sigma \left( \begin{pmatrix}g_{11} & 0 \\ g_{21}\varpi^{-d} & \overline{g_{11}}^{-1}\end{pmatrix} \right) \right)
	\end{align*}	
	where $g_{11}\in \mathcal{O}_{E}^\times$, and $g_{21}\in \mathfrak{p}_{E}^{d}$. Therefore, we have 
	\begin{align*}
		\chi_{\eta^{d}}(g)	&= \begin{cases}
			(q-1)\theta(g_{11}) & \text{if $g_{11} \in z + \mathfrak{p}_{E}$ for some $z \in E^{1}$, and $g_{21}\in \mathfrak{p}_{E}^{d+1}$;}\\
			-\theta(g_{11})  & \text{if $g_{11}\in z + \mathfrak{p}_{E}$ for some $z\in E^{1}$, and $g_{21} \in \mathfrak{p}_{E}^{d}\setminus \mathfrak{p}_{E}^{d+1}$;}\\
			0 & \textrm{otherwise.}
		\end{cases}
	\end{align*}
	We now provide a formula for the  character $\Psi_{X_{-\varpi^{-d}}, \theta}$. Let $g \in \mathcal{B}\mathcal{K}_d  \cap Z\mathcal{U}\mathcal{J}_{d}$. Then by Lemma~\ref{KcapKd=BKd}, $g_{21}\in\mathfrak{p}_{E}^{d}$, and $g_{11}^{-1}g_{12}\in \sqrt{\epsilon}\mathcal{O}_{F}$. Since $Z\mathcal{U}\subset \mathcal{B}$, we may factor $g$ as $g=t h$ where $t \in Z\mathcal{U}$ and $h = (h_{ij}) \in \mathcal{B}\mathcal{K}_{d}\cap \mathcal{J}_{d}$.  Then there exists $z\in E^{1}$, such that  $g_{11} \equiv z\mod \mathfrak{p}_{E}^{\lceil \frac{d}{2}\rceil}$  and $h_{21} = z^{-1}g_{21}$.  Since $g_{11}\equiv z \mod \mathfrak{p}_{E}^{\lceil \frac{d}{2} \rceil}$ we have $z\in g_{11}(1+ g_{11}^{-1}\mathfrak{p}_{E}^{\lceil \frac{d}{2}\rceil})$ which implies that $z^{-1}\equiv g_{11}^{-1}\mod \mathfrak{p}_{E}^{\lceil \frac{d}{2}\rceil}$  and therefore we have     $h_{21} = z^{-1}g_{21}\equiv g_{11}^{-1}g_{21}\mod\mathfrak{p}_{E}^{\lceil d+1 \rceil}$. 
	As $g_{21}\in\mathfrak{p}_{E}^{d}$, $h_{21}\equiv g_{11}^{-1}g_{21}\mod\mathfrak{p}_{E}^{\lceil d+1 \rceil}$ and $\psi$ is trivial on $\mathfrak{p}_{E}$, it follows that
	\begin{align*}\Psi_{X_{-\varpi^{-d}}, \theta}(g)&=\theta(t)\psi(\mathrm{Tr}(X_{-\varpi^{-d}} (h-I)))=\theta(g_{11})\psi(-\varpi^{-d}h_{21})
		=\theta(g_{11})\psi(-g_{11}^{-1}g_{21}\varpi^{-d}).
	\end{align*}	 
	Computing the intertwining number is technically challenging, but follows easily once we arrange the sum in an appropriate way. Note that if $g_{21}\in\mathfrak{p}_{E}^{d+1}$, then $\psi(-g_{11}^{-1}g_{21}\varpi^{-d})=1$ as $\psi$ is trivial on $\mathfrak{p}_{E}$, and it follows that 
	$$\Psi_{X_{-\varpi^{-d}}, \theta}(g) =\theta(g_{11})\psi(-g_{11}^{-1}g_{21}\varpi^{-d})=\theta(g_{11}).$$ 
	Thus, to evaluate $\Psi_{X_{-\varpi^{-d}}, \theta}$, it suffices to consider its restriction to the following two subsets of $N := (\mathcal{B}\mathcal{K}_{d} \cap Z\mathcal{U}\mathcal{J}_{d})/\mathcal{K}_{d+1}$:
	\begin{align*}
		S_{1}&:=\{g\in N\;|\; g_{11}\in z+\mathfrak{p}_{E}^{\lceil \frac{d}{2}\rceil}\;\text{for some}\;z\in E^{1}\;\text{and}\; g_{21}\in \mathfrak{p}_{E}^{ d+1 }\}.\\
		S_{2}&:=\{g\in N\;|\; g_{11}\in z+\mathfrak{p}_{E}^{\lceil \frac{d}{2}\rceil}\;\text{for some}\;z\in E^{1}\;\text{and}\; g_{21}\in \mathfrak{p}_{E}^{d}\setminus\mathfrak{p}_{E}^{ d+1 }\}.
	\end{align*}
	Note that $N=S_{1}\sqcup S_{2}$.
	Thus we have
	\begin{align*}
		\Psi_{X_{-\varpi^{-d}}, \theta}(g)	&= \begin{cases}
			\theta(g_{11}) & \text{if $g\in S_{1}$},\\
			\theta(g_{11})\psi(-g_{11}^{-1}g_{21}\varpi^{-d})  & \text{if $g\in S_{2}$.}
		\end{cases}
	\end{align*}

	Thus the intertwining number $\mathcal{I}(\chi_{\eta^{d}}, \Psi_{X_{-\varpi^{-d}}, \theta})=\mathrm{dim}_{\mathbb{C}}(\mathrm{Hom}_{\mathcal{B}\mathcal{K}_{d}\cap (ZU\mathcal{J}_{d})^{g}}(\sigma^{\eta^{d}}, \Psi_{X_{-\varpi^{-d}}, \theta}))$ is given by
	\begin{align*}
		\mathcal{I}(\chi_{\eta^{d}}, \Psi_{X_{-\varpi^{-d}}, \theta})&=\frac{1}{\vert N\vert }\sum_{g\in  N}\chi_{\eta^{d}}(g)\overline{\Psi_{X_{-\varpi^{-d}}, \theta}(g)}\\
		&=\frac{1}{|N|}\sum_{g\in S_{1}} (q-1)\theta(g_{11})\overline{\theta(g_{11})}+\frac{1}{|N|}\sum_{g\in S_{2}}(-\theta(g_{11}))\overline{\theta(g_{11})\psi(-g_{11}^{-1}g_{21}\varpi^{-d})}.
	\end{align*}
	
	For convenience, set
	\begin{align*}
		M_{1}:=& \sum_{g\in S_{1}} (q-1)\theta(g_{11})\overline{\theta(g_{11})}\quad\text{and}\quad
		M_{2}:=\sum_{g\in S_{2}}(-\theta(g_{11}))\overline{\theta(g_{11})\psi(-g_{11}^{-1}g_{21}\varpi^{-d})}.
	\end{align*}	
	Then 
	$$\mathcal{I}(\chi_{\eta^{d}}, \Psi_{X_{-\varpi^{-d}}, \theta})=\frac{M_{1}+M_{2}}{\vert N\vert }.$$
	We first compute the second sum $M_2$. Since $\psi$ is unitary, $\psi(-g_{11}^{-1}g_{21}\varpi^{-d})=\overline{\psi(g_{11}^{-1}g_{21}\varpi^{-d})}$ and therefore
	$$M_{2}=\sum_{g\in S_{2}}(-\theta(g_{11}))\overline{\theta(g_{11})}\psi(g_{11}^{-1}g_{21}\varpi^{-d}).$$
	Observe that for each $g_{11}\in z+\mathfrak{p}_{E}^{\lceil \frac{d}{2} \rceil}$, the sum 
	$$\sum_{g_{21}\in \mathfrak{p}_{E}^{d}\backslash \mathfrak{p}_{E}^{d+1}}\psi(g_{11}^{-1}g_{21}\varpi^{-d}\sqrt{\epsilon})=\sum_{y\in \mathfrak{e}\setminus \{0\}}\psi(y).$$	
	Since $\sum_{y\in \mathfrak{e}}\psi(y)=1$, and $\psi(0)=1$, we obtain that
	$$\sum_{g_{21}\in \mathfrak{p}_{E}^{d}\backslash \mathfrak{p}_{E}^{d+1}}\psi(zg_{21}\varpi^{-d}\sqrt{\epsilon})=-1.$$
	Therefore,
	$$M_{2}=\sum_{g\in S_{2}}\theta(g_{11})\overline{\theta(g_{11})}.$$
	Adding $M_{1}$ and $M_{2}$, and applying the orthogonality relations for characters, we obtain
	$$M_{1}+M_{2}=\sum_{g\in N}\theta(g_{11})\overline{\theta(g_{11})}=|N|.$$
	Therefore the intertwining number between the two characters is  
	\[
	\mathcal{I}(\chi_{\eta^{d}}, \Psi_{X_{-\varpi^{-d}}, \theta})
	= \frac{M_{1}+M_{2}}{|N|}
	= \frac{|N|}{|N|}
	= 1.
	\]
	Hence the representations $\mathrm{Ind}_{\mathcal{B}\mathcal{K}_d}^{\mathcal{K}} \sigma^{\eta^{d}}$ and $\mathcal{S}_{d}(X_{-\varpi^{-d}}, \theta)$ intertwine. Since $\mathcal{S}_{d}(X_{-\varpi^{-d}}, \theta)$ is irreducible and has the same  degree as $\mathrm{Ind}_{\mathcal{B}\mathcal{K}_d}^{\mathcal{K}} \sigma^{\eta^{d}}$, we conclude that $\mathrm{Ind}_{\mathcal{B}\mathcal{K}_d}^{\mathcal{K}} \sigma^{\eta^{d}}$ is irreducible and isomorphic to $\mathcal{S}_{d}(X_{-\varpi^{-d}}, \theta)$, which by Theorem~\ref{repfromnilpotentelements} is isomorphic to $\mathcal{S}_{d}(X_{\varpi^{-d}}, \theta)$, proving the theorem.
\end{proof}	

As a corollary, we obtain our main results of this section, and our first set of branching rules.
\begin{corollary}[Branching rules for depth-zero supercuspidal representations] \label{T:depthzero}
	Let $\sigma$ be a cuspidal representation of $\mathcal{K}/\mathcal{K}_{0+}$ with central character $\theta$.  Then the decomposition into irreducible $\mathcal{K}$-representations of the restrictions to $\mathcal{K}$ of the corresponding depth-zero supercuspidal representations of $G$ are given by
	\begin{align*}
		\mathrm{Res}_{\mathcal{K}}c  \scalebox{0.9}[0.9]{-} \mathrm{Ind}_{\mathcal{K}}^{G}\sigma \cong \sigma \oplus \bigoplus_{d \in 2\mathbb{Z}_{\geq 1} } \mathcal{S}_{d}(X_{\varpi^{-d}}, \theta)\quad\text{and}\quad \mathrm{Res}_{\mathcal{K}}c  \scalebox{0.9}[0.9]{-} \mathrm{Ind}_{\mathcal{K^{\eta}}}^{G}\sigma^{\eta} \cong \bigoplus_{d \in 2\mathbb{Z}_{\geq 0}+1} \mathcal{S}_{d}(X_{\varpi^{-d}}, \theta).
	\end{align*}
\end{corollary}

We conclude by noting that every component occurring in the decomposition has distinct depth and degree. Even more, the origin of the representation determines the parity of the depths of the irreducible constituents occurring in the restriction to $\mathcal{K}$, for compact induction from $\mathcal{K}$ yields only even depth components, while compact induction from $\mathcal{K}^{\eta}$ yields only odd depth components. We also observe that if $\sigma_{1}$ and $\sigma_{2}$ are cuspidal representations of $\mathcal{K}/\mathcal{K}_{0+}$ with the same central character, then the corresponding depth-zero supercuspidal representations of $G$ have isomorphic restrictions to $\mathcal{K}$, \emph{i.e.},
\[
\mathrm{Res}_{\mathcal{K}}\,c  \scalebox{0.9}[0.9]{-} \mathrm{Ind}_{\mathcal{K}}^{G}\sigma_1\;\cong\;
\mathrm{Res}_{\mathcal{K}}\,c  \scalebox{0.9}[0.9]{-} \mathrm{Ind}_{\mathcal{K}}^{G}\sigma_2
\qquad\text{and}\qquad 
\mathrm{Res}_{\mathcal{K}}\,c  \scalebox{0.9}[0.9]{-} \mathrm{Ind}_{\mathcal{K^{\eta}}}^{G}\sigma_1^{\eta}\;\cong\;
\mathrm{Res}_{\mathcal{K}}\,c  \scalebox{0.9}[0.9]{-} \mathrm{Ind}_{\mathcal{K^{\eta}}}^{G}\sigma_2^{\eta}.
\]
It is also worth comparing this behaviour with that of depth-zero irreducible supercuspidal representations of $\mathrm{SL}(2,F)$. In the case of $\mathrm{SL}(2,F)$, each Mackey component $\mathrm{Ind}_{\mathcal{B}\mathcal{K}_{d}}^{\mathcal{K}}\sigma^{\eta^{d}}$ decomposes further into two irreducible pieces, whereas for our unitary group each such component remains irreducible.

\section{Branching rules for positive-depth supercuspidal representations of $G$}\label{positive-depth}
In this section, we first construct all positive-depth irreducible supercuspidal representations of \( G \) using a method originally due to J.~K.~Yu, often referred to as the Adler–Fintzen–Yu method~\cite{Adl98, JKKYu2001, Fintzen2021}. Since \( p > 2 \), this construction yields all irreducible supercuspidal representations of \( G \)~\cite[Theorem~8.1]{Fintzen21b}. We then restrict these representations to \( \mathcal{K} \) and describe their decomposition in terms of the irreducible representations of \( \mathcal{K} \) constructed in Theorem~\ref{Rep of K}. We prove that this decomposition is multiplicity-free and is characterized by distinct depth and degree.
\subsection{Construction of the representations}
Our goal in this section is to provide an explicit parametrization of all positive-depth irreducible supercuspidal representations of $G$.

Note that if $\lambda$ is a positive-depth character of $G$ and $\pi_{0}$ is an irreducible depth-zero supercuspidal representation of $G$, then $\lambda \otimes \pi_{0}$ is an irreducible positive-depth supercuspidal representation of $G$. Since we have already described the branching rules for all depth-zero irreducible supercuspidal representations of $G$ in ~\S\ref{chapter 4}, we now turn to the construction of those positive-depth supercuspidal representations not arising in this way.

We begin by recalling the Adler--Fintzen--Yu  method for constructing positive-depth supercuspidal representations of connected reductive $p$-adic groups (together with the twist introduced by Fintzen, Kaletha, and Spice), specializing to the case of our unitary group $G$. We first present the datum required for the construction and then briefly summarize the latter.

\subsubsection{Datum for the construction} \label{datumforconstruction}
There are various equivalent ways to present the datum; here we follow the formulation of \cite{Fintzen2021, FKS23}. The datum for constructing positive-depth supercuspidal representations of a general group $G$ is quite complex, but since our group $G$ has rank one, it can be simplified as follows.

\begin{enumerate}
	\item[(D1)] An anisotropic torus $\mathcal{T}$  of $G$;
	\item[(D2)] The unique point $y \in \mathscr{B}(\mathcal{T}) \subset \mathscr{B}(G)$;
	\item[(D3)] A real number $r > 0$, and we set $s := \frac{r}{2}$;
	\item[(D4)]\label{5} A $G$-generic character $\phi$ of $\mathcal{T}$ of depth $r$ relative to $y$;
	\item[(D4')] A character $\phi'$ of $G$ that is either trivial or of depth $r'>r$.
\end{enumerate}

For the remainder of this subsection, we briefly recall Yu’s  datum for constructing positive-depth supercuspidal representations as formulated in~\cite[\S~2.1]{Fintzen2021}, and explain why, in our setting, it reduces to the simplified form above.

\begin{itemize}
	\item[(a)] The general construction requires a sequence 
	\[
	G = G_{1} \supseteq G_{2} \supset G_{3} \supset \cdots \supset G_{n+1}
	\]
	of twisted Levi subgroups\footnote{For us the symbol $\supset$ always denotes a proper subset inclusion.} of $G$ that split over a tamely ramified extension of $F$, with $Z(G_{n+1})/Z(G)$ anisotropic. Since $G$ has rank one, the only possible sequences are $G \supset \mathcal{T}$ or $G \supseteq G \supset \mathcal{T}$, where $\mathcal{T}$ is an anisotropic torus. We shall see later~\eqref{secondlevisequence} that representations arising from the latter are simply twists (by characters of $G$ as in (D4')) of those from the former, so it suffices to consider the sequence $G \supset \mathcal{T}$, as in (D1).
	
	\item[(b)] In the general setting, one requires a point $y \in \mathscr{B}^{\mathrm{en}}(G_{n+1}) \subset \mathscr{B}^{\mathrm{en}}(G)$ such that the image of $y$ in $\mathscr{B}^{\mathrm{en}}(G_{n+1}^{\mathrm{der}})$ is a vertex.
	For our group $G$, the center is compact, and so $\mathscr{B}^{\mathrm{en}}(G_{\mathrm{der}})$ is the reduced building $\mathscr{B}(G)$. Moreover, since $\mathcal{T}$ is compact its building (enlarged or reduced) is a single point, which is thus a vertex of $\mathscr{B}(\mathcal{T})$, justifying (D2).

	\item[(c)] The general datum requires a decreasing sequence of real numbers $r_1 > r_2 > \cdots > r_n > 0$, depending on the length of the chain of twisted Levi subgroups. In our case, the two possibilities are a single parameter $r > 0$, or a pair $r' > r > 0$, corresponding to the sequences noted in (a). So if the sequence is $ G\supset \mathcal{T}$, we have a character $\phi$ of $\mathcal{T}$ of depth $r$, and if the sequence is $G\supseteq G \supset \mathcal{T}$, then in  addition to $\phi$, we also have a character $\phi'$ of $G$ of depth $r'$.

	\item[(d)] One also requires, for each $1 \leq i \leq n$, a character $\phi_i$ of $G_{i+1}$ of depth $r_i$, trivial on $(G_{i+1})_{y, r_i+}$. If $G_{i}\neq G_{i+1}$, we also require $\phi_i$ to be  $G_i$-generic of depth $r_i$ relative to $y$. In particular for our case, for either sequence, we need a character of $\mathcal{T}$, which we will denote by $\phi$, that is $G$-generic of depth $r>0$, which is the condition (D4); and for the sequence $G\supseteq G\supset\mathcal{T}$ we additionally need a character $\phi'$ of $G$ of depth $r'>r$, which is the optional condition (D4').
	
	\item[(e)]\label{condition on depth-zero part of datum} Another ingredient is an irreducible representation $\sigma$ of $(G_{n+1})_{[y]}$ that is trivial on $(G_{n+1})_{y,0+}$ and is a cuspidal representation of
	\[
	(G_{n+1})_{y,0}/(G_{n+1})_{y,0+},
	\] 
	where $[y]$ denotes the image of $y$ in the reduced building $\mathscr{B}(G_{n+1})$. In our case, $(G_{n+1})_{[y]}=\mathcal{T}$, since $\mathscr{B}^{\mathrm{en}}(\mathcal{T})=\mathscr{B}^{\mathrm{red}}(\mathcal{T})=\{y\}$, and  $\sigma$ is a cuspidal representation of $\mathcal{T}/\mathcal{T}_{0+}$. Since this group is abelian, $\sigma$ is one-dimensional, and without loss of generality we may absorb it into $\phi$ by replacing $\phi$ with $\sigma \otimes \phi$~\cite[Theorem~6.7]{HM08}.
\end{itemize}

The datum in particular defines a compact open subgroup $K'$ and a representation $\rho$ thereof such that $\compactinduction_{K'}^{G}\rho$ is an irreducible supercuspidal representation as we will summarize in the next section. In our case, the inducing subgroup is simply $K'=G_{y,s}\mathcal{T}$.

\begin{remark}
	Fintzen--Kaletha--Spice~\cite[Definition 4.1.10]{FKS23}  introduced an additional sign character $\varepsilon$ of the inducing subgroup $K'$. However, in our case, since $G_{n+1}=\mathcal{T}$ is just a torus, $\varepsilon$ is a sign character  of $G_{y,s}\mathcal{T}$, trivial on $G_{y,s}\mathcal{T}_{0+}$. Since  $G_{y,s}\mathcal{T} / G_{y,s}\mathcal{T}_{0+} \cong \mathcal{T}/\mathcal{T}_{0+}$, $\varepsilon$ is simply a depth-zero character of $\mathcal{T}$. As in (e), we may replace $\phi$ with $\varepsilon \otimes \phi$. We note that Yu’s original construction did not include $\varepsilon$; this correction was introduced in~\cite{FKS23} to make Yu's proof work -- his proof relied on a typo in a published reference. Fintzen also showed in \cite{Fintzen2021} that Yu’s original construction is valid even without $\varepsilon$, and provided a different proof of Yu's main theorem. For our purposes, the presence or absence of $\varepsilon$ does not affect the outcome, though it is crucial in broader contexts such as the local Langlands correspondence~\cite{kaletha2021supercuspidallpackets} and character theory~\cite{Lor2018}.
\end{remark}	

Note that we have already classified all anisotropic tori $\mathcal{T}$ of $G$ up to conjugacy in Proposition~\ref{conjugacyclassesofanisotropictoriofG}, with corresponding vertices $y$ listed in Table~\ref{table:1}. We have also provided a description of $G$-generic characters of $\mathcal{T}$ in~\ref{genericcharacters}. Having assembled the necessary ingredients, we now proceed with the construction.

\subsubsection{The construction}Given a quadruple $(\mathcal{T}, y, r, \phi)$ as in Section~\ref{datumforconstruction}, the idea of the construction is to extend $\phi$ to a uniquely determined depth $r$ representation $\rho=\rho(\mathcal{T}, y, r, \phi)$ of the compact open subgroup $G_{y,s}\mathcal{T}$, whose compact induction to $G$ is irreducible, and hence supercuspidal. It proceeds as follows. Recall that we denoted by $e$ the Moy--Prasad isomorphisms of abelian groups
\[
\mathfrak{t}_{s+}/\mathfrak{t}_{r+}\;\longrightarrow\;\mathcal{T}_{s+}/\mathcal{T}_{r+},
\qquad
\mathfrak{g}_{x,s+}/\mathfrak{g}_{x,r+}\;\longrightarrow\; G_{x,s+}/G_{x,r+}.
\] 
Since the character $\phi$ of $\mathcal{T}$ has depth $r$, its restriction to $\mathcal{T}_{s+}$ factors through $\mathcal{T}_{s+}/\mathcal{T}_{r+}$. As established in Section~\ref{genericcharacters}, every such character is represented by an element of $\mathfrak{t}_{-r}$; that is, there exists 
$\Gamma_{u,v} =\begin{psmallmatrix}u\sqrt{\epsilon}& v\sqrt{\epsilon}\\ v\sqrt{\epsilon}& u\sqrt{\epsilon}\end{psmallmatrix}\in \mathfrak{t}_{-r} \subset \mathfrak{g}$ such that
\begin{equation}\label{E:Aphi}
	\phi(t)=\Psi\!\left(\mathrm{Tr}\left(\Gamma_{u,v}\,e^{-1}(t)\right)\right)
	\qquad\text{for all } t\in \mathcal{T}_{s+}.
\end{equation}
Moreover, as explained in that section, the image of $\Gamma_{u,v}$ in $\mathfrak{t}_{-r}/\mathfrak{t}_{-s}$ is uniquely determined by this relation, and the genericity of $\phi$ forces 
$\nu(v\gamma_1)=-r-y$. The element $\Gamma:=\Gamma_{u,v}$ also defines a character $\Psi_{\Gamma}$ of $G_{y,s+}/G_{y,r+}$ given by
\[
\Psi_{\Gamma}(g) = \Psi(\mathrm{Tr}(\Gamma e^{-1}(g))) \quad  \text{for all $g \in G_{y,s+}$}.
\]
Since $\phi$ and $\Psi_{\Gamma}$ agree on the intersection $\mathcal{T}_{s+}$ of their domains, together they define a unique character $\hat{\phi}$ of $G_{y,s+}\mathcal{T}$, given by
\begin{equation}\label{descriptionofphihat}
	\hat{\phi}(gt) = \Psi_{\Gamma}(g)\,\phi(t) \qquad \forall g \in G_{y,s+},\; \forall t\in \mathcal{T}.
\end{equation}

The following lemma is a summary of the consequences of the key technical step of Yu's construction applied to our case.

\begin{lemma}\label{isotypicpropertyofrho}
	If $\mathcal{T}$ is ramified, or if $\mathcal{T}$ is unramified and $r$ is odd, then $G_{y,s+} = G_{y,s}$. In these two cases, $\rho = \hat{\phi}$ is a one-dimensional representation of $G_{y,s}\mathcal{T}$ of depth $r$.
	
	If $\mathcal{T}$ is unramified and $r$ is even, then there exists an irreducible representation $\rho$ of $G_{y,s}\mathcal{T}$ of dimension $q$ such that $\rho\!\mid_{G_{y,s+}\mathcal{T}_{s}}$ is $\hat{\phi}$-isotypic, $\rho\!\mid_{Z\mathcal{T}_{0+}}$ is $\phi$-isotypic, and $\mathrm{Res}_{G_{y, s+}}\rho$ is $\Psi_{\Gamma}$-isotypic.
\end{lemma}

\begin{proof}[Idea of proof.] If $G_{y,s+} = G_{y,s}$, then we already have a description of $\rho=\hat{\phi}$ given by~\eqref{descriptionofphihat}. If $G_{y,s+}\neq G_{y,s}$, then $\rho$ is constructed using the Weil--Heisenberg representation. 	A detailed description of the construction, together with the fact that $	\dim(\rho)=\sqrt{\lvert G_{y, s}/G_{y,s+}\mathcal{T}_{s} \rvert}$ and $	\rho\!\mid_{G_{y,s+}\mathcal{T}_{s}}$ is $\hat{\phi}$-isotypic, is given in~\cite[Section~2.5]{Fintzen2021}. Since $G_{y,s+}\leq G_{y,s+}\mathcal{T}_{s}\leq G_{y,s}\mathcal{T}_{s}=G_{y,s}$, using group isomorphism theorems we have
	\[[G_{y,s}\colon G_{y,s+}\mathcal{T}_{s}]=\frac{[G_{y,s}\colon G_{y,s+}]}{[G_{y,s+}\mathcal{T}_{s}\colon G_{y,s+}]}=\frac{[G_{y,s}\colon G_{y,s+}]}{[\mathcal{T}_{s}\colon \mathcal{T}_{s+}]}.\]
	As $s>0$, by properties of Moy-Prasad filtration subgroups, we have $G_{y,s}/ G_{y,s+}\cong \mathfrak{g}_{y,s}/\mathfrak{g}_{y,s+}$ and $\mathcal{T}_{s}/\mathcal{T}_{s+}\cong\mathfrak{t}_{s}/\mathfrak{t}_{s+}$.  Therefore 
	\[
	[ G_{y, s}\colon G_{y,s+}\mathcal{T}_{s}] =\frac{[\mathfrak{g}_{y,s}\colon \mathfrak{g}_{y,s+}]}{[\mathfrak{t}_{s}\colon \mathfrak{t}_{s+}]}=\frac{q^{4}}{q^2}= q^{2},
	\]
	and hence $\dim(\rho)=q$. The remaining properties, namely that 
	$\rho\!\mid_{Z\mathcal{T}_{0+}}$ is $\phi$-isotypic and that 
	$\mathrm{Res}_{G_{y,s+}}\rho$ is $\Psi_{\Gamma}$-isotypic, are established 
	in~\cite[Lemma~3]{MN13}.
\end{proof}

We write $\rho=\rho(\mathcal{T}, y, r, \phi)$ for the representation produced by Lemma~\ref{isotypicpropertyofrho}. Then by~\cite[Theorem~3.1]{Fintzen2021} the compactly induced representation
\[
\pi_{\rho}:=c \!\!\;\operatorname{-Ind}_{G_{y,s}\mathcal{T}}^{G}\rho
\]
is an irreducible supercuspidal representation of $G$ of depth $r$.

Now suppose our datum includes a character $\phi'$ of $G$ of depth $r'>r$, as in (D4'). Then the representation $\rho'=\rho(\mathcal{T}, y, r, \phi, r', \phi')$ is just $\phi'\otimes \rho$. Since $\phi'$ is a character of $G$, by applying Lemma~~\ref{positive depth lemma} we have
\begin{equation}\label{secondlevisequence}
	c \!\!\;\operatorname{-Ind}_{G_{y,s}\mathcal{T}}^{G}(\phi'\otimes\rho)\;\cong\; \phi'\otimes \Big(c \!\!\;\operatorname{-Ind}_{G_{y,s}\mathcal{T}}^{G}\rho\Big),
\end{equation}
and therefore $c\!\!\;\operatorname{-Ind}_{G_{y,s}\mathcal{T}}^{G}(\phi'\otimes\rho)$ is also an irreducible supercuspidal representation of $G$, this time of depth $r'$. Hence the simplification in (D1).
\begin{theorem}{\cite{Fintzen21b}}\label{positivedepth}
	All positive-depth irreducible supercuspidal representations of $G$ have the form 
	$$\lambda\otimes c\scalebox{0.9}[0.9]{-} \mathrm{Ind}_{G_{y,s}\mathcal{T}}^{G}\rho$$
	for some character $\lambda$ of $G$,  or $\lambda \otimes \pi_{0}$, where $\lambda$ is a positive-depth character of $G$, and $\pi_{0}$ is a depth-zero irreducible supercuspidal representation of $G$.
\end{theorem}

We denote the simplified datum \((\mathcal{T}, y, r, \phi)\) arising from (D1)–(D4) by \(\Sigma\). 
When (D4') is also included, we denote the extended datum \((\mathcal{T}, y, r, \phi, r', \phi')\) by \(\Sigma'\). A natural question is: if we tensor $\rho$ (or $\pi_{\rho}$) by a character of $G$ of depth $r'>r$, then it is the representation arising from the datum $(\mathcal{T}, y, r, \phi, r', \phi')$, but what happens if we tensor by a character of $G$ of depth $\leq r$? How does it arise from a datum?

\begin{lemma}\label{rhohasdepthzerocentralcharacter}
	Let $\rho = \rho(\mathcal{T}, y, r, \phi)$, let $\lambda$ be a character of $G$ of depth $\leq r$, and let $\widetilde{\phi}$ be a character of $\mathcal{T}$ such that $\phi = \lambda\!\mid_{\mathcal{T}} \otimes \widetilde{\phi}.$ Then
	\begin{enumerate}
		\item $\widetilde{\phi}$ is a $G$-generic character of $\mathcal{T}$ of depth $r$; and 
		\item if $\widetilde{\rho} := \rho(\mathcal{T}, y, r, \widetilde{\phi})$, then with $s=\frac{r}{2}$ we have 
		$c\!\operatorname{-Ind}_{G_{y,s}\mathcal{T}}^{G} \widetilde{\rho}
		\;\cong\;
		\lambda \otimes \left(c\!\operatorname{-Ind}_{G_{y,s}\mathcal{T}}^{G} \rho\right).$
	\end{enumerate}
\end{lemma}

\begin{proof}
	We leverage the work of Hakim and Murnaghan~\cite{HM08}, who showed when two data produce equivalent representations. Let $\phi_{G}$ be a character of $G$ of depth $r_{1}>r$. Then 
	$\Sigma' := (\mathcal{T}, y, r, \phi, r_1, \phi_{G})$ is a well-defined input for the construction of a supercuspidal representation. Consequently, by~\eqref{secondlevisequence},
	\[
	\phi_{G} \otimes c\!\operatorname{-Ind}_{G_{y,s}\mathcal{T}}^{G}\rho
	\]
	is an irreducible supercuspidal representation of $G$ of depth $r_{1}$, where $\rho=\rho(\mathcal{T}, y, r, \phi)$.

	To prove the lemma, consider $\dot{\Sigma} = (\mathcal{T}, y, r, \widetilde{\phi}, r_1, \lambda\phi_{G}).$  To verify that it is a valid input, we have to prove that $\widetilde{\phi}$ is a $G$-generic character of $\mathcal{T}$ of depth $r$. To do so, we prove that $\dot{\Sigma}$ is a refactorization of $\Sigma'$ in the sense of \cite[Definition~4.19]{HM08}. 
	To verify that $\dot{\Sigma}$ is a refactorization of $\Sigma'$, we must check conditions (F0), (F1), and (F2) of Definition~4.19 in~\cite{HM08}, 
	which in our setting reduce to the following two conditions:
	\begin{enumerate}
		\item[(F0)] If $\phi_{G} = \mathbbm{1}$, then $\lambda \phi_{G} = \mathbbm{1}$.
		\item[(F1)] $\widetilde{\phi}\!\mid_{\mathcal{T}_{0+}} = (\phi  \lambda^{-1})\!\mid_{\mathcal{T}_{0+}}$ and 
		$(\lambda  \phi_{G})\!\mid_{G_{y,r+}} = \phi_{G}\!\mid_{G_{y,r+}}$.
	\end{enumerate}
	
	Condition (F0) is vacuously satisfied because $\phi_{G}$ has depth $r_{1}>0$ and is therefore nontrivial.  
	For (F1), note that $\phi = \lambda\mid_{\mathcal{T}}\otimes \widetilde{\phi}$, so the first equality holds immediately.  
	Moreover, since $\lambda$ has depth $\leq r$, it is trivial on $G_{y,r+}$, which implies the second equality as well.

	In~\cite[Proposition~4.24]{HM08}, Hakim and Murnaghan prove that if $\dot{\Sigma}$ is a refactorizaton of $\Sigma'$, then  $\dot{\Sigma}$ is also a well-defined input, and the representations 
	\[
	\rho(\Sigma):=\phi_{G}\!\mid_{G_{y,s}\mathcal{T}}\otimes \rho \qquad\mathrm{and}\qquad \rho(\dot{\Sigma}):=(\lambda \phi_{G})\!\mid_{G_{y,s}\mathcal{T}}\otimes \widetilde{\rho}
	\]	 
	are equivalent. As a consequence, we obtain an isomorphism
	\[
	\phi_{G} \otimes \left(c\!\operatorname{-Ind}_{G_{y,s}\mathcal{T}}^{G} \rho\right)
	\;\cong\;
	(\lambda \phi_{G}) \otimes \left(c\!\operatorname{-Ind}_{G_{y,s}\mathcal{T}}^{G} \widetilde{\rho}\right),
	\]
	of supercuspidal representations. 
	Finally, canceling $\phi_{G}$ on both sides yields the desired result.
\end{proof}

We can now give a further simplification to our data, one that will be crucial in \S\ref{section5.3}.

\begin{proposition}\label{cor:rhohasdepthzerocentralcharacter}	Let $\rho=\rho(\mathcal{T}, y, r, \phi)$, and set $s := \frac{r}{2}$.  
	Then there exists a character $\lambda$ of $G$ such that
	\[
	c\scalebox{0.9}[0.9]{-}\mathrm{Ind}_{G_{y,s}\mathcal{T}}^{G}\rho
	\;\cong\;
	\lambda \otimes c\scalebox{0.9}[0.9]{-}\mathrm{Ind}_{G_{y,s}\mathcal{T}}^{G}\widetilde{\rho},
	\]
	where $\widetilde{\rho}=\rho(\mathcal{T}, y, r, \phi')$ is chosen so that 
	$c\scalebox{0.9}[0.9]{-}\mathrm{Ind}_{G_{y,s}\mathcal{T}}^{G}\widetilde{\rho}$ is an irreducible supercuspidal representation of depth $r$ with depth-zero central character.
\end{proposition}

\begin{proof}
	Let $\theta$ denote the central character of $c\scalebox{0.9}[0.9]{-} \mathrm{Ind}_{G_{y,s}\mathcal{T}}^{G}\rho$. Then we may write $\theta=\delta^{k}\mu^{2}$ for some $k\in\{0,1\}$ and some character $\mu$ of $E^{1}$, where $\delta$ is the non-trivial quadratic character of $E^{1}$. Consider the character $\mu\circ \mathrm{det}$ of $G$. Since $c\!\operatorname{-Ind}_{G_{y,s}\mathcal{T}}^{G}\rho$ has depth $r$, its central character $\theta$ has depth at most $r$. Note that $\mu \circ \det$  has same depth as that of $\theta$. Thus $\mu \circ \det$ is a character of $G$ of depth $\leq r$.
	
	Define $\phi' := (\mu^{-1}\!\circ \det)\otimes \phi.$ Then  $\phi'\mid_{Z}$ is $\delta^{k}$, and we may write $\phi=(\mu\circ \mathrm{det})\otimes \phi'$.  Applying Lemma~\ref{rhohasdepthzerocentralcharacter} yields that $\phi'$ is a $G$-generic character of $\mathcal{T}$ of depth $r$, and 
	$$ c\scalebox{0.9}[0.9]{-} \mathrm{Ind}_{G_{y,s}\mathcal{T}}^{G}\rho\cong (\mu\circ \mathrm{det}) \otimes  c\scalebox{0.9}[0.9]{-} \mathrm{Ind}_{G_{y,s}\mathcal{T}}^{G}\widetilde{\rho} $$
	where $\widetilde{\rho}=\rho(\mathcal{T}, y, r, \phi')$.	Since $\phi'\mid_{Z}$ is $\delta^{k}$,  it follows that the central character of $c\!\operatorname{-Ind}_{G_{y,s}\mathcal{T}}^{G}\widetilde{\rho}$ is  $\delta^{k}$, which has depth zero for either value of $k$. 
	This completes the proof.
\end{proof}
The following is a direct consequence of Theorem~\ref{positivedepth} and Corollary~\ref{cor:rhohasdepthzerocentralcharacter}. This corollary will be used in~$\S\ref{section5.3}$.

\begin{corollary}\label{reductiontodepthzero3}
	Every irreducible supercuspidal representation \(\pi\) of \(G\) of depth \(r>0\) can be expressed as $\pi \cong \lambda \otimes \pi'$,  
	where \(\lambda\) is a character of \(G\)  and \(\pi'\) is an irreducible supercuspidal representation of \(G\) of depth \(r\) whose central character is either $\mathbbm{1}$ or $\delta$.
\end{corollary}

\subsection{Restriction to $\mathcal{K}$}

By Theorem~\ref{positivedepth}, we know that the positive-depth irreducible supercuspidal representations of $G$ have the form 
$$\lambda\otimes c\scalebox{0.9}[0.9]{-} \mathrm{Ind}_{G_{y,s}\mathcal{T}}^{G}\rho$$
for some character $\lambda$ of $G$,  or $\lambda \otimes \pi_{0}$, where $\lambda$ is a positive-depth character of $G$, and $\pi_{0}$ is a depth-zero irreducible supercuspidal representation of $G$.
Since the decomposition of depth-zero representations of $G$ was already described in \S\ref{depthzero rep of G}, and for a character $\lambda$ of $G$
$$\mathrm{Res}_{\mathcal{K}}\left(\lambda\otimes c\scalebox{0.9}[0.9]{-} \mathrm{Ind}_{G_{y,s}\mathcal{T}}^{G}\rho\right)\cong \lambda\mid_{\mathcal{K}}\otimes \mathrm{Res}_{\mathcal{K}}\left(c\scalebox{0.9}[0.9]{-} \mathrm{Ind}_{G_{y,s}\mathcal{T}}^{G}\rho\right),$$
it suffices to describe the branching for $\pi_{\rho}=c\scalebox{0.9}[0.9]{-} \mathrm{Ind}_{G_{y,s}\mathcal{T}}^{G}\rho$, where  $\rho=\rho(\mathcal{T}, y, r, \phi)$.  In this section, we do not need to assume that the depth of the central character of $\pi_{\rho}$ is zero. By Mackey theory we have
\begin{align*}
	\mathrm{Res}_{\mathcal{K}}^{G}\,\pi_\rho 
	&\;\cong\; 
	\bigoplus_{g\in \mathcal{K}\backslash G/G_{y,s}\mathcal{T}}
	c\!\!\;\operatorname{-Ind}_{\mathcal{K}\cap (G_{y,s}\mathcal{T})^{g}}^{\mathcal{K}} \rho^{g}.
\end{align*}
For each double coset representative $g\in \mathcal{K}\backslash G/G_{y,s}\mathcal{T}$, the depth of the Mackey component  $c\!\!\;\operatorname{-Ind}_{\mathcal{K}\cap (G_{y,s}\mathcal{T})^{g}}^{\mathcal{K}} \rho^{g}$ depends on $\mathcal{T}$ (via the associated point $y$), on $g$, and on $r$. We denote this depth by $d = d(\mathcal{T}, g, r)$.  

\begin{theorem}\label{thm4}
	For each double coset representative $g \in \mathcal{K}\backslash G/G_{y,s}\mathcal{T}$, the corresponding Mackey component
	\[
	c\!\!\;\operatorname{-Ind}_{\mathcal{K}\cap (G_{y,s}\mathcal{T})^{g}}^{\mathcal{K}} \rho^{g}
	\]
	is an irreducible representation of $\mathcal{K}$.
\end{theorem}

The proof of Theorem~\ref{thm4} proceeds in several steps. In~\S\ref{doublecosetrepresentativespositivedepth} we give an explicit description of double coset representatives for the space $\mathcal{K}\backslash G / G_{y,s}\mathcal{T}$. Using these, we compute the degrees of the corresponding Mackey components in~\S\ref{sectiondegreeofmackeycomponent}. Finally, in \S\ref{proofofmaintheorem}, we show that these Mackey components intertwine with the irreducible representations of the same degree constructed in Theorem~\ref{Rep of K}, thereby establishing their irreducibility.

\subsubsection{Double coset representatives}\label{doublecosetrepresentativespositivedepth}
Our goal in this section is to explicitly compute a set of double coset representatives for the double coset space $\mathcal{K}\backslash G/G_{y,s}\mathcal{T}$ for each of the anisotropic tori listed in Table~\ref{table:1}, together with their corresponding point \(y\) and $s>0$.
Note that, except for \(\mathcal{T}_{\varpi^{-1}, \varpi}\), all such tori are contained in \(\mathcal{K}\).
In contrast, the torus \(\mathcal{T}_{\varpi^{-1}, \varpi}=\mathcal{T}_{1,1}^{\eta}\) lies inside the conjugate subgroup \(\mathcal{K}^{\eta}\), where $\eta=\begin{psmallmatrix}1&0\\[2pt]0&\varpi\end{psmallmatrix}.
$
By Lemma~\ref{doublecosetrepresentativeofKGK} a set of representatives for \(\mathcal{K}\backslash G/\mathcal{K}\) (and likewise for \(\mathcal{K}^{\eta}\backslash G/\mathcal{K}^{\eta}\)) is $\left\{\;\alpha^{t}\;|;\ t\ge 0\;\right\}$, where $\alpha^{t}:=\begin{psmallmatrix}\varpi^{-t}&0\\0&\varpi^{t}\end{psmallmatrix}$.

When \(\mathcal{T}\subset \mathcal{K}\), we also have \(G_{y,s}\mathcal{T}\subset \mathcal{K}\); and  for the torus corresponding to \(y=1\), we have \(\mathcal{T}=\mathcal{T}_{1,1}^{\eta}\subset G_{y,s}\mathcal{T}_{1,1}^{\eta}\subset  \mathcal{K}^{\eta}\).
In both situations, applying Lemma~\ref{db1} shows that each double coset in \(\mathcal{K}\backslash G /\mathcal{T}G_{y,s}\) has a representative of the form \(\alpha^{t}\beta\), with \(t\ge 0\) and \(\beta\) a representative of $\mathcal{K}\cap\mathcal{K}^{\alpha^{-t}}\backslash \mathcal{K}/G_{y,s}\mathcal{T}
\;\text{or}\;
\mathcal{K}^{\eta}\cap(\mathcal{K}^{\eta})^{\alpha^{-t}}\backslash \mathcal{K}^{\eta}/G_{y,s}\mathcal{T},$ respectively. Therefore, we begin by computing sets of double coset representatives for these latter double coset spaces.

Note that $\alpha^{-t}=\varpi^{t}I\eta^{-2t}$, therefore $\mathcal{K}\cap\mathcal{K}^{\alpha^{-t}}=\mathcal{K}\cap \mathcal{K}^{\eta^{-2t}}=(\mathcal{K}\cap \mathcal{K}^{\eta^{2t}})^{\eta^{-2t}}$. By Lemma~\ref{KcapKd=BKd}, $(\mathcal{K}\cap \mathcal{K}^{\eta^{2t}})=\mathcal{B}\mathcal{K}_{2t}$ and by Lemma~\ref{valueonBKd} we have $(\mathcal{B}\mathcal{K}_{2t})^{\eta^{-2t}}=\mathcal{B}^{op}\mathcal{K}_{2t}$. Hence \begin{equation}\label{intersectionforalpha}\mathcal{K}\cap\mathcal{K}^{\alpha^{-t}}=\mathcal{B}^{op}\mathcal{K}_{2t}.\end{equation}
Since \(\mathcal{B}^{op}\subseteq \mathcal{K}\cap\mathcal{K}^{\alpha^{-t}}\) and \(\mathcal{T}\subset G_{y,s}\mathcal{T}\), each double coset is a union of smaller double cosets, namely, $(\mathcal{K}\cap\mathcal{K}^{\alpha^{-t}})\backslash \mathcal{K}/G_{y,s}\mathcal{T}
\ \text{is a union of double cosets from }\
\mathcal{B}^{op}\backslash \mathcal{K}/\mathcal{T}.$ Similarly, as $\eta$ and $\alpha^{-t}$ commute, we have \((\mathcal{B}^{op})^{\eta}\subseteq \mathcal{K}^{\eta}\cap(\mathcal{K}^{\eta})^{\alpha^{-t}}\) and $\mathcal{T}_{1,1}^{\eta}\subset (G_{0,s}\mathcal{T}_{1,1})^{\eta}=G_{y,s}\mathcal{T}$. Thus $(\mathcal{K}^{\eta}\cap(\mathcal{K}^{\eta})^{\alpha^{-t}})\backslash \mathcal{K}^{\eta}/G_{y,s}\mathcal{T}$ 
is a union of double cosets from 
$(\mathcal{B}^{op})^{\eta}\backslash \mathcal{K}^{\eta}/\mathcal{T}_{1,1}^{\eta}.$

\begin{proposition}\label{db 6}
	Let $\mathcal{T}$ be one of the tori listed in Table~\ref{table:1}, with $\mathcal{T} \subset \mathcal{K}$. Then 
	\begin{equation*}
		\mathcal{K} =
		\begin{cases}
			\mathcal{B}^{op}\mathcal{T}, & \text{if } \mathcal{T}=\mathcal{T}_{1,1}\\
			\mathcal{B}^{op}\mathcal{T}\sqcup \mathcal{B}^{op}\mathrm{w}\mathcal{T}, & \text{if } \mathcal{T} \text{ is ramified}
		\end{cases}
	\end{equation*}		
	where $\mathrm{w}=\begin{psmallmatrix}0&1\\1&0\end{psmallmatrix}$. Similarly, $\mathcal{K}^{\eta}=(\mathcal{B}^{op})^{\eta}(\mathcal{T}_{1,1})^{\eta}$.		\end{proposition}

\begin{proof}Assume that $\mathcal{T}_{\gamma_1, \gamma_2}\subset \mathcal{K}$ and let $k:=\begin{psmallmatrix}
		x & y \\
		z & w
	\end{psmallmatrix} \in \mathcal{K}$. If $y=0$, then $k$ belongs to the identity coset $\mathcal{B}^{op}\mathcal{T}_{\gamma_1, \gamma_2}$. If $y \neq 0$ but $y\in \mathfrak{p}_{E}$, then Lemma \ref{Kproperty} yields $x\in\mathcal{O}_{E}^{\times}$. Since $x^{-1} = \frac{\overline{x}}{x\overline{x}}$ and $\overline{x}z\in\sqrt{\epsilon}F$, we conclude that $x^{-1}z\in \sqrt{\epsilon}F$. We also have $\overline{x}w+\overline{z}y=1$, implying that $w=\overline{x}^{-1}-\overline{x}^{-1}\overline{z}y=\overline{x}^{-1}+x^{-1}zy$. Hence we can express $k$ as
	\[
	\begin{pmatrix}
		x & y \\
		z & w
	\end{pmatrix} = \begin{pmatrix}
		x & 0 \\
		z & \overline{x}^{-1}
	\end{pmatrix} \begin{pmatrix}
		1 & yx^{-1} \\
		0 & 1
	\end{pmatrix}.
	\]
	Observe that if $\begin{psmallmatrix}x&y\\z&w\end{psmallmatrix}\in\mathcal{K}$, then $ \begin{psmallmatrix}
		x & 0 \\
		z & \overline{x}^{-1}
	\end{psmallmatrix}\in\mathcal{K}$, thus $\begin{psmallmatrix}
		1 & yx^{-1} \\
		0 & 1
	\end{psmallmatrix}$ has to be an element of $\mathcal{K}$, and therefore it follows that $x^{-1}y\in\sqrt{\epsilon}\mathfrak{p}_{F}$. Then $(1-(yx^{-1})^{2}\gamma_1^{-1}\gamma_2)\in\mathcal{O}_{F}^{\times}$. Choose $a\in\mathcal{O}_{E}^{\times}$ such that $a\overline{a}=1-(yx^{-1})^2\gamma_1^{-1}\gamma_2$. Note that such a choice of $a$ is possible because the norm map $N_{E/F}$ maps $\mathcal{O}_{E}^{\times}$ surjectively onto $\mathcal{O}_{F}^{\times}$. We then have
	$$\begin{pmatrix}
		x & y \\
		z & w
	\end{pmatrix} = \begin{pmatrix}
		x & 0 \\
		z & \overline{x}^{-1}
	\end{pmatrix}\begin{pmatrix}a& 0\\ -\overline{a}^{-1}yx^{-1}\gamma_{1}^{-1}\gamma_2& \overline{a}^{-1}\end{pmatrix}\begin{pmatrix}a^{-1}&yx^{-1}a^{-1}\gamma_1^{-1}\gamma_1\\ yx^{-1}a^{-1}\gamma_1^{-1}\gamma_2&a^{-1}\end{pmatrix}\in \mathcal{B}^{op}\mathcal{T} .$$
	Lastly, if  $y\in \mathcal{O}_{E}^{\times}$, then  we can express $k$ as $\begin{psmallmatrix}
		x & y \\
		z & w
	\end{psmallmatrix} = \begin{psmallmatrix}
		y & 0 \\
		w & \overline{y}^{-1}
	\end{psmallmatrix} \begin{psmallmatrix}
		xy^{-1} & 1 \\
		1 & 0
	\end{psmallmatrix},$ and $xy^{-1}\in\sqrt{\epsilon}\mathcal{O}_{F}$. If $xy^{-1}\in \sqrt{\epsilon}\mathcal{O}_{F}^{\times}$, then $(\gamma_1^{-1}\gamma_2-(xy^{-1})^2)\in\mathcal{O}_{F}^{\times}$. Choose $a\in\mathcal{O}_{E}^{\times}$ such that $a\overline{a}=\gamma_1^{-1}\gamma_2-(xy^{-1})^2$. We then have
	$$\begin{pmatrix}
		x & y \\
		z & w
	\end{pmatrix} = \begin{pmatrix}
		y & 0 \\
		w & \overline{y}^{-1}
	\end{pmatrix} \begin{pmatrix}a& 0\\ -\overline{a}^{-1}xy^{-1}\gamma_{1}^{-1}\gamma_2& \overline{a}^{-1}\end{pmatrix}\begin{pmatrix}a^{-1}xy^{-1}&a^{-1}\gamma_1^{-1}\gamma_1\\ a^{-1}\gamma_1^{-1}\gamma_2&a^{-1}xy^{-1}\end{pmatrix}\in \mathcal{B}^{op}\mathcal{T} .$$
	Finally, let $xy^{-1}\in \sqrt{\epsilon}\mathfrak{p}_{F}$. Then $1-\gamma_1\gamma_2^{-1}(xy^{-1})^2\in \mathcal{O}_{F}^{\times}$. Choose $a\in\mathcal{O}_{E}^{\times}$ such that $a\overline{a}=1-\gamma_1\gamma_2^{-1}(xy^{-1})^2$. We then have
	$$\begin{pmatrix}
		x & y \\
		z & w
	\end{pmatrix} = \begin{pmatrix}
		y & 0 \\
		w & \overline{y}^{-1}
	\end{pmatrix} \begin{pmatrix}a& 0\\ -\overline{a}^{-1}xy^{-1}\gamma_{1}\gamma_2^{-1}& \overline{a}^{-1}\end{pmatrix}\begin{pmatrix}0& 1\\ 1&0\end{pmatrix}\begin{pmatrix}
		a^{-1}& (a\gamma_2)^{-1}xy^{-1}\gamma_1\\ (a\gamma_2)^{-1}xy^{-1}\gamma_2 & a^{-1}
	\end{pmatrix}\in \mathcal{B}^{op}\mathrm{w}\mathcal{T} .$$
	It remains to show that when $\mathcal{T}$ is unramified, then $\mathrm{w}$ belongs to the identity double coset;  and when $\mathcal{T}$ is ramified, then $\mathcal{B}^{op}\mathcal{T}\neq \mathcal{B}^{op}\mathrm{w}\mathcal{T}$. Indeed, from direct computation one can show that $\mathrm{w}\in \mathcal{B}^{op}\mathcal{T}$ if and only if there exists some $a\in\mathcal{O}_{E}^{\times}$ such that  $a\overline{a}=\gamma_1^{-1}\gamma_2$. Since $N_{E/F}$ maps $\mathcal{O}_{E}^{\times}$ surjectively onto $\mathcal{O}_{F}^{\times}$, such a choice of $a$ is possible only when $\gamma_1^{-1}\gamma_2\in\mathcal{O}_{F}^{\times}$, and this happens only when $\mathcal{T}$ is unramified.  Finally, since the groups $\mathcal{B}^{op}$, $\mathcal{K}$, and $\mathcal{T}_{1,1}$ are conjugated simultaneously, the map $x \mapsto \eta x \eta^{-1}$ induces a bijection between the double coset spaces $\mathcal{B}^{op} \backslash \mathcal{K} / \mathcal{T}_{1,1}$ and $(\mathcal{B}^{op})^{\eta} \backslash \mathcal{K}^{\eta} / \mathcal{T}_{1,1}^{\eta}$. Therefore, $\mathcal{K}^{\eta}=(\mathcal{B}^{op})^{\eta}(\mathcal{T}_{1,1})^{\eta}$.
\end{proof}

\begin{lemma}
	Let $\mathcal{T}$ be one of the tori listed in Table~\ref{table:1}, and let $y=\mathcal{A}(G, \mathcal{T})$. Then 
	\begin{equation*}
		\mathcal{K} =
		\begin{cases}
			(\mathcal{K}\cap\mathcal{K}^{\alpha^{-t}})G_{y,s}\mathcal{T}, & \text{if } \mathcal{T}=\mathcal{T}_{1,1}\\
			(\mathcal{K}\cap\mathcal{K}^{\alpha^{-t})}G_{y,s}\mathcal{T}\sqcup (\mathcal{K}\cap\mathcal{K}^{\alpha^{-t}})\mathrm{w}G_{y,s}\mathcal{T}	, & \text{if } \mathcal{T} \text{ is ramified}.
		\end{cases}
	\end{equation*}		
	Similarly, $\mathcal{K}^{\eta}=	(\mathcal{K}^{\eta}\cap(\mathcal{K}^{\eta})^{\alpha^{-t}})G_{y,s}\mathcal{T}$ if $\mathcal{T}=\mathcal{T}_{1,1}^{\eta}$.	
\end{lemma}	

\begin{proof}
	As noted earlier $(\mathcal{K}\cap\mathcal{K}^{\alpha^{-t}})\backslash \mathcal{K}/G_{y,s}\mathcal{T}
	\quad \left(\text{resp. }(\mathcal{K}^{\eta}\cap(\mathcal{K}^{\eta})^{\alpha^{-t}})\backslash \mathcal{K}^{\eta}/G_{y,s}\mathcal{T}\right)
	$ is a union of double cosets from 
	$
	\mathcal{B}^{op}\backslash \mathcal{K}/\mathcal{T}
	\quad \left(\text{resp. }(\mathcal{B}^{op})^{\eta}\backslash \mathcal{K}^{\eta}/\mathcal{T}_{1,1}^{\eta}\right)
	$. When $\mathcal{T}$ is unramified, by Proposition~\ref{db 6} we have $\mathcal{K}=\mathcal{B}^{op}\mathcal{T}
	\quad \left(\text{resp. }\mathcal{K}^{\eta}=(\mathcal{B}^{op})^{\eta}\mathcal{T}_{1,1}^{\eta}\right),
	$
	from which we conclude that, when $\mathcal{T}=\mathcal{T}_{1,1}$,
	$
	\mathcal{K}=(\mathcal{K}\cap\mathcal{K}^{\alpha^{-t}})G_{y,s}\mathcal{T},
	$
	and, when $\mathcal{T}=\mathcal{T}_{1,1}^{\eta}$,
	$
	\mathcal{K}^{\eta}=(\mathcal{K}^{\eta}\cap(\mathcal{K}^{\eta})^{\alpha^{-t}})G_{y,s}\mathcal{T}.
	$

	It remains to show that $I$ and $\mathrm{w}$ represent distinct double cosets of $\mathcal{K}\cap \mathcal{K}^{\alpha^{-t}}\backslash \mathcal{K}/G_{y,s}\mathcal{T}$  when $\mathcal{T}$ is ramified. Let $\mathcal{T}$ be a ramified torus with corresponding point $y=\frac{1}{2}$.
	Since $(\mathcal{K}\cap \mathcal{K}^{\alpha^{-t}})=\mathcal{B}^{op}\mathcal{K}_{2t}$~\eqref{intersectionforalpha}, we have $(\mathcal{K}\cap \mathcal{K}^{\alpha^{-t}})G_{y,s}\mathcal{T}\subset \mathcal{B}^{op}\mathcal{K}_{2t}G_{y,s}\mathcal{T}$. Since the $(1,1)$-entry of any element of $\mathcal{B}^{op}\mathcal{K}_{2t}G_{y,s}\mathcal{T}$ will always be an element of $\mathcal{O}_{E}^{\times}$, we deduce that $\mathrm{w}\notin (\mathcal{K}\cap \mathcal{K}^{\alpha^{-t}})G_{y,s}\mathcal{T}\subset \mathcal{B}^{op}\mathcal{K}_{2t}G_{y,s}\mathcal{T} $, and hence $I$ and $\mathrm{w}$ must represent distinct double cosets.
\end{proof}

\begin{theorem}\label{maindoublecoset}
	A set of representatives for the double coset  space $\mathcal{K}\backslash G/\mathcal{T}G_{y,s}$ is given by
	$$   M(\mathcal{T}):= \begin{cases}
		
		\left\{\alpha^{t}\mid t\geq0\right\}& \mathrm{if}\;\mathcal{T}\;\mathrm{is\; unramified}\\
		\{I, \alpha^{t}, \alpha^{t}\mathrm{w}\mid t>0\}& \mathrm{if}\;\mathcal{T}\;\mathrm{is\; ramified}
	\end{cases}$$
	where $\alpha^{t}=\begin{psmallmatrix}
		\varpi^{-t}&0 \\ 0&\varpi^{t}
	\end{psmallmatrix}$, and $\mathrm{w}=\begin{psmallmatrix}
		0&1\\1&0
	\end{psmallmatrix}$.
\end{theorem}

\subsubsection{The degree of the Mackey components}\label{sectiondegreeofmackeycomponent}
For each $g \in M(\mathcal{T})$, we now compute the degrees of the corresponding Mackey components
\[
\mathrm{Ind}_{\mathcal{K}\cap (G_{y,s}\mathcal{T})^{g}}^{\mathcal{K}} \rho^{g}.
\]
Recall that for $r>0$ and $s=\tfrac{r}{2}$, the representation $\rho=\rho(\mathcal{T}, y, r, \phi)$ of $G_{y,s}\mathcal{T}$ is an irreducible representation of depth $r$ that has degree $1$ when $G_{y,s}=G_{y,s+}$, and degree $q$ otherwise. Thus, for all $g\in M(\mathcal{T})$, the degree of the corresponding Mackey component is
\[
\deg\!\left(\mathrm{Ind}_{\mathcal{K}\cap (G_{y,s}\mathcal{T})^{g}}^{\mathcal{K}} \rho^{g}\right)
= \deg(\rho)\,[\mathcal{K} : \mathcal{K}\cap (G_{y,s}\mathcal{T})^{g}].
\]

Thus, the computation reduces to evaluating the index $[\mathcal{K} : \mathcal{K}\cap (G_{y,s}\mathcal{T})^{g}]$ for each $g\in M(\mathcal{T})$.

\begin{proposition}\label{factorisation of intersection}
	Let $g \in M(\mathcal{T})$. Then $\mathcal{K}\cap (G_{y,s}\mathcal{T})^{g} \;=\; (\mathcal{K}\cap G_{y,s}^{g})(\mathcal{K}\cap \mathcal{T}^{g}).$
\end{proposition}

\begin{proof}
	Since $G_{y,s}\trianglelefteq G_{y}$ and $\mathcal{T}\subseteq G_{y}$, the torus $\mathcal{T}$ normalizes $G_{y,s}$. In particular, $\mathcal{K}\cap \mathcal{T}^{g}$ normalizes $\mathcal{K}\cap G_{y,s}^{g}$. Hence the product $(\mathcal{K}\cap G_{y,s}^{g})(\mathcal{K}\cap \mathcal{T}^{g})$ is a subgroup of $\mathcal{K}$.  Let $g=\alpha^{t}$. If $t=0$, then $g=I$. Since for all $s>0$ we have $G_{y,s}\subseteq \mathcal{K}$, we have $(\mathcal{K}\cap G_{y,s}\mathcal{T})=G_{y,s}(\mathcal{K}\cap \mathcal{T})$, and there is nothing to prove. Assume $t>0$, and write $k=hu \in G_{y,s}\mathcal{T}$ with $h\in G_{y,s}$ and $u\in \mathcal{T}$. Then $k^{\alpha^{t}} = h^{\alpha^{t}} u^{\alpha^{t}} \in G_{y,s}^{\alpha^{t}} \mathcal{T}^{\alpha^{t}}.$ Explicitly we have
	\[
	k^{\alpha^{t}}=\begin{pmatrix}
		h_{11}u_{11}+h_{12}u_{21} & (h_{11}u_{12}+h_{12}u_{22})\varpi^{-2t} \\
		(h_{21}u_{11}+h_{22}u_{21})\varpi^{2t} & h_{21}u_{12}+h_{22}u_{22}
	\end{pmatrix}.
	\]
	For $k^{\alpha^{t}}$ to lie in $\mathcal{K}$, the following conditions must be satisfied. Since $t>0$, the $(2,1)$ entry lies in $\mathfrak{p}_{E}$; hence, by Lemma~\ref{Kproperty}, the $(1,1)$ and $(2,2)$ entries must lie in $\mathcal{O}_{E}^{\times}$, and the $(1,2)$ entry must lie in $\mathcal{O}_{E}$. Since $h_{21}\in\mathfrak{p}_{E}^{\lceil s+y \rceil}\subseteq \mathfrak{p}_{E}$ and $h_{22}\in 1+\mathfrak{p}_{E}^{\lceil s \rceil}$,  solving the condition on the $(2,2)$ entry of $k^{\alpha^{t}}$ gives $u_{22}\in\mathcal{O}_{E}^{\times}$.
	
	Next, consider the $(1,2)$ entry of $hu$. Since $h_{11},u_{22}\in\mathcal{O}_{E}^{\times}$, $\nu(h_{11}u_{12}+h_{12}u_{22}) \;\geq\; \min\{\nu(u_{12}),\nu(h_{12})\}.$ If $\nu(u_{12})\neq \nu(h_{12})$, then in order to have $k^{\alpha^{t}}\in \mathcal{K}$ it is necessary that $\nu(u_{12}),\nu(h_{12})\geq 2t$. This implies that both $h^{\alpha^{t}}$ and $u^{\alpha^{t}}$ lie in $\mathcal{K}$ whenever $k^{\alpha^{t}}\in \mathcal{K}$. Thus we have the desired factorization in this case.
	
	On the other hand, if $\nu(h_{12})=\nu(u_{12})\geq \lceil s-y \rceil$, more care is required. Since $u\in \mathcal{T}$, we have $u_{11}=u_{22}$ and $u_{21}=u_{12}\gamma_{1}^{-1}\gamma_{2}$. Moreover, we have $u_{11}\overline{u_{11}}+u_{12}\overline{u_{12}}\gamma_{1}^{-1}\gamma_{2}=1$, which implies that $u_{11}\overline{u_{11}} \equiv 1 \pmod{\mathfrak{p}_{E}^{2\lceil s-y \rceil}}$. Since for all $n\geq 1$ the norm map $N_{E/F}\colon 1+\mathfrak{p}_{E}^{n}\to 1+\mathfrak{p}_{F}^{n}$ is surjective when $E/F$ is a unramified quadratic extension,  there exists $c\in 1+\mathfrak{p}_{E}^{2\lceil s-y \rceil}$ with $N_{E/F}(u_{11})=N_{E/F}(c)$. As $N_{E/F}(u_{11}c^{-1})=1$, $u_{11}c^{-1}\in E^{1}$. Therefore, we can factor
	\[
	\begin{pmatrix}
		u_{11}&u_{12}\\ u_{12}\gamma_{1}^{-1}\gamma_2& u_{11}
	\end{pmatrix}
	=\begin{pmatrix}
		c&c u_{11}^{-1}u_{12}\\ c u_{11}^{-1}u_{12}\gamma_{1}^{-1}\gamma_2&c
	\end{pmatrix}
	\begin{pmatrix}
		u_{11}c^{-1}&0\\ 0&u_{11}c^{-1}
	\end{pmatrix}\in G_{y,s}Z.
	\]
	Thus $u=\dot{h}z$ for some $\dot{h}\in G_{y,s}$ and $z\in Z$. Hence $k=h\dot{h}z=\tilde{h}z$ with $\tilde{h}\in G_{y,s}$.  Since $z^{\alpha^{t}}=z\in \mathcal{K}$,  $k^{\alpha^{t}}\in \mathcal{K}$ if and only if $\tilde{h}^{\alpha^{t}}\in \mathcal{K}$. This yields the desired factorization. The case $g=\alpha^{t}\mathrm{w}$ is analogous, and can be proved by replacing $G_{y,s}$ with $G_{y,s}^{\mathrm{w}}$ and $\lceil s-y \rceil$ with $\lceil s+y \rceil$.
\end{proof}

Let $\mathcal{T} = \mathcal{T}_{\gamma_1,\gamma_2}$, and take $y=\mathcal{A}(G, \mathcal{T})$ as in Table~\ref{table:1}.  
Before proceeding, for each $g\in M(\mathcal{T})$ and $s>0$, we describe explicitly the subgroups $\mathcal{K}\cap \mathcal{T}^{g}$ and $\mathcal{K}\cap G_{y,s}^{g}$.

For $g\in M(\mathcal{T})$, we define
$$\delta(g)=\begin{cases} 2t-y& \mathrm{if}\;g=\alpha^{t}\mathrm{w}\\
	2t+y&\mathrm{otherwise}.
\end{cases}$$

Let $k \in\mathcal{K}$. Then $k\in\mathcal{T}^{g}$ if and only if $g^{-1}kg\in\mathcal{T}$. A direct computation shows that
\begin{equation}\label{intersection1}
	\mathcal{K}\cap\mathcal{T}^{g}=   \begin{cases}\left\{\begin{pmatrix}a&b\\ b\gamma_1^{-1}\gamma_2\varpi^{4t}&a\end{pmatrix}\middle\vert a,b\in\mathcal{O}_{E}\right\}\cap G&\mathrm{if}\;g=\alpha^{t},\\
		&\\
		\left\{\begin{pmatrix}a&b\\ b\gamma_1\gamma_2^{-1}\varpi^{4t}&a\end{pmatrix}\middle\vert a,b\in\mathcal{O}_{E}\right\}\cap G&\mathrm{if}\;g=\alpha^{t}\mathrm{w}.
	\end{cases}
\end{equation}	
Similarly, $k\in G_{y,s}^{g}$ if and only if $g^{-1}kg\in G_{y,s}$, and we have
\begin{equation}\label{intersection2}
	\mathcal{K}\cap G_{y,s}^{g}=\mathcal{K}\cap\left\{\begin{pmatrix}1+\mathfrak{p}_{E}^{\lceil s \rceil}&\mathfrak{p}_{E}^{\lceil s-\delta(g) \rceil}\\ \mathfrak{p}_{E}^{\lceil s+\delta(g) \rceil}&1+\mathfrak{p}_{E}^{\lceil s\rceil}\end{pmatrix}\right\}=\left\{\begin{pmatrix}1+\mathfrak{p}_{E}^{\lceil s \rceil}&\mathfrak{p}_{E}^{M}\\ \mathfrak{p}_{E}^{\lceil s+\delta(g) \rceil}&1+\mathfrak{p}_{E}^{\lceil s\rceil}\end{pmatrix}\right\}\cap \mathcal{K}
\end{equation}
where $M=\mathrm{max}\{0, \lceil s-\delta(g)\rceil\}$.

\begin{lemma}\label{Z-isotypic}
	Let $g$ be $\alpha^{t}\;(t\in\mathbb{Z}_{\geq 0})\;\mathrm{or}\;\alpha^{t}\mathrm{w}\;(t\in\mathbb{Z}_{>0})$, and let $y=\mathcal{A}(G, \mathcal{T})$. Then for $t>0$ or $y\neq 0$, we have $(\mathcal{K}\cap\mathcal{T}^{g})=Z(\mathcal{K}\cap \mathcal{T}_{0+}^{g})$.
\end{lemma}	

\begin{proof}
	Since $y$ is the point associated to $\mathcal{T}$, $\mathcal{T}_{0+}=\mathcal{T}\cap G_{y, 0+}$.
	The computations required for the proof of this lemma are similar to those of~\eqref{intersection1}.	Let $k=\begin{psmallmatrix} a&b\\ c&d\end{psmallmatrix}$; then $k\in \mathcal{T}_{0+}^{g}$ if and only if $g^{-1}kg\in\mathcal{T}_{0+}$. Suppose first that $g=\alpha^{t}$. Then $\alpha^{-t}k\alpha^{t}=\begin{psmallmatrix} a& b\varpi^{2t}\\ c\varpi^{-2t}&d\end{psmallmatrix}$ and we see that
	$$	\mathcal{K}\cap\mathcal{T}_{0+}^{\alpha^{t}}=\left\{\begin{pmatrix}a&b\\ b\gamma_1^{-1}\gamma_2\varpi^{4t}&a\end{pmatrix}\mid a\in1+\mathfrak{p}_{E},b\in\mathcal{O}_{E}\right\}\cap G.$$
	Similarly, when $g=\alpha^{t}\mathrm{w}$, we have $y=\frac{1}{2}$ and $t>0$, therefore 
	$$	\mathcal{K}\cap\mathcal{T}_{0+}^{\alpha^{t}\mathrm{w}}=\left\{\begin{pmatrix}a&b\\ b\gamma_1\gamma_2^{-1}\varpi^{4t}&a\end{pmatrix}\mid a\in1+\mathfrak{p}_{E},b\in\mathcal{O}_{E}\right\}\cap G.$$
	Now let  $h=\begin{pmatrix} a&b\\ b\gamma_{1}^{-1}\gamma_{2}\varpi^{4t}&a\end{pmatrix}\in\mathcal{K}\cap\mathcal{T}^{\alpha^{t}}$. Then $b\gamma_1^{-1}\gamma_2\varpi^{4t}\in\mathfrak{p}_{F}$, and $a\overline{a}+b\overline{b}\varpi^{4t}\gamma_{1}^{-1}\gamma_2=1$. It follows that $a\overline{a}\in 1+\mathfrak{p}_{F}$. Since $N_{E/F}\colon 1+\mathfrak{p}_{E}\rightarrow 1+\mathfrak{p}_{F}$ is surjective, there exists $c\in1+\mathfrak{p}_{E}$ such that $N_{E/F}(a)=N_{E/F}(c)$, and we have the  decomposition
	$$\begin{pmatrix} a&b\\ b\gamma_{1}^{-1}\gamma_{2}\varpi^{4t}&a\end{pmatrix}=\begin{pmatrix}ac^{-1}&0\\ 0& ac^{-1}\end{pmatrix}\begin{pmatrix} c&ca^{-1}b\\ ca^{-1}b\gamma_{1}^{-1}\gamma_{2}\varpi^{4t}&c\end{pmatrix}\in Z(\mathcal{K}\cap\mathcal{T}_{0+}^{\alpha^{t}}).$$
	Similarly, when $h=\begin{pmatrix} a&b\\ b\gamma_{1}\gamma_{2}^{-1}\varpi^{4t}&a\end{pmatrix}\in\mathcal{K}\cap\mathcal{T}^{\alpha^{t}\mathrm{w}}$, we have $b\gamma_1\gamma_2^{-1}\varpi^{4t}\in\mathfrak{p}_{F}$,  $a\overline{a}\in1+\mathfrak{p}_{F}$, and 
	$$\begin{pmatrix} a&b\\ b\gamma_{1}\gamma_{2}^{-1}\varpi^{4t}&a\end{pmatrix}=\begin{pmatrix}ac^{-1}&0\\ 0& ac^{-1}\end{pmatrix}\begin{pmatrix} c&ca^{-1}b\\ ca^{-1}b\gamma_{1}\gamma_{2}^{-1}\varpi^{4t}&c\end{pmatrix}\in Z(\mathcal{K}\cap\mathcal{T}_{0+}^{\alpha^{t}\mathrm{w}}),$$
	as was required.
\end{proof}

\begin{proposition}	\label{degree of Mackey component}
	Let $\rho=\rho(\mathcal{T}, y, r, \phi)$ and let $g\in M(\mathcal{T})$. Then
	$$\mathrm{deg}\left(\mathrm{Ind}_{\mathcal{K}\cap (G_{y,s}\mathcal{T})^{g}}^{\mathcal{K}} \rho^{g}\right)=\begin{cases}
		(q-1)q^r&\mathrm{when}\;t=0\;\mathrm{and}\;y=0,\\
		(q^{2}-1)q^{d-1}&\mathrm{otherwise}
	\end{cases}$$
	
	where $d=r+\delta(g)$.
\end{proposition}	

\begin{proof}
	We want to compute the index $[\mathcal{K} : \mathcal{K}\cap (G_{y,s}\mathcal{T})^{g}]$. By Proposition~\ref{factorisation of intersection}, this equals $[\mathcal{K}\colon(\mathcal{K}\cap G_{y,s}^{g}) (\mathcal{K}\cap\mathcal{T}^{g})]$. From the description of these intersections~\eqref{intersection2}, $G_{y,s+2t+1}$ is contained in  $(\mathcal{K}\cap G_{y,s}^{g})(\mathcal{K}\cap\mathcal{T}^{g})$ and therefore  in $\mathcal{K}$. Hence
	$$[\mathcal{K}\colon(\mathcal{K}\cap G_{y,s}^{g}) (\mathcal{K}\cap\mathcal{T}^{g})]=\frac{[\mathcal{K}\colon G_{y, s+2t+1}]}{[(\mathcal{K}\cap G_{y,s}^{g})(\mathcal{K}\cap\mathcal{T}^{g})\colon G_{y, s+2t+1}]}.$$
	
	We also have $G_{y, s+2t+1}\leq (\mathcal{K}\cap G_{y,s}^{g})\leq (\mathcal{K}\cap G_{y,s}^{g})(\mathcal{K}\cap\mathcal{T}^{g})$, so by the third isomorphism theorem for groups we have
	\begin{align*}
		&[(\mathcal{K}\cap G_{y,s}^{g})(\mathcal{K}\cap\mathcal{T}^{g})\colon G_{y, s+2t+1}]
		=[(\mathcal{K}\cap G_{y,s}^{g})(\mathcal{K}\cap\mathcal{T}^{g})\colon \mathcal{K}\cap G_{y,s}^{g}][\mathcal{K}\cap G_{y,s}^{g}\colon G_{y, s+2t+1}].
	\end{align*}	
	Since $(\mathcal{K}\cap\mathcal{T}^{g})$ normalizes $(\mathcal{K}\cap G_{y,s}^{g})$, the second isomorphism theorem yields 
	\begin{align*}
		[(\mathcal{K}\cap\mathcal{T}^{g})(\mathcal{K}\cap G_{y,s}^{g})\colon \mathcal{K}\cap G_{y,s}^{g}]
		&=[(\mathcal{K}\cap\mathcal{T}^{g})\colon(\mathcal{K}\cap\mathcal{T}^{g})\cap(\mathcal{K}\cap G_{y,s}^{g})]
		=[(\mathcal{K}\cap\mathcal{T}^{g})\colon (\mathcal{K}\cap \mathcal{T}^{g}\cap G_{y,s}^{g})].
	\end{align*}
	Hence,
	\begin{align*}
		[\mathcal{K}\colon (\mathcal{K}\cap G_{y,s}^{g})(\mathcal{K}\cap\mathcal{T}^{g})]
		=&\frac{[\mathcal{K}\colon G_{y, s+2t+1}]}{[(\mathcal{K}\cap\mathcal{T}^{g})\colon (\mathcal{K}\cap \mathcal{T}^{g}\cap G_{y,s}^{g})][(\mathcal{K}\cap G_{y,s}^{g})\colon G_{y, s+2t+1}]}.
	\end{align*}
	Since $G_{y, s+2t+1}\leq \mathcal{K}_{0+}\leq \mathcal{K}$, the third isomorphism theorem  gives
	\begin{align*}
		[\mathcal{K}\colon G_{y, s+2t+1}]=&[\mathcal{K}\colon \mathcal{K}_{0+}][\mathcal{K}_{0+}\colon G_{y, s+2t+1}]=\mid \mathbb{U}(1,1)(\mathfrak{f})\mid [\mathcal{K}_{0+}\colon G_{y, s+2t+1}].
	\end{align*}
	The index $[\mathcal{K}_{0+}\colon G_{y, s+2t+1}]$	equals the index of the corresponding $\mathcal{O}_{F}$-modules in the Lie algebra,
	which is
	$$[\mathcal{K}_{0+}\colon G_{y, s+2t+1}]=q^{2\lceil s+2t \rceil}q^{\lceil s+2t-y \rceil}q^{\lceil s+2t+y \rceil},$$
	so
	$$[\mathcal{K}\colon G_{y, s+2t+1}]=q(q-1)(q+1)^{2}q^{2\lceil s+2t \rceil}q^{\lceil s+2t-y \rceil}q^{\lceil s+2t+y \rceil}
	=(q-1)(q+1)^{2}q^{2\lceil s \rceil +8t+1}q^{\lceil s-y\rceil}q^{\lceil s+y \rceil}.$$
	From~\eqref{intersection2}, the quotient $(\mathcal{K}\cap G_{y,s}^{g})/ (\mathcal{K}_{0+}\cap G_{y,s}^{g})$ is a unipotent subgroup of $G/ (\mathcal{K}_{0+}\cap G_{y,s}^{g})$, and therefore is in bijection with the corresponding $\mathcal{O}_{F}$-module in the Lie algebra. This allows us to compute
	\begin{align*}
		[(\mathcal{K}\cap G_{y,s}^{g})\colon G_{y, s+2t+1}]=&q^{2(\lceil s+2t+1\rceil-\lceil s \rceil)}q^{\lceil s+2t+1-y\rceil-M}q^{\lceil s+2t+1+y \rceil-\lceil s+\delta(g) \rceil}\\
		=&q^{8t+4-M-\lceil s+\delta(g) \rceil}q^{\lceil s-y \rceil}q^{\lceil s+y \rceil}
	\end{align*}
	where $M=\mathrm{max}\{0, \lceil s-\delta(g)\rceil\}$.
	
	When $t=0$ and $y=0$, let $\mathbb{T}(\mathfrak{f})=\mathcal{T}/\mathcal{T}_{0+}$, which is an anisotropic torus over $\mathfrak{f}$. We then have  $[(\mathcal{K}\cap\mathcal{T}^{g})\colon (\mathcal{K}\cap \mathcal{T}^{g}\cap G_{y,s}^{g})]=[\mathcal{T}\colon \mathcal{T}_{s}]=\mid \mathbb{T}(\mathfrak{f}) \mid[\mathcal{T}_{1}\colon \mathcal{T}_{s}]=(q+1)^{2}q^{2\lceil s \rceil -2}$. Hence
	$$[\mathcal{K}\colon (\mathcal{K}\cap(\mathcal{T}G_{y,s})^{g})]=\frac{(q-1)(q+1)^{2}q^{4\lceil s \rceil +1}}{q^{4}(q+1)^{2}q^{2\lceil s \rceil-2}}=(q-1)q^{2\lceil s \rceil-1}.$$
	Since $\mathrm{deg}(\rho)=q$  exactly when $r=2\lceil s \rceil$, and $\mathrm{deg}(\rho)=1$ when $r=2\lceil s \rceil -1$, the total degree is $(q-1)q^{2\lceil s \rceil-1}\mathrm{deg}(\rho)=(q-1)q^{r}$.
	
	In all other cases, that is, when $t>0$ or $y\neq 0$, since $\mathcal{T}\cap G_{y,s}=\mathcal{T}_s$, we have $\mathcal{T}^{g}\cap G_{y,s}^{g}=\mathcal{T}_{s}^{g}$. Therefore,
	$$[(\mathcal{K}\cap\mathcal{T}^{g})\colon (\mathcal{K}\cap \mathcal{T}^{g}\cap G_{y,s}^{g})]=[(\mathcal{K}\cap \mathcal{T}^{g})\colon (\mathcal{K}\cap \mathcal{T}_{s}^{g})]=[(\mathcal{K}\cap \mathcal{T}^{g})\colon (\mathcal{K}\cap \mathcal{T}_{0+}^{g})][(\mathcal{K}\cap \mathcal{T}_{0+}^{g})\colon (\mathcal{K}\cap \mathcal{T}_{s}^{g})].$$ By Lemma~\ref{Z-isotypic}, $[(\mathcal{K}\cap\mathcal{T}^{g})\colon (\mathcal{K}\cap\mathcal{T}_{0+}^{g})]=[Z \colon Z_{0+}]=\;[E^{1} :  (1+\mathfrak{p}_{E})\cap E^{1}]$, which equals $q+1 $ by Proposition~\ref{E1capE11+pE}. The second factor is again the index of the corresponding $\mathcal{O}_{F}$-module in the Lie algebra, which is   $q^{\lceil s \rceil -1}q^{M}$.
	Hence
	\begin{align*}
		\mathrm{deg}\left(\mathrm{Ind}_{\mathcal{K}\cap (G_{y,s}\mathcal{T})^{g}}^{\mathcal{K}} \rho^{g}\right)&=(q^{2}-1)q^{\lceil s \rceil+\lceil s+\delta(g)\rceil-2}\mathrm{deg}(\rho).
	\end{align*}
	When $r$ is an even integer, $\mathrm{deg}(\rho)=q$, and $\delta(g)$ is an integer, so the expression simplifies to $(q^2-1)q^{2\lceil s \rceil+\delta(g)-2}q=(q^{2}-1)q^{d-1}$.  Otherwise, we have $\mathrm{deg}(\rho)=1$, and either $r$ is an odd integer and $y$ is an integer, or else $r$ and $y$ are half-integers. In either case, $\lceil s\rceil +\lceil s+\delta(g)\rceil=2s+\delta(g)-1,$ and we again obtain $(q^2-1)q^{d-1}$, as desired.
\end{proof}

\subsubsection{Proof of the main theorem}\label{proofofmaintheorem}
In this section, we present the proof of Theorem~\ref{thm4}.  
That is, given \(\rho=\rho(\mathcal{T}, y, r, \phi)\), we show that for each double coset representative  
\(g \in \mathcal{K}\backslash G /G_{y,s}\mathcal{T}\), the Mackey component  
\[
\mathrm{Ind}_{(G_{y,s}\mathcal{T})^{g}\cap\mathcal{K}}^{\mathcal{K}} \rho^{g}
\]
is an irreducible representation of \(\mathcal{K}\).  
By Proposition~\ref{factorisation of intersection}, we have  
\[
\mathrm{Ind}_{(G_{y,s}\mathcal{T})^{g}\cap\mathcal{K}}^{\mathcal{K}} \rho^{g}
= \mathrm{Ind}_{(\mathcal{K}\cap\mathcal{T}^{g})(\mathcal{K}\cap G_{y,s}^{g})}^{\mathcal{K}}\rho^g,
\]
so it suffices to prove the irreducibility of the right-hand side. As established in Theorem~\ref{maindoublecoset}, the double coset space 
\(\mathcal{K}\backslash G /G_{y,s}\mathcal{T}\) admits a set \(M(\mathcal{T})\) of representatives 
consisting of elements of the form \(\alpha^{t}\) with \(t \in \mathbb{Z}_{\geq 0}\) and 
\(\alpha^{t}\mathrm{w}\) with \(t \in \mathbb{Z}_{>0}\).  
Moreover, Proposition~\ref{degree of Mackey component} gives the degree of the Mackey component 
corresponding to each \(g \in M(\mathcal{T})\).

We divide the proof of this theorem into two cases: the first case, \(t=y=0\), is straightforward. 
\begin{lemma}\label{t=y=0}
	Let \(\rho=\rho(\mathcal{T},0,r,\phi)\) and let $t=0$. Then the Mackey component
	\(\mathrm{Ind}_{\mathcal{K}\cap(G_{0,s}\mathcal{T}) }^{\mathcal{K}}\rho
	=\mathrm{Ind}_{G_{0,s}\mathcal{T}}^{\mathcal{K}}\rho\)
	is an irreducible representation of \(\mathcal{K}\).
\end{lemma}

\begin{proof}
	Note that \(G_{0,s}\mathcal{T}\subseteq\mathcal{K}\), and therefore 
	\(\mathcal{K}\cap G_{0,s}\mathcal{T}=G_{0,s}\mathcal{T}\). Since $\compactinduction_{G_{0,s}\mathcal{T}}^{G}\rho$ is irreducible, $\mathrm{Ind}_{G_{0,s}\mathcal{T}}^{\mathcal{K}}\rho$ must be irreducible by transitivity of induction.
	
	%
\end{proof}

Since we have already handled the case \(t=y=0\), we shall, from this point onward, assume that \(t>0\) or \(y\neq 0\). Our strategy for this case is to construct an irreducible representation of the form \(\mathcal{S}_{d}(X,\zeta)\) (Lemma~\ref{constructionofSdXzeta}), as in Theorem~\ref{Rep of K}, having the same degree as the corresponding Mackey component, and to show that they intertwine, thereby deducing irreducibility (Proposition~\ref{intertwiningofMackeycomponentswithShalikatyperepresentations}).

To compare with the Mackey component attached to \(g\),
it is enough to find \(X\) and \(\zeta\) such that
\begin{equation*}
	\mathrm{Hom}_{\mathcal{K}}
	\!\left(
	\mathrm{Ind}_{T(X)\mathcal{J}_{d}}^{\mathcal{K}}\Psi_{X, \zeta}, \;
	\mathrm{Ind}_{(\mathcal{K}\cap\mathcal{T}^{g})(\mathcal{K}\cap G_{y,s}^{g})}^{\mathcal{K}}\rho^g
	\right)
	\neq 0.
\end{equation*}
Applying Frobenius reciprocity and the Mackey decomposition theorem, this homomorphism space decomposes as
\begin{equation*}
	\bigoplus_{h\in(\mathcal{K}\cap\mathcal{T}^{g})(\mathcal{K}\cap G_{y,s}^{g})
		\backslash \mathcal{K}/T(X)\mathcal{J}_{d}}
	\!\!\mathrm{Hom}_{(\mathcal{K}\cap\mathcal{T}^{g})(\mathcal{K}\cap G_{y,s}^{g})
		\cap (T(X)\mathcal{J}_{d})^{h}}
	\!\left(\Psi_{X, \zeta}^{h}, \rho^{g}\right).
\end{equation*}
The summand corresponding to \(h=1\) is $\mathrm{Hom}_{(\mathcal{K}\cap\mathcal{T}^{g})(\mathcal{K}\cap G_{y,s}^{g})\cap
	T(X)\mathcal{J}_{d}}
\!\left(\Psi_{X, \zeta}, \rho^{g}\right),$ and the non-vanishing of this term is precisely the condition we will verify. This will show that $\mathcal{S}_{d}(X, \zeta)$ is a subrepresentation of the Mackey component, and hence the irreducibility will follow once we confirm that their degrees coincide.

\begin{lemma}\label{pgZ-isotypic}
	Let \(\rho=\rho(\mathcal{T}, y, r, \phi)\) and let \(g\in M(\mathcal{T})\).
	Then the restriction of \(\rho^{g}\) to the subgroup
	\(\mathcal{K}\cap \mathcal{T}^{g}\) is \(\phi^{g}\)-isotypic.
\end{lemma}

\begin{proof}
	By Lemma~\ref{Z-isotypic}, we have
	\(\mathcal{K}\cap \mathcal{T}^{g}=Z(\mathcal{K}\cap \mathcal{T}_{0+}^{g})\).
	Moreover, by Lemma~\ref{isotypicpropertyofrho},
	the restriction of \(\rho\) to \(Z\mathcal{T}_{0+}\) is \(\phi\)-isotypic.
	Conjugating by \(g\), it follows that the restriction of \(\rho^{g}\) to
	\(\mathcal{K}\cap \mathcal{T}^{g}\) is \(\phi^{g}\)-isotypic, as claimed.
\end{proof} 

By Lemma~\ref{pgZ-isotypic}, the representation \(\rho^{g}\) is \(\phi^{g}\)-isotypic on \(\mathcal{K}\cap \mathcal{T}^{g}\).  
Moreover, by Lemma~\ref{Z-isotypic}, the restriction of \(\rho\) to \(G_{y,s+}\) is \(\Psi_{\Gamma}\)-isotypic, where \(\Gamma \in \mathfrak{t}_{-r}\) realizes \(\phi\) on \(\mathcal{T}_{s+}/ \mathcal{T}_{r+}\).  
Consequently, the restriction of \(\rho^{g}\) to \(\mathcal{K}\cap G_{y,s+}^{g}\) is \(\Psi_{\Gamma^{g}}\)-isotypic. 
Since the centralizer of \(\Gamma\) in \(G\) is \(\mathcal{T}\), it follows that the centralizer  $T(\Gamma^{g})$ of \(\Gamma^{g}\) in \(\mathcal{K}\) is \(\mathcal{K}\cap \mathcal{T}^{g}\).  
Recalling that \(T(X)\) denotes the centralizer of \(X\) in \(\mathcal{K}\), these observations naturally suggest taking \(X=\Gamma^{g}\) and \(\zeta=\phi^{g}\). 
These choices will be justified in the next lemma and in Proposition~\ref{intertwiningofMackeycomponentswithShalikatyperepresentations}.

\begin{lemma}\label{constructionofSdXzeta}
	Let \(\rho=\rho(\mathcal{T}, y, r, \phi)\) and let 
	\(\Gamma\in\mathfrak{t}_{-r}\) be the element realizing \(\phi\) on 
	\(\mathcal{T}_{s+}/\mathcal{T}_{r+}\). 
	Then, for every \(g\in M(\mathcal{T})\) satisfying our standing assumption 
	that \(t>0\) or \(y\neq 0\), the element \(\Gamma^{g}\) lies in 
	\(\mathfrak{g}_{0,-d}\) and has the form
	\[
	\begin{pmatrix} 
		z\sqrt{\epsilon} & u\sqrt{\epsilon} \\ 
		v\sqrt{\epsilon} & z\sqrt{\epsilon} 
	\end{pmatrix}
	\]
	with \(z,u,v\in F\) satisfying 
	\(\nu(z),\nu(v)>\nu(u)=-d\).
	Moreover, \(\phi^{g}\) and \(\Psi_{\Gamma^{g}}\) agree on the intersection 
	\(T(\Gamma^{g})\cap \mathcal{J}_{d}\).
\end{lemma}

\begin{proof} Let \(\mathcal{T}=\mathcal{T}_{\gamma_1,\gamma_2}\). Then by the genericity of \(\phi\) we know that $\Gamma=\begin{psmallmatrix}
		z\sqrt{\epsilon} & u\gamma_1\sqrt{\epsilon} \\
		u\gamma_2\sqrt{\epsilon} & z\sqrt{\epsilon}
	\end{psmallmatrix}\in\mathfrak{t}_{-r},$  with $z\in\mathfrak{p}_{F}^{\lceil -r \rceil}$ and \(\nu(u\gamma_1)=-r-y\). Since elements of $M(\mathcal{T})$ are  of the form
	\(\alpha^{t}\) with \(t\in\mathbb{Z}_{\geq 0}\) and 
	\(\alpha^{t}\mathrm{w}\) with \(t\in\mathbb{Z}_{>0}\), 
	a direct computation yields
	\[
	\Gamma^{\alpha^{t}}=
	\begin{pmatrix}
		z\sqrt{\epsilon} & u\gamma_1\varpi^{-2t}\sqrt{\epsilon} \\
		u\gamma_2\varpi^{2t}\sqrt{\epsilon} & z\sqrt{\epsilon}
	\end{pmatrix}
	\qquad\text{and}\qquad
	\Gamma^{\alpha^{t}\mathrm{w}}=
	\begin{pmatrix}
		z\sqrt{\epsilon} & u\gamma_2\varpi^{-2t}\sqrt{\epsilon} \\
		u\gamma_1\varpi^{2t}\sqrt{\epsilon} & z\sqrt{\epsilon}
	\end{pmatrix}.
	\]
	Moreover,
	\[
	\nu(u\gamma_1\varpi^{-2t}\sqrt{\epsilon})
	=-r-y-2t=-(r+\delta(\alpha^{t}))=-d,
	\]
	and
	\[
	\nu(u\gamma_2\varpi^{-2t}\sqrt{\epsilon})
	=-r-\tfrac{1}{2}+1-2t
	=-(r+2t-y)=-(r+\delta(g))=-d,
	\]
	where, in the latter case, \((\gamma_{1}, \gamma_{2}) \in \{(1,\varpi),\,(1,\epsilon^{-1}\varpi)\}\) and \(y=\tfrac{1}{2}\). Since \(d>r\) when \(t>0\) or \(y\neq 0\), we conclude that in both cases the valuations of
	\(z,\; u\gamma_2\varpi^{2t},\; u\gamma_1\varpi^{2t}\) are strictly greater than \(-d\).
	Thus \(\Psi_{\Gamma^{g}}\) is a character of \(\mathcal{J}_{d}\).
	
	It remains to show that \(\phi^{g}\) and \(\Psi_{\Gamma^{g}}\) agree on
	\(T(\Gamma^{g})\cap \mathcal{J}_{d}\). If \(t>0\) or \(y\neq 0\), then \(\delta(g)>0\) and the subgroup \(\mathcal{J}_{d}\) is given by
	\[
	\mathcal{J}_{d}=\begin{pmatrix}
		1+\mathfrak{p}_{E}^{\,s+\frac{\delta(g)}{2}} & \mathfrak{p}_{E}^{\,s+\frac{\delta(g)}{2}} \\
		\mathfrak{p}_{E}^{\,s+\frac{\delta(g)+1}{2}} & 1+\mathfrak{p}_{E}^{\,s+\frac{\delta(g)}{2}}
	\end{pmatrix}
	\cap G.
	\]
	This subgroup lies in \(G_{y,s+}\). Since \(T(\Gamma^{g}) = \mathcal{K}\cap \mathcal{T}^{g}\) and 
	\(\mathcal{J}_{d}\subseteq G_{y,s+}\), we have
	\(
	(\mathcal{K}\cap \mathcal{T}^{g})\cap \mathcal{J}_{d}
	\subseteq \mathcal{K}\cap \mathcal{T}_{s+}^{g}
	\). Since \(\phi^{g}\) is realized by \(\Gamma^{g}\) on $ \mathcal{K}\cap \mathcal{T}_{s+}^{g}$, the two characters agree on the intersection $(\mathcal{K}\cap \mathcal{T}^{g})\cap \mathcal{J}_{d}$ .
\end{proof}

By Theorem~\ref{Rep of K}, there exists an irreducible representation of $\mathcal{K}$ constructed from this data, given by
\begin{equation}\label{Shalika11'}
	\mathcal{S}_{d}(\Gamma^{g}, \phi^{g})
	:=\mathrm{Ind}_{(\mathcal{K}\cap \mathcal{T}^{g})\mathcal{J}_{d}}^{\mathcal{K}}\Psi_{\Gamma^{g}, \phi^{g}},
\end{equation}
which has depth $d$ and degree \((q^{2}-1)q^{d-1}\).

\begin{proposition}\label{intertwiningofMackeycomponentswithShalikatyperepresentations}
	Let \(\rho=\rho(\mathcal{T}, y, r, \phi)\) and let $\Gamma\in\mathfrak{t}_{-r}$ be an element realizing $\phi$ on  $\mathcal{T}_{s+}/\mathcal{T}_{r+}$. Then for all $g\in M(\mathcal{T})$ that satisfies our assumption $t>0$ or $y\neq 0$, we have
	$$\mathrm{Ind}_{\mathcal{K}\cap (G_{y,s}\mathcal{T})^{g}}\rho^{g}\cong\mathcal{S}_{d}(\Gamma^{g}, \phi^{g}).$$ 
\end{proposition}	
\begin{proof}
	For $g\in M(\mathcal{T})$, our goal is to show that the representation
	\begin{equation}\label{Shalika1'}
		\mathcal{S}_{d}(\Gamma^{g}, \phi^{g})
		:= \mathrm{Ind}_{(\mathcal{K}\cap\mathcal{T}^{g})\mathcal{J}_{d}}^{\mathcal{K}}
		\Psi_{\Gamma^{g}, \phi^{g}}
	\end{equation}
	intertwines with the Mackey component
	\begin{equation}\label{Mackey}
		\mathrm{c\text{-}Ind}_{\mathcal{K}\cap (G_{y,s}\mathcal{T})^{g}}^{\mathcal{K}} \rho^{g}
		= \mathrm{Ind}_{(\mathcal{K}\cap G_{y,s}^{g})(\mathcal{K}\cap \mathcal{T}^{g})}^{\mathcal{K}} \rho^{g}.
	\end{equation}
	
	As discussed before Lemma~\ref{constructionofSdXzeta}, it suffices to prove that
	\begin{equation}\label{homspacethatwewanttoshowisnonzero}
		\mathrm{Hom}_{(\mathcal{K}\cap\mathcal{T}^{g})(\mathcal{K}\cap G_{y,s}^{g})\cap
			(\mathcal{K}\cap \mathcal{T}^{g})\mathcal{J}_{d}}
		\!\left(\Psi_{\Gamma^{g}, \phi^{g}}, \rho^{g}\right)
	\end{equation}
	is non-zero. The subgroup on which we compare the two representations is 
	\begin{equation*}
		((\mathcal{K}\cap\mathcal{T}^{g})(\mathcal{K}\cap G_{y,s}^{g}))\cap \big((\mathcal{K}\cap\mathcal{T}^{g})\mathcal{J}_{d}\big)
		= (\mathcal{K}\cap\mathcal{T}^{g})(G_{y,s}^{g}\cap\mathcal{J}_{d}).
	\end{equation*}
	If \(G_{y,s}=G_{y,s+}\), we have \(G_{y,s}^{g}\cap\mathcal{J}_{d}\subset G_{y,s+}^{g}\). Therefore, for an element 
	\(tu\in(\mathcal{K}\cap\mathcal{T}^{g})(G_{y,s}^{g}\cap\mathcal{J}_{d})\), with 
	\(t\in \mathcal{K}\cap\mathcal{T}^{g}\) and \(u\in G_{y,s}^{g}\cap\mathcal{J}_{d}\), we obtain
	\begin{equation*}
		\rho^{g}(tu)=\phi^{g}(t)\Psi_{\Gamma}^{g}(u)=\phi^{g}(t)\Psi_{\Gamma^{g}}(u)
		= \Psi_{\phi^{g}, \Gamma^{g}}(tu).
	\end{equation*}
	Hence \eqref{Shalika1'} and \eqref{Mackey} intertwine. This establishes the isomorphism in the case \(G_{y,s}=G_{y,s+}\), which occurs either when \(\mathcal{T}\) is ramified or when \(\mathcal{T}\) is unramified and \(r\) is odd.

	We now consider the case when \(G_{y,s}\neq G_{y,s+}\), that is, when \(\mathcal{T}\) is unramified and \(r\) is even. Note that in this case the double coset representative is just \(\alpha^{t}\). The subgroup $	G_{y,s}^{\alpha^{t}}\cap \mathcal{J}_{d}$ is then given by
	\begin{equation*}
		G_{y,s}^{\alpha^{t}}\cap \mathcal{J}_{d}
		= \begin{pmatrix} 
			1+\mathfrak{p}_{E}^{\lceil \tfrac{d}{2}\rceil} & \mathfrak{p}_{E}^{\lceil \tfrac{d}{2}\rceil} \\ 
			\mathfrak{p}_{E}^{ s+2t+y} & 1+\mathfrak{p}_{E}^{\lceil \tfrac{d}{2}\rceil}
		\end{pmatrix}\cap G
		= \begin{pmatrix} 
			1+\mathfrak{p}_{E}^{s+t+\lceil \tfrac{y}{2}\rceil} & \mathfrak{p}_{E}^{s+t+\lceil \tfrac{y}{2}\rceil} \\ 
			\mathfrak{p}_{E}^{ s+2t+y} & 1+\mathfrak{p}_{E}^{ s+t+\lceil\tfrac{y}{2}\rceil}
		\end{pmatrix}\cap G.
	\end{equation*}
	This subgroup is clearly not contained in
	\begin{equation*} 
		G_{y,s+}^{\alpha^{t}}
		= \begin{pmatrix} 
			1+\mathfrak{p}_{E}^{s+1} & \mathfrak{p}_{E}^{s+1-2t-y} \\ 
			\mathfrak{p}_{E}^{s+1+2t+y} & 1+\mathfrak{p}_{E}^{s+1}
		\end{pmatrix}\cap G,
	\end{equation*}
	since the \((2,1)\)-entry of \(G_{y,s+}^{\alpha^{t}}\) lies in 
	\(\mathfrak{p}_{E}^{s+1+2t+y}\), whereas the \((2,1)\)-entry of 
	\(G_{y,s}^{\alpha^{t}}\cap \mathcal{J}_{d}\) lies in 
	\(\mathfrak{p}_{E}^{s+2t+y}\), and 
	\(\mathfrak{p}_{E}^{s+2t+y}\not\subseteq\mathfrak{p}_{E}^{s+2t+y+1}\).
	Hence the simpler argument above does not apply. In this case, the proof is a bit convoluted and proceeds as follows.

	The first point to note is that although \(\Gamma\) is uniquely determined by \(\phi\) in \(\mathfrak{t}_{-r}\) only modulo \(\mathfrak{t}_{-s}\), the restriction of \(\Psi_{\Gamma^{g}}\) to the subgroup \(G_{y,s}^{\alpha^{t}}\cap \mathcal{J}_{d}\) depends on the choice of \(\Gamma\) modulo \(\mathfrak{t}_{-s+1}\). In what follows, we show that for all distinct choices of \(\Gamma\) modulo \(\mathfrak{t}_{-s+1}\), the representation \(\Psi_{\Gamma^{\alpha^{t}}, \phi^{\alpha^{t}}}\) intertwines with \(\rho^{\alpha^{t}}\) on \((\mathcal{K}\cap\mathcal{T}^{\alpha^{t}})(G_{y,s}^{\alpha^{t}}\cap\mathcal{J}_{d})\). 
	
	Since $\rho$ has depth $r$, $\rho^{\alpha^{t}}\mid_{(G_{y,s}^{\alpha^{t}}\cap\mathcal{J}_{d})}$ factors through the group $G_{y,r+}^{\alpha^{t}}\cap\mathcal{J}_{d}$, which is given as follows. Set $A=\mathrm{max}\{r+1, s+t+\lceil \tfrac{y}{2}\rceil\}, B=\mathrm{max}\{ r+1-\delta(\alpha^{t}), s+t+\lceil \tfrac{y}{2}\rceil\},$ and $C=\lceil r+1+\delta(\alpha^{t})\rceil$. Then 
	$$G_{y,r+}^{\alpha^{t}}\cap\mathcal{J}_{d}=\left\{\begin{pmatrix} 1+\mathfrak{p}_{E}^{A}&\mathfrak{p}_{E}^{B}\\ \mathfrak{p}_{E}^{C}&1+\mathfrak{p}_{E}^{A}\end{pmatrix}\right\}\cap G.$$
	Since \(2(s+t+\lceil \tfrac{y}{2}\rceil)> r+1\) and \(2(s+t+1)\geq r+1+2t+y\), we obtain that the quotient group  $H=(G_{y,s}^{\alpha^{t}}\cap\mathcal{J}_{d})/ (G_{y,r+}^{\alpha^{t}}\cap\mathcal{J}_{d})$ is abelian and finite. As \(H\) is abelian and finite, and \(\deg(\rho^{\alpha^{t}})=q\), the representation \(\rho^{\alpha^{t}}\) decomposes as a direct sum of \(q\) characters upon restriction to \(H\). The distinct characters of \(H\) are given by \(\Psi_{Y}\), where \(Y\) is an element of the dual lattice quotient \(\widehat{H}\), which is given by
	\begin{equation*}\hat{H}:=\left\{\begin{pmatrix}x&y\sqrt{\epsilon}\\ z\sqrt{\epsilon}&-\overline{x}\end{pmatrix}\middle\vert x\in\mathfrak{p}_{E}^{-A+1}/\mathfrak{p}_{E}^{-\lceil \frac{d}{2}\rceil+1}, y\in \mathfrak{p}_{F}^{-C+1}/ \mathfrak{p}_{F}^{-s-\delta(g)+1}, z\in \mathfrak{p}_{F}^{-B+1}/\mathfrak{p}_{F}^{-\lceil \frac{d}{2} \rceil+1}\right\}.\end{equation*}
	Thus we have 
	\begin{equation}\label{rhogrestrictedtoA}
		\mathrm {Res}_{G_{y,s}^{\alpha^{t}}\cap \mathcal{J}_{d}}\rho^{\alpha^{t}}=\oplus_{i}\Psi_{Y_{i}}
	\end{equation} 
	for some $Y_{i}\in\hat{H}$. Since $\rho^{\alpha^{t}}$ restricted to $\mathcal{K}\cap \mathcal{T}^{\alpha^{t}}$ is $\phi^{\alpha^{t}}$-isotypic, we obtain that
	\begin{equation}
		\rho^{\alpha^{t}}\mid_{(\mathcal{K}\cap\mathcal{T}^{\alpha^{t}})\cap (G_{y,s}^{\alpha^{t}}\cap \mathcal{J}_{d})}\cong \phi^{\alpha^{t}}\mid_{(\mathcal{K}\cap\mathcal{T}^{\alpha^{t}})\cap (G_{y,s}^{\alpha^{t}}\cap \mathcal{J}_{d})}\mathrm{Id}\cong \oplus_{i\in I}\Psi_{Y_{i}}\mid_{(\mathcal{K}\cap\mathcal{T}^{\alpha^{t}})\cap (G_{y,s}^{\alpha^{t}}\cap \mathcal{J}_{d})}.
	\end{equation}	
	Therefore for all $i\in I$ we have $	\phi^{\alpha^{t}}=\Psi_{Y_{i}}$ on $(\mathcal{K}\cap\mathcal{T}^{\alpha^{t}})\cap (G_{y,s}^{\alpha^{t}}\cap \mathcal{J}_{d})$, and it follows that 
	\begin{equation}\label{rhorestrictedtoC}
		\mathrm{Res}_{(\mathcal{K}\cap\mathcal{T}^{\alpha^{t}})(G_{y,s}^{\alpha^{t}}\cap\mathcal{J}_{d})}\rho^{\alpha^{t}}=\oplus_{i\in I}\Psi_{Y_{i}, \phi^{\alpha^{t}}}.
	\end{equation}
	Since $\rho^{\alpha^{t}}$ is $\Psi_{\Gamma^{\alpha^{t}}}$-isotypic upon restriction to $G_{y,s+}^{\alpha^{t}}$, the characters $\Psi_{Y_{i}}$ and $\Psi_{\Gamma^{g}}$ must agree on $G_{y,s+}^{\alpha^{t}}\cap\mathcal{J}_{d}$. Next we compute the form of the elements of $\hat{H}$ that satisfy this condition. Let
	$h:=\begin{psmallmatrix}1+a&b \\ c&1+d\end{psmallmatrix}\in G_{y,s+}^{\alpha^{t}}\cap\mathcal{J}_{d}$, $\Gamma^{\alpha^{t}}=\begin{psmallmatrix}u\sqrt{\epsilon}& v\sqrt{\epsilon}\varpi^{-2t}\gamma_1\\v\sqrt{\epsilon}\varpi^{2t}\gamma_2&u\sqrt{\epsilon}\end{psmallmatrix}$ and  $Y=\begin{psmallmatrix}x&y\sqrt{\epsilon}\\ z\sqrt{\epsilon}&-\overline{x}\end{psmallmatrix}\in\hat{H}$. Then we have
	\begin{align*}
		\Psi_{Y}(h)=\psi(ax+cy\sqrt{\epsilon}+bz\sqrt{\epsilon}-d\overline{x}),\;\text{and}
	\end{align*}
	\begin{align*}
		\Psi_{\Gamma^{\alpha^{t}}}(h)=\psi(au\sqrt{\epsilon}+cv\sqrt{\epsilon}\varpi^{-2t}\gamma_1+bv\sqrt{\epsilon}\varpi^{2t}\gamma_2+du\sqrt{\epsilon}).
	\end{align*}
	For these characters to agree, we must have, for all 
	\(a \in \mathfrak{p}_{E}^{\lceil \frac{d}{2} \rceil}\),
	\(b \in \mathfrak{p}_{E}^{\lceil \frac{d}{2} \rceil}\), and 
	\(c \in \mathfrak{p}^{s+1+\delta(\alpha^{t})}\),
	\begin{align*}
		\psi\left(a(x - u\sqrt{\epsilon})\right) = 1,\quad
		\psi\left(c(y\sqrt{\epsilon} - v\sqrt{\epsilon}\varpi^{-2t}\gamma_{1})\right) = 1,\quad
		\psi\left(b(z\sqrt{\epsilon} - v\sqrt{\epsilon}\varpi^{2t}\gamma_{2})\right) = 1.
	\end{align*}
	These equalities respectively imply that
	\begin{align*}
		x &\equiv u\sqrt{\epsilon} \pmod{\mathfrak{p}_{E}^{-\lceil \frac{d}{2} \rceil+1}},\quad
		y \equiv v\varpi^{-2t}\gamma_{1} \pmod{\mathfrak{p}_{F}^{-s-\delta(\alpha^{t})}},\quad
		z \equiv v\varpi^{2t}\gamma_{2} \pmod{\mathfrak{p}_{F}^{-\lceil \frac{d}{2} \rceil+1}}.
	\end{align*}
	Thus we see that the elements of $\hat{H}$ that satisfy the required conditions have the form
	\begin{align*}
		Y(\beta):=\beta\begin{pmatrix} 0&\sqrt{\epsilon} \varpi^{-s-\delta(\alpha^{t})}\\ 0&0\end{pmatrix}
		+\Gamma^{\alpha^{t}}
	\end{align*}
	for some $\beta\in\mathcal{O}_{F}$. 
	Since the $(2,1)$-entry of $Y(\beta)$ is only modulo $\mathfrak{p}_{E}^{-\lceil\frac{d}{2} \rceil+1}$, let us pick a more convenient coset representative as follows.
	
	Set $$Y'(\beta)=\begin{pmatrix} 0&\sqrt{\epsilon}\beta \varpi^{-s-2t-y}\gamma_{1}^{-1}\gamma_{1}\\\sqrt{\epsilon}\beta \varpi^{-s-y+2t}\gamma_{1}^{-1}\gamma_{2} &0\end{pmatrix}+\Gamma^{\alpha^{t}}.$$
	Then we have $Y'(\beta)=\widetilde{\Gamma}_{\beta}^{\alpha^{t}}$ where
	$$\widetilde{\Gamma}_{\beta}=\begin{pmatrix} u\sqrt{\epsilon}& (v+\beta\varpi^{-s-y}\gamma_{1}^{-1})\gamma_1\sqrt{\epsilon}\\(v+\beta\varpi^{-s-y}\gamma_{1}^{-1})\gamma_2\sqrt{\epsilon}&u\sqrt{\epsilon}\end{pmatrix}\in \mathfrak{t}_{-r}\subset\mathfrak{g}_{y, -r}, $$	
	and $\widetilde{\Gamma}_{\beta}^{\alpha^{t}}\in\mathfrak{g}_{2t+y, -r}$.
	
	Note that  $Y(\beta)$ and $\widetilde{\Gamma}_{\beta}^{\alpha^{t}}$ 
	represent the same coset space in $\mathfrak{g}_{2t+y, -r }/\mathfrak{g}_{2t+y, -s+}$.
	
	Choose $i\in I$ such that $\Psi_{Y_i}$ occurs in the decomposition~\eqref{rhogrestrictedtoA}. By the argument above, we have \(Y_{i}=Y(\beta)\cong\widetilde{\Gamma}_{\beta}^{\alpha^{t}}\) for some 
	\(\beta\in\mathcal{O}_{F}^{\times}\). 
	Note that \(\widetilde{\Gamma}_{\beta}\in\mathfrak{t}\); therefore, by~\eqref{definitionofTofX}, 
	\(\Gamma\) and \(\widetilde{\Gamma}_{\beta}\) have the same centralizer. 
	Consequently, the centralizer of \(Y_{i}=\widetilde{\Gamma}_{\beta}^{g}\) in \(\mathcal{K}\) is 
	\(\mathcal{K}\cap \mathcal{T}^{\alpha^{t}}\). Since $(\mathcal{K}\cap \mathcal{T}^{\alpha^{t}})\cap (G_{y,s}^{\alpha^{t}}\cap\mathcal{J}_{d})
	= (\mathcal{K}\cap \mathcal{T}^{\alpha^{t}})\cap \mathcal{J}_{d},$  and, for all \(i\in I\), we have \(\Psi_{Y_{i}}=\phi^{\alpha^{t}}\) on this intersection, 
	it follows that \(\Psi_{Y_{i}, \phi^{\alpha^{t}}}\) is a well-defined character of 
	\((\mathcal{K}\cap\mathcal{T}^{\alpha^{t}})\mathcal{J}_{d}\). 
	By Theorem~\ref{Rep of K}, we obtain that
	\begin{equation}\label{rep3}
		\mathrm{Ind}_{(\mathcal{K}\cap\mathcal{T}^{\alpha^{t}})\mathcal{J}_{d}}^{\mathcal{K}}
		\Psi_{Y_i, \phi^{\alpha^{t}}}
	\end{equation}
	is an irreducible representation of \(\mathcal{K}\) of depth \(d\) and 
	degree \((q^{2}-1)q^{d-1}\). Since $Y_i$ was chosen such that $\Psi_{Y_i}$ occurs in the decomposition~\ref{rhogrestrictedtoA}, we obtain that
	\begin{align*}&\mathrm{Hom}_{\mathcal{K}}(\mathrm{Ind}_{(\mathcal{K}\cap \mathcal{T}^{\alpha^{t}})\mathcal{J}_{d}}^{\mathcal{K}}\Psi_{Y_i, \phi^{\alpha^{t}}}, \mathrm{Ind}_{(\mathcal{K}\cap\mathcal{T}^{\alpha^{t}})( \mathcal{K}\cap G_{y,s}^{\alpha^{t}})}^{\mathcal{K}}\rho^{\alpha^{t}} )\\
		\cong&\bigoplus_{h\in(\mathcal{K}\cap\mathcal{T}^{\alpha^{t}})(\mathcal{K}\cap G_{y,s}^{\alpha^{t}})\backslash \mathcal{K}/ (\mathcal{K}\cap \mathcal{T}^{\alpha^{t}})\mathcal{J}_{d} }\mathrm{Hom}_{(\mathcal{K}\cap\mathcal{T}^{\alpha^{t}})(\mathcal{K}\cap G_{y,s}^{\alpha^{t}})\cap ((\mathcal{K}\cap \mathcal{T}^{\alpha^{t}})\mathcal{J}_{d})^{h} }(\Psi_{Y_i, \phi^{\alpha^{t}}}^{h}, \rho^{\alpha^{t}})\neq 0.
	\end{align*}	
	Since~\eqref{rep3} is an irreducible representation of $\mathcal{K}$, and both~\eqref{rep3} and~\eqref{Mackey} have the same degree, we obtain that~$\eqref{Mackey}$ is an irreducible representation of $\mathcal{K}$. Note that at this point we have shown that the Mackey component~\eqref{Mackey} is irreducible. It remains to prove the isomorphism stated in the proposition, that is,  we may take $\Gamma_{\alpha^{t}}$ in place of $Y_{i}$.

	Since both representations~\eqref{rep3} and~\eqref{Mackey} are irreducible, it follows that 
	\begin{align*}&\mathrm{Hom}_{\mathcal{K}}(\mathrm{Ind}_{(\mathcal{K}\cap \mathcal{T}^{\alpha^{t}})\mathcal{J}_{d}}^{\mathcal{K}}\Psi_{Y_i, \phi^{\alpha^{t}}},  \mathrm{Ind}_{(\mathcal{K}\cap\mathcal{T}^{\alpha^{t}})( \mathcal{K}\cap G_{y,s}^{\alpha^{t}})}^{\mathcal{K}}\rho^{\alpha^{t}} )
		\cong\mathrm{Hom}_{(\mathcal{K}\cap\mathcal{T}^{\alpha^{t}})(\mathcal{K}\cap G_{y,s}^{\alpha^{t}})\cap ((\mathcal{K}\cap \mathcal{T}^{\alpha^{t}})\mathcal{J}_{d}) }(\Psi_{Y_i, \phi^{\alpha^{t}}}, \rho^{\alpha^{t}})\cong \mathbb{C}.
	\end{align*}
	As the restriction of $\rho^{\alpha^{t}}$ to $(\mathcal{K}\cap\mathcal{T}^{\alpha^{t}})(\mathcal{K}\cap G_{y,s}^{\alpha^{t}})\cap ((\mathcal{K}\cap \mathcal{T}^{\alpha^{t}})\mathcal{J}_{d})$ is given by~\eqref{rhorestrictedtoC}, we obtain
	\begin{align*}\mathrm{Hom}_{(\mathcal{K}\cap\mathcal{T}^{\alpha^{t}})(\mathcal{K}\cap G_{y,s}^{\alpha^{t}})\cap ((\mathcal{K}\cap \mathcal{T}^{\alpha^{t}})\mathcal{J}_{d}) }(\Psi_{Y_i, \phi^{\alpha^{t}}}, \oplus_{i\in I}\Psi_{Y_{i}, \phi^{\alpha^{t}}})\cong\mathbb{C}.
	\end{align*}
	Thus all the characters  $\Psi_{Y_{i}}$ occurring in the decomposition~\eqref{rhogrestrictedtoA} must be distinct. We conclude that $\mathrm{Res}_{G_{y,s}^{\alpha^{t}}\cap\mathcal{J}_{d}}\rho^{\alpha^{t}}=\oplus_{\beta\in\mathfrak{f}}\Psi_{\widetilde{\Gamma}_{\beta}^{\alpha^{t}}}$. Since $Y(0)=\Gamma^{\alpha^{t}}$, we see that $\rho^{\alpha^{t}}$ and $\Psi_{\Gamma^{\alpha^{t}}}$ intertwine on $G_{y,s}^{\alpha^{t}}\cap\mathcal{J}_{d}$ as was required.
\end{proof}

Having established both the irreducibility of these representations and their precise description, 
we may now state our main theorem, that provides an explicit and complete description of \(\mathrm{Res}_{\mathcal{K}}\pi_{\rho}\), analogous to the one obtained in Theorem~\ref{T:depthzero} for depth-zero irreducible supercuspidal representations of our unitary group \(G\). As in the depth-zero case, all irreducible constituents appearing in the  decomposition occur with distinct depth.
\begin{theorem}\label{positive-depth branching} Let $\rho=\rho(\mathcal{T}, y, r,\phi)$, and let \( \Gamma \) be a \( G \)-generic element of depth \( -r \) that represents \( \phi \). The decomposition of the restriction of the associated supercuspidal representation of \( G \) into irreducible \( \mathcal{K} \) representations is given as follows.
	\begin{align*}
		&\mathrm{Res}_{\mathcal{K}}^{G}c  \scalebox{0.9}[0.9]{-} \mathrm{Ind}_{\mathcal{T}\mathcal{K}_{s}}^{G}\rho\cong \mathrm{Ind}_{\mathcal{T}\mathcal{K}_{s}}^{\mathcal{K}}\rho\oplus\bigoplus_{t>0}\mathcal{S}_{r+2t}(\Gamma^{\alpha^{t}}, \phi^{\alpha^{t}}),\\
		&\mathrm{Res}_{\mathcal{K}}^{G}c  \scalebox{0.9}[0.9]{-} \mathrm{Ind}_{\mathcal{T}G_{1,s}}^{G}\rho\cong\bigoplus_{t\geq0}\mathcal{S}_{r+2t+1}( \Gamma^{\alpha^{t}}, \phi^{\alpha^{t}} ),\\
		&\mathrm{Res}_{\mathcal{K}}^{G}c  \scalebox{0.9}[0.9]{-} \mathrm{Ind}_{\mathcal{T}G_{\frac{1}{2},s}}^{G}\rho\cong 
		\mathcal{S}_{r+\frac{1}{2}}(\Gamma, \phi)\oplus\bigoplus_{t>0}\left(\mathcal{S}_{r+2t+\frac{1}{2}}(\Gamma^{\alpha^{t}}, \phi^{\alpha^{t}})
		\oplus \mathcal{S}_{r+2t-\frac{1}{2}}( \Gamma^{\alpha^{t}\mathrm{w}}, \phi^{\alpha^{t}\mathrm{w}})\right).
	\end{align*}
\end{theorem}

The resulting decomposition aligns with the restriction behaviour of irreducible supercuspidal representations of \(\mathrm{SL}(2,F)\). This is consistent with expectations, since \(\mathrm{SL}(2,F)\) is isomorphic to the derived subgroup of \(\mathrm{U}(1,1)\). Unlike in the depth-zero setting, where our results diverged from those for \(\mathrm{SL}(2,F)\), the positive-depth case follows the same pattern.

\section{Applications}\label{section5.3}
Let $\pi_{\rho}$ be an irreducible supercuspidal representation of $G$ of depth $r>0$, associated with the datum $(\mathcal{T}, y, r, \phi)$, where $\rho = \rho(\mathcal{T}, y, r, \phi)$.  
Denote by $\theta$ the central character of $\pi_{\rho}$.  
By Corollary~\ref{reductiontodepthzero3}, we may assume that $\theta$ has depth zero, more precisely, $\theta = \mathbbm{1}$ or the unique nontrivial quadratic character $\delta$ of $E^{1}$, at the expense of twisting our branching rules by characters of $G$.

\subsection{Connection with depth-zero representations} In this section, we study the relationship between the higher-depth components appearing in the decomposition of $\pi_{\rho}$ and those appearing in the decomposition of depth-zero supercuspidal representations with the same central character as $\pi_{\rho}$.

\begin{theorem}\label{keyidentification} 	Let $\Gamma = 
	\begin{psmallmatrix} 
		u\sqrt{\epsilon} & v\sqrt{\epsilon}\gamma_1 \\ 
		v\sqrt{\epsilon}\gamma_2 & u\sqrt{\epsilon} 
	\end{psmallmatrix}
	\in \mathfrak{t}_{-r}$ be a $G$-generic element of depth $-r$ representing $\phi$.
	Then for all $d > 2r$ and all $g \in M(\mathcal{T})$, $\mathcal{S}_{d}(\Gamma^{g}, \phi^{g}) \cong \mathcal{S}_{d}(X_{\varpi^{-d}}, \theta).$
\end{theorem}

\begin{proof}
	For $g \in M(\mathcal{T})$, we have $\Gamma^{g} \in \mathfrak{g}_{0,-d}$ where $d = r + \delta(g)$, and we assume $d > 2r$.  If $g = \alpha^{t}$, then $\Gamma^{g} = 
	\begin{psmallmatrix} 
		u\sqrt{\epsilon} & v\sqrt{\epsilon}\varpi^{-2t}\gamma_1 \\ 
		v\sqrt{\epsilon}\varpi^{2t}\gamma_2 & u\sqrt{\epsilon} 
	\end{psmallmatrix},$ and 
	\(
	\Gamma^{g} - X_{v\varpi^{-2t}\gamma_1} \in \mathfrak{g}_{0,-\lceil d/2 \rceil}.
	\)
	Hence, by Lemma~\ref{d>2rcong}, $\Psi_{\Gamma^{g}} = \Psi_{X_{v\varpi^{-2t}\gamma_1}}$ on $\mathcal{J}_{d}$.  Similarly, if $g = \alpha^{t}\mathrm{w}$ with $t>0$, then $\Gamma^{g} = 
	\begin{psmallmatrix} 
		u\sqrt{\epsilon} & v\sqrt{\epsilon}\varpi^{-2t}\gamma_2 \\ 
		v\sqrt{\epsilon}\varpi^{2t}\gamma_1 & u\sqrt{\epsilon} 
	\end{psmallmatrix}.$ Since $t > 0$, we again have 
	\(
	\Gamma^{g} - X_{v\varpi^{-2t}\gamma_2} \in \mathfrak{g}_{0,-\lceil d/2 \rceil}
	\), so 
	$\Psi_{\Gamma^{g}} = \Psi_{X_{v\varpi^{-2t}\gamma_2}}$ on $\mathcal{J}_{d}$, and $T(X)\mathcal{J}_{d}=T(X_{v\gamma_i\varpi^{-2t}})\mathcal{J}_{d}$.
	
	Since in both cases $\phi^{g}\mid_{Z}=\theta$, and the central character $\theta$ has depth zero, it coincides with $\Psi_{X_{v\gamma_i\varpi^{-2t}}}$ on $Z\mathcal{U} \cap \mathcal{J}_{d}$, as $\Psi_{X_{v\gamma_i\varpi^{-2t}}}$ is also trivial on this intersection. Extend $\theta$ trivially to $\mathcal{U}$, and let $\Psi_{X_{v\gamma_i\varpi^{-2t}}, \theta}$ denote the unique extension of $\theta$ and $\Psi_{X_{v\gamma_i\varpi^{-2t}}}$ to $T(X_{v\gamma_i\varpi^{-2t}})\mathcal{J}_{d}$. Then, by Theorem~\eqref{repfromnilpotentelements}, 
	$\mathcal{S}_{d}(X_{v\gamma_i\varpi^{-2t}}, \theta)$ 
	is an irreducible representation of $\mathcal{K}$ of depth $d$ and degree $(q^{2} - 1) q^{d-1}$.  
	
	Since $T(X)=(\mathcal{K}\cap \mathcal{T}^{g})$, by Lemma~\ref{Z-isotypic}, we have $T(X)=Z(\mathcal{K}\cap \mathcal{T}_{0+}^{g})$, and a similar computation as in proof of Lemma~~\ref{Z-isotypic} yields $T(X)^{g^{-1}}=Z(\mathcal{K}^{g^{-1}}\cap \mathcal{T}_{0+})\subset Z\mathcal{T}_{2t}$.
	As $d>2r$,  we have $T(X)^{g^{-1}}\subset Z\mathcal{T}_{2t}\subset Z\mathcal{T}_{r+}$, and therefore  $\phi^{g}\mid_{T(X)}=\phi\mid_{Z}=\theta$. Hence, 
	$\Psi_{\Gamma^{g}, \phi^{g}}=\Psi_{X_{v\gamma_i\varpi^{-2t}}, \theta}.$ 
	Since $T(X)\mathcal{J}_{d}=T(X_{v\gamma_i\varpi^{-2t}})\mathcal{J}_{d}$ and $\Psi_{\Gamma^{g}, \phi^{g}}=\Psi_{X_{v\gamma_i\varpi^{-2t}}, \theta}$, we have
	\[	\mathcal{S}_{d}(\Gamma^{g}, \phi^{g})=\mathrm{Ind}_{T(X)\mathcal{J}_{d}}^{K}\Psi_{\Gamma^{g}, \phi^{g}}=\mathrm{Ind}_{T(X)\mathcal{J}_{d}}^{K}\Psi_{X_{v\gamma_i\varpi^{-2t}}, \theta} = \mathcal{S}_{d}(X_{v\gamma_i\varpi^{-2t}}, \theta) 
	\]
	Finally, by~Theorem~\ref{repfromnilpotentelements}, we obtain $\mathcal{S}_{d}(\Gamma^{g}, \phi^{g}) 
	= \mathcal{S}_{d}(X_{v\gamma_i\varpi^{-2t}}, \theta) 
	\cong \mathcal{S}_{d}(X_{\varpi^{-d}}, \theta)$, as was required.
\end{proof}

We fix $\sigma_{0}$ and $\sigma_{1}$ to be cuspidal representations of $\mathcal{K}/\mathcal{K}_{0+}$ with central characters $\theta=\mathbbm{1}$ and $\theta=\delta$, respectively, where $\delta$ is the nontrivial quadratic character of $E^{1}$.  
Let 
\[
\tau_{00}=\compactinduction_{\mathcal{K}}^{G}\sigma_{0}, \qquad   
\tau_{01}=\compactinduction_{\mathcal{K}^{\eta}}^{G}\sigma_{0}^{\eta}, \qquad 
\tau_{10}=\compactinduction_{\mathcal{K}}^{G}\sigma_{1}, \qquad 
\tau_{11}=\compactinduction_{\mathcal{K}^{\eta}}^{G}\sigma_{1}^{\eta}
\]
be depth-zero irreducible supercuspidal representations of $G$.  Then the following corollary follows immediately from 
Theorem~\ref{keyidentification}, Corollary~\ref{reductiontodepthzero3},  Corollary~\ref{reductiontodeltafordepthzero} and~\cite[Theorem~7.1]{ET251} which answers~\cite[Question 1.2]{guy2024representationsglndnearidentity} for all irreducible smooth representations of $G$.
\begin{corollary}\label{Question1.2supercuspidalseries}
	Given any irreducible smooth representation $\pi$ of $G$, the higher-depth components occurring in the decomposition of $\pi$ upon restriction to $\mathcal{K}$ are the same as the higher-depth components occurring in the decomposition upon restriction to $\mathcal{K}$ of a subset of $\{\tau_{00}, \tau_{01}, \tau_{10}, \tau_{11}\}$, up to scaling by a character of $G$.
\end{corollary}

The following table lists the various possibilities where the fourth column indicates the corresponding depth-zero supercuspidal representation whose higher-depth components agree with those of $\pi_{\rho}$ or a principal series representations when restricted to $\mathcal{K}$.
\begingroup
\renewcommand{\arraystretch}{1.5}
\begin{table}[h]
	\centering
	\begin{tabular}{ |c|c|c|c| }
		
		\hline
		
	Representations&	$\mathcal{T}$& $r$ & Depth-zero supercuspidal \\

		\hline
		
&	$\mathcal{T}_{1,1}$ & $r$ is even & $\tau_{00}$\;or\;$\tau_{10}$\\
	&	$\mathcal{T}_{1,1}$& $r$ is odd & $\tau_{01}$\;or\;$\tau_{11}$\\
$\pi_{\rho}$ where $\rho=\rho(\mathcal{T}, y, r,\phi)$	&	$\mathcal{T}_{\varpi^{-1}, \varpi}$& $r$ is even & $\tau_{01}$\;or\;$\tau_{11}$\\
	&	$\mathcal{T}_{\varpi^{-1}, \varpi}$& $r$ is even & $\tau_{00}$\;or\;$\tau_{10}$ \\
	&	$\mathcal{T}$ is ramified & $r\in \frac{1}{2}+\mathbb{Z}_{\geq 0}$  & $\tau_{00}\oplus \tau_{01}$ or $\tau_{10}\oplus \tau_{11}$\\	
	\hline
	Principal series&	 & $r\in \mathbb{Z}_{\geq 0}$  & $\tau_{00}\oplus \tau_{01}$ or $\tau_{10}\oplus \tau_{11}$\\	
		\hline
	\end{tabular}
	\vspace{1em}
	\caption{Correspondence between higher depth components of irreducible smooth representations and depth-zero supercuspidal representations upon restriction to $\mathcal{K}$.}
	\label{table:depthcomparison}
\end{table}
\endgroup

\subsection{Restriction to $\mathcal{K}_{2r+}$}\label{section 5.3.3}
In this section, we determine the restriction of $\pi_{\rho}$ to $\mathcal{K}_{2r+}$. Throughout this section, we assume that the central character $\theta$ of $\pi_{\rho}$ is of depth zero.

Recall from ~\S\ref{nilpotentorbits} that there are three nilpotent $G$-orbits in $\mathfrak{g}$, and they are parametrized by $\left\{ X_{\delta} \;\middle\vert\; \delta \in \{0, 1, \varpi\} \right\},$
where $X_{\delta} = \begin{psmallmatrix} 0 & \delta \sqrt{\epsilon} \\ 0 & 0 \end{psmallmatrix}.$ For $\delta \in \{0, 1, \varpi\}$, we recall from~\cite[\S5]{ET251} that the nilpotent orbits were denoted by $\mathcal{N}_{\delta}$, and that each $G$-orbit $\mathcal{N}_{\delta}$ decomposes into the following $\mathcal{K}$-orbits
\[
\mathcal{N}_{0} = \mathcal{K} \cdot X_{0}, \hspace{3em}
\mathcal{N}_{1} = \bigsqcup_{m \in 2\mathbb{Z}} \mathcal{K} \cdot X_{\varpi^{-m}}, \hspace{3em}
\mathcal{N}_{\varpi} = \bigsqcup_{n \in 2\mathbb{Z}+1} \mathcal{K} \cdot X_{\varpi^{-n}}.
\]
For each depth-zero character $\theta$ of the center of $G$, we define  highly reducible representations 
\[
\begin{aligned}
	\tau_{\mathcal{N}_{1}}(\theta) &= \bigoplus_{d \in 2\mathbb{Z}_{> 0}} \mathcal{S}_{d}(X_{\varpi^{-d}}, \theta), \qquad
	\tau_{\mathcal{N}_{\varpi}}(\theta) = \bigoplus_{d \in 2\mathbb{Z}_{\geq 0} + 1 } \mathcal{S}_{d}(X_{\varpi^{-d}}, \theta).
\end{aligned}
\]
The dimension of $\mathcal{K}_{2r+}$ fixed points of the these representations are computed in~\cite[Lemma~7.3]{ET251} which we recall here for convenience.
\begin{lemma}\label{dimensionofrepresentationsassociatedtonilpotentorbitsofG}
	Let $\theta$ be a depth-zero character of the center of $G$. Then
	\[
	\begin{aligned}
		\dim\!\left(\tau_{\mathcal{N}_{1}}(\theta)^{\mathcal{K}_{2r+}}\right) &= q \left(q^{2r} - 1\right), \qquad
		\dim\!\left(\tau_{\mathcal{N}_{\varpi}}(\theta)^{\mathcal{K}_{2r+}}\right) = q^{2r} - 1.
	\end{aligned}
	\]
\end{lemma}

We now state our main theorem of this section which is a representation theoretic analog of local character expansion, confirming~\cite[Theorem~1.1]{Mon2024} for $G$.

\begin{theorem}\label{restrictiontoK2r}
	Let $\pi_{\rho}= c\!\!\;\operatorname{-Ind}_{G_{y,s}\mathcal{T}}^{G}\rho$ be an irreducible supercuspidal representation of $G$ of depth $r > 0$, associated with the datum $\rho=\rho(\mathcal{T}, y, r, \phi)$, and let $\theta$ denote its central character. Then, in the Grothendieck group of representations,
	\[
	\mathrm{Res}_{\mathcal{K}_{2r+}}\pi_{\rho}
	\;=\;
	n(\pi_{\phi}) \mathbbm{1}
	\;\oplus\;
	a_{\mathcal{N}_{1}}\mathrm{Res}_{\mathcal{K}_{2r+}} \tau_{\mathcal{N}_{1}}(\theta)
	\;\oplus\;
	a_{\mathcal{N}_{\varpi}}\mathrm{Res}_{\mathcal{K}_{2r+}} \tau_{\mathcal{N}_{\varpi}}(\theta),
	\]
	where $a_{\mathcal{N}_{1}}, a_{\mathcal{N}_{\varpi}} \in \{0,1\}$.
\end{theorem}

\begin{proof}
	By Theorem~\ref{positive-depth branching}, the restriction of $\pi_{\rho}$ decomposes as a direct sum of irreducible representations of the form $\mathcal{S}_{d}(\Gamma^{g}, \phi^{g})$ with $d \ge r$.
	Since $\mathcal{K}_{2r+}$ acts trivially on the components with $d \le 2r$, only the terms with $d>2r$ contribute nontrivially to the restriction to $\mathcal{K}_{2r+}$. If $d>2r$, then by Lemma~\ref{keyidentification} we have 
	\[
	\mathcal{S}_{d}(\Gamma^{g}, \phi^{g}) 
	= \mathcal{S}_{d}(X_{v\gamma_i\varpi^{-2t}}, \theta) 
	\cong \mathcal{S}_{d}(X_{\varpi^{-d}}, \theta),
	\]
	therefore in the Grothendieck group of representation we have the desired decomposition with
	\[
	n(\pi_{\rho})
	= \dim\!\left(\pi_{\rho}^{\mathcal{K}_{2r+}}\right)
	- a_{\mathcal{N}_{1}}\,\dim\!\left(\tau_{\mathcal{N}_{1}}(\theta)^{\mathcal{K}_{2r+}}\right)
	- a_{\mathcal{N}_{\varpi}}\,\dim\!\left(\tau_{\mathcal{N}_{\varpi}}(\theta)^{\mathcal{K}_{2r+}}\right).
	\]
	As $g$ varies, the depths $d>2r$ occur with a fixed parity determined by $\mathcal{T}$ and $r$; even $d$ assemble into $\tau_{\mathcal{N}_{1}}(\theta)$, and odd $d$ into $\tau_{\mathcal{N}_{\varpi}}(\theta)$.
	Therefore,
	\[
	a_{\mathcal{N}_{1}} =
	\begin{cases}
		1, &
		\begin{array}{l}
			\text{if $\mathcal{T}=\mathcal{T}_{1,1}$ with $r$ even,}\\[2pt]
			\text{or $\mathcal{T}=\mathcal{T}_{\varpi^{-1},\varpi}$ with $r$ odd,}\\[2pt]
			\text{or $\mathcal{T}$ is ramified,}
		\end{array}\\[6pt]
		0, &\hspace{.429em}\text{otherwise},
	\end{cases}
	\qquad
	a_{\mathcal{N}_{\varpi}} =
	\begin{cases}
		1, &
		\begin{array}{l}
			\text{if $\mathcal{T}=\mathcal{T}_{1,1}$ with $r$ odd,}\\[2pt]
			\text{or $\mathcal{T}=\mathcal{T}_{\varpi^{-1},\varpi}$ with $r$ even,}\\[2pt]
			\text{or $\mathcal{T}$ is ramified,}
		\end{array}\\[6pt]
		0, &\hspace{.429em}\text{otherwise}.
	\end{cases}
	\]
\end{proof}
We now compute $n(\pi_{\rho})$ explicitly and record the values in Table~\ref{tab:n-values}. Since we have already computed the dimensions $\dim\!\left(\tau_{\mathcal{N}_{1}}(\theta)^{\mathcal{K}_{2r+}}\right)
\;\text{and},\;
\dim\!\left(\tau_{\mathcal{N}_{\varpi}}(\theta)^{\mathcal{K}_{2r+}}\right)$ in Lemma~\ref{dimensionofrepresentationsassociatedtonilpotentorbitsofG},
and we know the values of \(a_{\mathcal{N}_{1}}\) and \(a_{\mathcal{N}_{\varpi}}\) for the various choices of the datum 
\((\mathcal{T}, y, r, \phi)\) with associated supercuspidal representation $\pi_{\rho}$ from the preceding proof, it remains to compute  $\dim\!\left(\pi_{\rho}^{\mathcal{K}_{2r+}}\right)$ for all such choices of $\pi_{\rho}$.

Let $\Gamma\in\mathfrak{t}_{-r}$ be an element realizing $\phi$ on $\mathcal{T}_{s+}/ \mathcal{T}_{r+}$. By Theorem~\ref{positive-depth branching}, the restriction of 
\(\pi_{\rho}\) decomposes as a direct sum of irreducible representations of the form 
\(\mathcal{S}_{d}(\Gamma^{g}, \phi^{g})\) where $g\in M(\mathcal{T})$.  
Moreover, the restriction of \(\pi_{\rho}\) to the subgroup \(\mathcal{K}_{2r+1}\) 
acts trivially on those components \(\mathcal{S}_{d}(\Gamma^{g}, \phi^{g})\) 
whose depths satisfy \(d \leq 2r\).  
Consequently, the dimension of \(\pi_{\rho}^{\mathcal{K}_{2r+}}\) is equal to 
the sum of the dimensions of those 
\(\mathcal{S}_{d}(\Gamma^{g}, \phi^{g})\) occurring in the decomposition 
with depth parameter \(d\) ranging from \(r\) to \(2r\).

If \(\mathcal{T} = \mathcal{T}_{1,1}\), then $$\dim\!\left(\pi_{\rho}^{\mathcal{K}_{2r+}}\right)=\dim\!\left(\mathrm{Ind}_{\mathcal{T}\mathcal{K}_{s}}^{\mathcal{K}}\rho\right) 
+ \sum_{t=1}^{\lfloor \frac{r}{2} \rfloor}(q^{2}-1)q^{(r+2t)-1}
=q^{r}\left(q^{2\lfloor \frac{r}{2} \rfloor+1} - 1\right).$$

If \(\mathcal{T} = \mathcal{T}_{\varpi^{-1}, \varpi}\), then 
\begin{align*}
	\dim\!\left(\pi_{\rho}^{\mathcal{K}_{2r+}}\right)
	= \sum_{t=0}^{\lfloor \frac{r-1}{2} \rfloor} (q^{2}-1) q^{r+2t} 
	= (q^{2}-1) q^{r} \sum_{t=0}^{\lfloor \frac{r-1}{2} \rfloor} (q^{2})^{t} 
	= q^{r}\left(q^{2\lfloor \frac{r-1}{2} \rfloor+2} - 1\right).
\end{align*}

Finally, if \(\mathcal{T}\) is ramified, 
\begin{align*}
	\dim\!\left(\pi_{\rho}^{\mathcal{K}_{2r+}}\right)
	&= (q^{2}-1) q^{r-\frac{1}{2}}
	+ \sum_{t=1}^{\lfloor \frac{r}{2}-\frac{1}{4} \rfloor} (q^{2}-1) q^{r+2t-\frac{1}{2}}
	+ \sum_{t=1}^{\lfloor \frac{r}{2}+\frac{1}{4} \rfloor} (q^{2}-1) q^{r+2t-\frac{3}{2}} 
	= (q^{2r} - q^{r-\frac{1}{2}})(q+1).
\end{align*}

Since $n(\pi_{\rho})
= \dim\!\left(\pi_{\rho}^{\mathcal{K}_{2r+}}\right)
- a_{\mathcal{N}_{1}}\,\dim\!\left(\tau_{\mathcal{N}_{1}}(\theta)^{\mathcal{K}_{2r+}}\right)
- a_{\mathcal{N}_{\varpi}}\,\dim\!\left(\tau_{\mathcal{N}_{\varpi}}(\theta)^{\mathcal{K}_{2r+}}\right)$, we obtain the values of  \(n(\pi_{\rho})\) as summarized in 
Table~\ref{tab:n-values}.
\begingroup
\renewcommand{\arraystretch}{1.5}
\begin{table}[H]
	\centering
	\begin{tabular}{ |c|c|  }
		
		\hline
		
		$n(\pi_{\rho})$& Choice of $\mathcal{T}$ and $r$ \\

		\hline
		$q-q^{r}$& $\mathcal{T}=\mathcal{T}_{1,1}$, and $r$ is even; or $\mathcal{T}=\mathcal{T}_{\varpi^{-1},\varpi}$, and $r$ is odd\\
		$1-q^{r}$& $\mathcal{T}=\mathcal{T}_{\varpi^{-1}, \varpi}$, and $r$ is odd; or $\mathcal{T}=\mathcal{T}_{\varpi^{-1},\varpi}$, and $r$ is even\\
		$(q+1)(q^{r-\frac{1}{2}}-1)$& $\mathcal{T}$ is ramified\\
		\hline
	\end{tabular}
	\vspace{1em}
	\caption{Values of $n(\pi_{\rho})$ for various choices of the datum $\rho=\rho(\mathcal{T}, y, r, \phi)$. }
	\label{tab:n-values}
\end{table}
\endgroup

\begin{remark}
	The values of $n(\pi_{\rho})$ obtained for $G$ agree with those for the group $\mathrm{SL}(2,F)$ (see~\cite[Table~2]{Mon2024}) when $\mathcal{T}$ is unramified. In the ramified case, however, the values for $G$ are precisely twice those for $\mathrm{SL}(2,F)$.
\end{remark}

\providecommand{\bysame}{\leavevmode\hbox to3em{\hrulefill}\thinspace}
\providecommand{\MR}{\relax\ifhmode\unskip\space\fi MR }
\providecommand{\MRhref}[2]{%
	\href{http://www.ams.org/mathscinet-getitem?mr=#1}{#2}
}
\providecommand{\href}[2]{#2}

\end{document}